\documentclass{article}
\title{Local well-posedness and global stability of one-dimensional shallow water equations with surface tension and constant contact angle}
\author{Jiaxu Li\footnote{The Institute of Mathematical Sciences,
The Chinese University of Hong Kong,
Shatin, N.T.,
Hong Kong, CN.
\href{mailto:jiaxvlee@gmail.com}{jiaxvlee@gmail.com}}
, Xin Liu\footnote{Department of Mathematics, Texas A\&M University, College Station, TX 77843-3368, USA. \href{mailto:xliu23@tamu.edu}{xliu23@tamu.edu}}
, Dirk Peschka\footnote{Weierstrass Institute,
Mohrenstr. 39, 10117 Berlin, DE.
\href{mailto:dirk.peschka@wias-berlin.de}{dirk.peschka@wias-berlin.de}}
}
\date{\today}

\usepackage{todonotes}

\usepackage{amsmath}
\usepackage{amssymb}
\usepackage{amsthm}
\usepackage[abbrev]{amsrefs}
\usepackage{mathtools}
\usepackage{cases}
\usepackage{hyperref}
\usepackage[scr=boondox,  
            cal=esstix]   
           {mathalpha}
\usepackage{tikz}
\usetikzlibrary{arrows.meta}
\usepackage{cleveref}
\usepackage{graphicx}
\usepackage{enumitem}

\usepackage[font=small,labelfont=bf]{caption}
\allowdisplaybreaks

\renewcommand{\vec}[1]{\mathbf{#1}}
\newcommand{\dv}{\mathrm{div}\,}
\newcommand{\dt}{\partial_t}
\newcommand{\mrm}[1]{\mathrm{#1}}
\newcommand{\mfk}[1]{\mathfrak{#1}}
\newcommand{\idx}{\,\mathrm{d}\xi}
\newcommand{\hs}{h_{\mrm{s}}}
\newcommand{\dx}{\,\mathrm{d}x}
\newcommand{\norm}[2]{\Vert #1 \Vert_{#2}}
\newcommand{\igamma}{\kappa}
\newcommand{\sgamma}{\gamma}

\theoremstyle{definition}

\theoremstyle{plain}
\newtheorem{thm}{Theorem}
\newtheorem{lem}{Lemma}
\theoremstyle{remark}
\newtheorem{rmk}{Remark}

\numberwithin{equation}{section}

\begin{document}

\maketitle
\begin{abstract}
We consider the one-dimensional shallow water problem with capillary surfaces and moving contact {lines}. An energy-based model is derived from the two-dimensional water wave equations, where we explicitly discuss the case of a stationary force balance at a moving contact line and highlight necessary changes to consider dynamic contact angles. The moving contact line becomes our free boundary at the level of shallow water equations, and the depth of the shallow water degenerates near the free boundary, which causes singularities for the derivatives and degeneracy for the viscosity. This is similar to the physical vacuum of compressible flows in the literature. The equilibrium, the global stability of the equilibrium, and the local well-posedness theory are established in this paper. 

\smallskip 

{\par\noindent\bf Keywords:} shallow water equations, thin films, surface tension, contact lines, physical vacuum

\smallskip

{\par\noindent\bf MSC:} 76N10, 
35R35, 
76B15, 
74K35 

\end{abstract}
\section{Introduction}
The shallow water problem is a system of nonlinear partial differential equations that characterizes the motion of thin fluid layers, considering gravitational, viscous, and Coriolis forces. It is commonly employed to model the behavior of surface waves in oceans, lakes, and other geophysical flows and was first derived by Saint-Venant \cite{de1871theorie}.
Here, we consider the one-dimensional shallow water problem for a film height 
$h :=  h(x,t)$ and a vertically averaged horizontal velocity $u := u(x,t)$ in a moving domain $x\in\omega=(a,b)$ with $a=a(t)$ and $b=b(t)$, which move together with the flow. About their three-dimensional variants, we call the points $a,b$ \emph{contact lines}, see \Cref{fig:sketch}. Combining all the unknowns into a state vector $q=(a,b,h,u)$, this system has a Hamiltonian 
\begin{equation}\label{eq:sw-energy}
    \mathscr{H}(q) :=  \int_{a}^{b}  \biggl( \tfrac{1}{2} h\vert u \vert^2 + U(h,\partial_x h)\biggr) \dx,\quad U(h,\xi) := \tfrac{g}{2} h^2 + \tfrac{\sgamma}{2}(\vert \xi \vert^2 +  \alpha^2).
\end{equation}
The first integrand of $\mathscr{H}$ is the kinetic energy, and the internal energy $U$ has contributions from gravity and surface energy. Here, $g,\sgamma,\alpha$ denote the constant of gravity, surface tension, and the contact slope (the tangent of the contact angle).  The height is non-negative and vanishes at the contact lines, i.e. $h(x,t)>0$ for $a(t)<x<b(t)$ and $h\bigl(a(t),t\bigr)=h\bigl(b(t),t\bigr)=0$. The film height is zero outside the domain $\omega(t)$. Alternatively, one can formulate this problem using the momentum $p:=hu$, which vanishes outside the domain $\omega(t)$.
\begin{figure}[hb!]
\centering
\begin{tikzpicture}
\def\fluidHeight#1{1-1/9*(#1-3.5)*(#1-3.5)}
\def\fluidHHeight#1{0.5*(1-1/9*(#1-3.5)*(#1-3.5))}
\def\xx{2.4}
\draw [fill=blue!20,thick] plot[domain=0.5:6.5, smooth] (\x,{\fluidHeight{\x}}) -- (6.5,0) -- (0.5,0) -- cycle;

\draw [-{>[scale=1.3]}] (0,0) -- (0,1.7);
\draw [-{>[scale=1.3]}] (0,0) -- (7,0);
\draw [dotted,white,thick] (0.5,0) -- (6.5,0);

\draw [dashed] (0.5,0) -- (2,1);
\draw [dashed] (1.5,0) arc [start angle=0, end angle=atan(2/3), radius=1];
\node at (1.25,0.25) {$\vartheta$};

\filldraw[black] (0.5,0) circle (1pt);
\filldraw[black] (6.5,0) circle (1pt);
\filldraw[black] (\xx,0) circle (1pt);

\node at (\xx,1.25) {$\epsilon h$};
\node at (0.5,-0.21) {$a$};
\node at (6.5,-0.21) {$b$};
\node at (\xx,-0.21) {$x$};
\node at (-0.3,1.65) {$z$};

\draw [-stealth,thick] (\xx,0) --  (3,0);
\draw [-stealth,thick] (\xx,{\fluidHeight{\xx}}) --  (3,{\fluidHeight{\xx}});
\draw [-stealth,thick] (\xx,{\fluidHHeight{\xx}}) -- node[above] {$u$} (3,{\fluidHHeight{\xx}});
\draw[-stealth,thick,dashed] (\xx,0) -- (\xx,{\fluidHeight{\xx}});
\end{tikzpicture}
\caption{Sketch of fluid film height $h=h(x,t)$, horizontally uniform velocity field $u=u(x,t)$, contact lines $a=a(t)$, $b=b(t)$, contact angle $0\le\vartheta\ll 1$.}
\label{fig:sketch}
\end{figure}
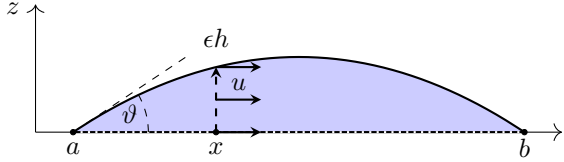

Then, for given initial data $(h,u)(t=0)=(h_0,u_0)$, $a(0)=a_0$, and $b(0)=b_0$, the free boundary problem describing the shallow water evolution $q(t)$ is 
\begin{subequations}\label{eq:sw-system}
\begin{align}
\label{eq:sw_mass}
    \partial_t h + \partial_x (h u) = 0, \\
\label{eq:sw_momentum}
    \partial_t (h{u}) + \partial_x (h {u}^2) + h\partial_x \pi - {4\mu}\partial_x (h \partial_x  u)=0,\\
\label{eq:sw_velocity}
    \dot{a}(t)=u\bigl(a(t),t\bigr)\quad\text{and}\quad
    \dot{b}(t)=u\bigl(b(t),t\bigr),
\end{align}
where $\pi=\partial_h U - {\partial_x (\partial_{\partial_x h} U})=gh-{\sgamma\partial_{xx}h}$ and $\dot{a}:=\tfrac{\rm d}{\mathrm{d}t}a$, $\dot{b}:=\tfrac{\rm d}{\mathrm{d}t}b$. Here $ \mu \in (0,\infty)  $ is a constant. These equations are satisfied by $h,u$ at time $t$ for all $x\in(a,b)$ and ensure conservation of mass \eqref{eq:sw_mass} and conservation of momentum \eqref{eq:sw_momentum}. At the contact lines $a(t) < b(t)$, we have the following kinematic and constant contact angle boundary conditions
\begin{align}
\label{eq:sw_height}
    h\bigl(a(t),t\bigr)=0 \quad&\text{and}\quad h\bigl(b(t),t\bigr)=0,\\
\label{eq:sw_slope}
    \partial_x h\bigl(a(t),t\bigr)=\alpha\quad&\text{and}\quad\partial_x h\bigl(b(t),t\bigr)=-\alpha.
\end{align}
Note that \eqref{eq:sw_height}, upon differentiation with respect to time, is equivalent to kinematic conditions of the form 
\begin{equation}
\tag{\ref*{eq:sw_height}'}
\label{eq:sw_kinematic}
\begin{split}
\partial_t h\bigl(a(t),t\bigr)+\dot{a}(t)\,\partial_x h\bigl(a(t),t\bigr)=0,\\
\partial_t h\bigl(b(t),t\bigr)+\dot{b}(t)\,\partial_x h\bigl(b(t),t\bigr)=0,
\end{split}
\end{equation}
which, together with \eqref{eq:sw_mass} and \eqref{eq:sw_slope}, imply \eqref{eq:sw_velocity}.
\end{subequations}
Later, we will show that solutions of \eqref{eq:sw-system} with Hamiltonian \eqref{eq:sw-energy} satisfy an energy-dissipation balance
\begin{equation}
\label{eq:SW-E-static}
\tfrac{\rm d}{\mathrm{d}t}\mathscr{H}\bigl(q(t)\bigr) = -\int_a^b {4 \mu}  h\vert\partial_x u\vert^2\dx\le 0,
\end{equation}
and therefore obtain consistency with the second law of thermodynamics. With \eqref{eq:sw_velocity} we ensure the conservation of mass at the contact lines. The constant contact angle in \eqref{eq:sw_slope} emerges from a stationary force balance when taking the derivative of the Hamiltonian  $\mathscr{H}$. 

\smallskip

A rigorous derivation of viscous shallow water equations without surface tension can be found in \cite{bresch_mathematical_2011}. Formal derivations of shallow water equations including surface tension based on asymptotic expansions can be found in \cites{erneux1993nonlinear,oron1997long,vaynblat2001rupture,MARCHE200749,munch2005lubrication}.
However, even without surface tension it was realized already by Lynch and Gray \cite{lynch1980finite} that the shallow water problem is a free boundary problem where \emph{wet} regions $\{x:h(x,t)>0\}$ can advance into or recede from \emph{dry} regions $\{x:h(x,t)=0\}$. This has led to the development of more complex numerical methods to treat the corresponding free boundary problem, cf. \cite{bates1999new} and references therein. The class of methods for the free boundary shallow water problem is mainly divided into Lagrangian and Eulerian methods: In the Lagrangian approach, the free boundary problem is mapped to a fixed domain and solved there \cite{akanbi1988model}. Such approaches result in very precise but highly nonlinear partial differential equations, but can be difficult to solve in higher dimensions and for topological transitions. Eulerian methods solve the shallow water problem on a fixed domain and then try to maintain good properties using specialized techniques \cite{monthe1999positivity}, e.g. non-negativity of the height or density. Length scales $L$ of most geophysical problems are way above typical capillary length $\lambda=\sqrt{\sgamma/g}$ and thus surface tension can be neglected. Alternatively, the contact angle can be treated by regularizing the surface energy with a wetting potential \cites{munch2005lubrication,peschka2010self,lallement2017shallow}, which avoids topological transitions and maintains the global positivity of solutions.

However, for microfluidic wetting and dewetting problems surface tension plays a vital role \cites{de1985wetting,bonn2009wetting} but the considered fluids are often very viscous. 
Any viscous hydrodynamic model needs to address the so-called ``no-slip paradox'' discovered by Huh \& Scriven and Dussan \& Davis \cites{huh1971hydrodynamic,dussan1974on} for example by modification of the no-slip boundary condition with an appropriate Navier-slip or free-slip boundary condition.
Corresponding formal asymptotic techniques result in \emph{thin-film models} \cites{oron1997long,bonn2009wetting} of the form
\begin{equation}\label{eq:thinfilm}
    \partial_t h - \partial_x\bigl[ m(h) \, \partial_x(-\sgamma\partial_{xx} h + \partial_h U_\mathrm{int})\bigr] = 0, \\
\end{equation}
where $U_\mathrm{int}=U_\mathrm{int}(h)$ is an intermolecular potential with a similar role as $U$ in \eqref{eq:sw-energy} and $m(h)=\tfrac13 \vert h\vert^3 + bh^2$ is a degenerate mobility with Navier-slip length $b$.
The thin-film free boundary problem with moving contact lines is well-understood mathematically \cites{bertozzi1998mathematics,giacomelli2003rigorous,bertsch2005thin,knupfer2011well,ghosh2022revisiting}, in particular regarding the (lack of) regularity near a moving contact line, e.g. cf. \cites{gnann2016regularity,giacomelli2013regularity,flitton2004surface}. Numerical algorithms with stationary and dynamic force balance at a moving contact line have been investigated in dimension $d=1$ \cites{peschka2015thin,peschka_variational_2018} and $d=2$ \cite{peschka2022model} based on energy-variational arguments. In particular, the importance of dynamic contact angles \cite{snoeijer2013moving} based on microscopic arguments and formulated in a thermodynamic framework should be emphasized \cites{shikhmurzaev1997moving,ren_boundary_2007}.

Without surface tension, the shallow water problem can be seen as an important special case of the compressible isentropic Navier-Stokes equations (viscous Saint-Venant system). Here, the height is replaced by the density $h \equiv \rho$ and one assumes a density-dependent viscosity coefficient, such that%
\begin{subequations}\label{ddns}
\begin{align}
    \partial_t \rho + \partial_x (\rho  u) = 0, \\
    \partial_t (\rho u) + \partial_x (\rho u^2) + \partial_x P
    - 2 \partial_x \bigl( \mu(\rho) \partial_x  u\bigr) =0 ,
\end{align}
\end{subequations}
with $P=\rho^2$ and $\mu (\rho)=\mu \rho$. 
For the general isentropic Navier-Stokes equations one considers
$P=\rho^\igamma$ and $\mu(\rho)=\rho^\alpha$ with the adiabatic index $\igamma>1$ and $\alpha\geq 0$.

Dry regions $\{x:h(x,t)=0\}$ in the shallow water equations correspond to \emph{vacuum} $\{x:\rho(x,t)=0\}$ in the compressible Navier-Stokes system.
There is a large amount of literature about the long-time existence and asymptotic behavior of solutions to the system \eqref{ddns} in the case $\mu(\rho)$ is constant $(\alpha=0)$. 
When the initial density is strictly away from the vacuum, see 
Kazhikhov \cite{KAZHIKHOV1977273} and Hoff 
 \cite{hoff1987global} for strong solutions, Hoff and Smoller \cite{hoff2001non} for weak solutions. When the initial density contains a vacuum,
this leads to some
singular behaviors of solutions, such as the
failure of continuous dependence of weak solutions on initial data \cite{hoff1991failure} and the finite time blow-up of smooth solutions \cites{xin1998blowup,jiu2015remarks}, and even non-existence
of classical solutions with finite energy \cite{li2019non}.

Therefore, one may consider density-dependent viscosity case $(\alpha>0)$. It is reasonable for compressible Navier-Stokes equations, see Liu-Xin-Yang \cite{liu1997vacuum}, a viscous Saint-Venant system for the shallow waters, see Gerbeau-Perthame \cite{gerbeau2000derivation}, and
some geophysical flows, see \cites{bresch2003some,bresch2003existence,bresch2006construction}. In particular, Didier-Beno{\^\i}t-Lin studied a compressible fluid model of Korteweg type in \cite{bresch2003some}:
\begin{subequations}
\begin{align}
 	\rho_t+\mathrm{div} (\rho u)=0,\\
 	(\rho u)_t+\mathrm{div}(\rho u\otimes u)=\dv(-P\mathbb{I}+2\mu\rho\mathbb{D}u)+\sgamma \rho\nabla\Delta \rho,
\end{align}  
\end{subequations} 
see also Danchin-Desjardins \cite{DANCHIN200197}, Hao-Hsiao-Li \cite{hao2009cauchy}, Germain-LeFloch \cite{germain2016finite}. 

The vacuum-free boundary problem of \eqref{ddns} has attracted a vast of attractions in recent years. In the case that the viscosity is constant,
Luo-Xin-Yang \cite{luo2000interface} studied the global regularity and behavior of the weak
solutions near the interface when the initial density connects to vacuum states in a very smooth manner. Zeng \cite{zeng2015global} showed the global existence
of smooth solutions for which the smoothness extends all the way to the boundary. In the case that the viscosity is density-dependent, the global existence of weak solutions was studied by many authors, see \cite{yang2002compressible} without external force, and
\cites{duan2011dynamics,zhang2006global,okada2004free} with external force and the references therein.
By taking the effect of external force into account, Ou-Zeng \cite{ou2015global} obtained the global well-posedness of strong solutions and the global regularity uniformly up to the vacuum boundary. When the viscosity coefficient vanishes at vacuum, Li-Wang-Xin \cite{li_well-posedness_2022} first establishes the local well-posedness of classic solutions of system \eqref{ddns} without surface tension.

\smallskip 

In this paper, we provide the ingredients to combine well-established models for moving contact lines that are valid on microscopic length scales with the shallow water problem on intermediate scales, where the capillary length is still relevant. 
%
We develop a theory to describe phenomena that combine capillarity, moving contact lines, and inertia. The major difficulty lies in the moving boundary and the degeneracy near the vacuum. We first investigate the stability of the stationary equilibrium. In particular, we analyze the linearized system \eqref{lgeq:lnr_sys} and find the key { energy functional} \eqref{lnest:014}, in which the concavity of the equilibrium plays an important role. Then we move on to investigate the nonlinear stability theory, showing that the weighted, degenerate energy functional is strong enough to control the nonlinearities globally in time, thanks to Hardy's inequality. Finally, we sketch the local well-posedness theory for general initial data.

This paper is organized as follows: In \Cref{sec:derivation_sw}, we give a brief derivation of system \eqref{eq:sw-system} from two-dimensional viscous water wave equations and summarize our main results in \Cref{sec:sw_and_main}. We recall some weighted embedding inequalities in \Cref{sec:preliminaries}. In \Cref{sec:lagrangian}, we reformulate the free
boundary problem \eqref{eq:sw-system} in the Lagrangian coordinates and state the main results of this paper. Then we present the linear stability and nonlinear stability in \Cref{sec:linear_stability,sec:nonlinear_estimates,sec:nonlinear_apriori} and therefore finish the proof of asymptotic stability of the stationary equilibrium. \Cref{sec:local_well_posedness} is devoted to the local well-posedness theory for general initial data.

\section{Shallow water equations with surface tension}\label{sec:derivation_sw}
\subsection{Shallow water approximation}

In the following, we will provide a systematic derivation of the one-dimensional shallow water equations with surface tension and moving contact lines from the two-dimensional water wave equations, which we state below. 
For the water wave model we follow the mathematical models presented and analyzed, for example, in \cites{ren_boundary_2007,guo2023stability}.

\smallskip 

Let $ 0 <  \epsilon \ll 1 $ be the asymptotically small thickness of the liquid film, $ \vec u:= (u,w)(x,z,t) $ be the two-dimensional velocity field, $ p:= p(x,z,t) $ be the pressure potential, and $ \epsilon h:= \epsilon h(x,t) $ be the height of the liquid film. Using the height, similar to \Cref{fig:sketch}, we define the time-dependent domain 
\begin{align*}
\Omega_\epsilon(t): = \lbrace (x,z): x \in \omega(t), 0<z<\epsilon h(x,t) \rbrace  
\quad\text{for}\quad \omega(t):=\bigl(a(t),b(t)\bigr).
\end{align*}
 Additionally, we define the two free boundaries 
\begin{align*}
 &\Gamma_h(t):=\{(x,\epsilon h(x,t)):a(t)<x<b(t)\}\subset\Gamma(t)\qquad\text{and}\\ &\Gamma_0(t):=\{(x,0):a(t)<x<b(t)\}\subset\Gamma(t),
\end{align*}
 i.e. the top and bottom part of $\Gamma(t)=\partial\Omega_\epsilon(t)$. The outer normal on $\Gamma_h(t)$ is 
\begin{align*}
    \mathbf{n}_\epsilon(x,z,t):= \frac{1}{(1+(\epsilon\partial_x h)^2)^{1/2}}\begin{pmatrix}-\epsilon \partial_x h (x,t) \\ 1 \end{pmatrix}.
\end{align*}

With these definitions for a shallow domain, the viscous water wave equations can be written as
\begin{subequations}\label{sys:water-waves}
\begin{align}
        \dt \vec u + \vec u \cdot\nabla \vec u + \dv\!\bigl( p\mathbb{I}_2 - \mu (\nabla \vec u + \nabla \vec u^\top)\bigr)& = - g_\epsilon \vec e_z  && \text{in} \quad \Omega_\epsilon(t), \label{eq:water-waves} \\
        \dv \vec u = \partial_x u + \partial_z w & = 0 && \text{in} \quad\Omega_\epsilon(t),
        \label{eq:ww-dv-free}
\end{align}
with kinematic conditions for the evolution of the boundaries $\Gamma_h(t),\Gamma_0(t)$ 
\begin{align}
        &\epsilon (\dt h + u\vert_{z=\epsilon h}   \partial_x h) - w\vert_{z=\epsilon h} = 0,
        \label{eq:ww-bdry}\\
        &\dot{a}(t)=u(a(t),0,t), \quad \dot{b}(t)=u(b(t),0,t).\label{eq:ww-contl}
\end{align}
The stress boundary condition on the moving boundary $\Gamma_h(t)$ is 
\begin{align}
    \left( p \mathbb I_2 - \mu (\nabla \vec u + \nabla \vec u^\top) \right)\mathbf{n}_\epsilon = - \sgamma_\epsilon \partial_x \biggl(\dfrac{\epsilon \partial_x h}{\sqrt{1+ \epsilon^2 \vert \partial_x h \vert^2}} \biggr) \mathbf{n}_\epsilon. \label{bc:moving-b-s-b}
\end{align}
On the bottom boundary $\Gamma_0(t)$ we have an impermeability boundary condition and a free slip boundary condition
\begin{align}
\label{eq:bc-impermeability}w&=0\qquad\text{on }\quad\Gamma_0(t),    \\
\label{eq:bc-freeslip}\partial_z u &= 0\qquad\text{on }\quad\Gamma_0(t).
\end{align}
Instead of free slip \eqref{eq:bc-freeslip}, also a Navier-slip condition $u - \mathscr{b}_\epsilon \partial_z u = 0$ with slip length $\mathscr{b}_\epsilon\ge 0$ in the sense of \cites{munch2005lubrication,bresch_mathematical_2011} are possible boundary conditions on $\Gamma_0(t)$. However, a no-slip boundary condition $u=0$ would be infeasible on $\Gamma_0(t)$ since this generates a logarithmic singularity in the energy dissipation.
The final missing condition for the contact angle $0\le \vartheta < \pi/2$ is 
\begin{align}
\label{eq:contact-angle}
\epsilon\,\partial_x h(a(t),t)=\tan\vartheta \quad \text{and} \quad \epsilon\,\partial_x h(b(t),t)=-\tan\vartheta.
\end{align}
\end{subequations}
Note that this system has the Hamiltonian
\begin{align}
\mathcal{H}_\epsilon
:= \int_{\Omega_\epsilon} \biggl( \tfrac{1}{2}|\mathbf{u}|^2 + g_\epsilon z \biggr) \dx\,\mathrm{d}z + \int_{\Gamma_h} \sgamma_\epsilon\left( \sqrt{1+\epsilon^2\vert\partial_x h\vert^2}-\cos\vartheta\right)\!\!\dx,
\end{align}
as a driving energy functional for the evolution with $\tfrac{\mathrm{d}}{\mathrm{d}t}\mathcal{H}_\epsilon\le 0$.
Moreover, in the shallow water regime, we employ a smallness assumption of the contact slope $\alpha$ of the form 
\begin{align}
\label{eq:small_slope}
1-\cos\vartheta = \tfrac{1}{2}\epsilon^2\alpha^2 \ll 1,
\end{align}
for some given $\alpha=\mathcal{O}(1)$. Then one can check that as $ \epsilon \rightarrow 0 $, the leading order of $ \mathcal{H}_\epsilon $ is exactly the Hamiltonian $\mathscr{H}$ of the shallow water system, i.e., \eqref{eq:sw-energy}.
%
%


\smallskip

For convenience, note that \eqref{eq:water-waves}  can be rewritten component-wisely as, 
\begin{subequations}
    \begin{align}
    \label{eq:ww-h}
    \dt u + u \partial_x u + w \partial_z u + \partial_x ( p - 2 \mu \partial_x u) -  \mu \partial_z(\partial_x w + \partial_z u) &= 0, \\
    \label{eq:ww-v}
\dt w + u \partial_x w + w \partial_z w + \partial_z ( p -  2 \mu \partial_z w) - \mu \partial_x (\partial_z u + \partial_x w ) &=- g_\epsilon ,
\end{align}
and at $ z = \epsilon h(x,t) $ \eqref{bc:moving-b-s-b} can be rewritten using the two components
\begin{align}
    (p - 2\mu \partial_x u) (-\epsilon\partial_x h) - \mu (\partial_x w + \partial_z u) &= - \sgamma_\epsilon \partial_x \left(\tfrac{\epsilon \partial_x h}{\sqrt{1+ \epsilon^2 \vert \partial_x h \vert^2}} \right) (- \epsilon \partial_x h ), \label{bc:moving-b-h} \\
    - \mu (\partial_z u + \partial_x w) (-\epsilon \partial_x h) + ( p - 2 \mu \partial_z w) &= - \sgamma_\epsilon \partial_x \left(\tfrac{\epsilon \partial_x h}{\sqrt{1+ \epsilon^2 \vert \partial_x h \vert^2}} \right).\label{bc:moving-b-v}
\end{align}
\end{subequations}
We will use an asymptotic scaling for surface tension $\sgamma_\epsilon$ and gravity $g_\epsilon$ in the shallow water approximation's final stages.
%
%
Guided by \cite{pedlosky_geophysical_1987}, we derive shallow water equations with surface tension and moving contact lines from \eqref{sys:water-waves}.\\[-0.8em]

{\par\noindent\bf Step 1: Re-scaling the vertical variables.} Owing to the vertical scale of the shallow domain, it is natural to introduce the changes of variables
\begin{equation}\label{def:change-of-variables-1}
    z: = \epsilon Z \quad \text{and} \quad w := \epsilon W.
\end{equation}
The system \eqref{sys:water-waves} can be recast, using $ (u,W,p) $ in the $(x,Z,t)$-coordinates 
in the rescaled fluid domain $ \Omega(t) := \lbrace (x,Z): x\in \omega(t), 0<Z<h(x,t) \rbrace $.
\begin{subequations}\label{sys:ww-shallow-domain}
Equations \eqref{eq:water-waves}--\eqref{eq:ww-dv-free} transform into
\begin{gather}
    \label{eq:ww-h-1}
    \dt u + u \partial_x u + W \partial_Z u + \partial_x ( p - 2 \mu \partial_x u) - \tfrac{1}{\epsilon} \mu \partial_Z (\epsilon \partial_x W + \tfrac{1}{\epsilon}\partial_Z u)  = 0, \\
    \label{eq:ww-v-1}
    \epsilon^2 (\dt W + u \partial_x W + W \partial_Z W) + \partial_Z ( p - 2 \mu \partial_Z W) - \mu \partial_x (\partial_{Z} u + \epsilon^2 \partial_x W ) =- \epsilon g_\epsilon ,\\
    \label{eq:ww-dv-free-1}
    \partial_x u + \partial_Z W  = 0,
\end{gather}
and the kinematic condition becomes
\begin{gather}
    \label{eq:ww-bdry-1}
    \dt h + u\vert_{Z=h} \partial_x h - W \vert_{Z=h}  = 0.
\end{gather}
Meanwhile, boundary conditions 
(\ref{bc:moving-b-s-b},
\ref{eq:bc-impermeability},
\ref{eq:bc-freeslip})
can be rewritten as
\begin{gather}
    \begin{gathered} 
    (p - 2\mu \partial_x u) (-\epsilon\partial_x h) - \mu (\epsilon \partial_x W + \dfrac{1}{\epsilon}\partial_Z u)\\
    = - \sgamma_\epsilon \partial_x \biggl(\dfrac{\epsilon \partial_x h}{\sqrt{1+ \epsilon^2 \vert \partial_x h \vert^2}} \biggr) (- \epsilon \partial_x h ) 
    \qquad \text{on} \quad \lbrace (x,Z= h(x,t)) \rbrace, 
    \end{gathered} \label{bc:moving-b-h-1} \\
    \begin{gathered} - \mu (\dfrac{1}{\epsilon}\partial_Z u + \epsilon \partial_x W) (-\epsilon \partial_x h) + ( p - 2\mu  \partial_Z W) 
    = - \sgamma_\epsilon \partial_x \biggl(\dfrac{\epsilon \partial_x h}{\sqrt{1+ \epsilon^2 \vert \partial_x h \vert^2}} \biggr)\\
    \qquad\qquad\qquad \text{on} \quad \lbrace (x,Z= h(x,t)) \rbrace,
    \end{gathered}  \label{bc:moving-b-v-1}\\
    W\vert_{Z=0}=0, \quad \partial_Z u\vert_{Z=0}=0. \label{bc:fixed-b-1}
\end{gather}
\end{subequations}

In the following, the barotropic component (vertical average) of an arbitrary function $ f(x,Z,t) $ is defined by
\begin{equation}\label{def:barotropic-cpn}
    \overline f(x,t) := \dfrac{1}{h} \int_0^{h(x,t)} f(x,Z,t)\,\mathrm{d}Z.
\end{equation}
This definition allows for trivial identities of the form
\begin{align}
\label{eq:identity_barotropic}
\partial_t (h\bar{f}) = (\partial_t h)f\vert_{Z=h} + h\overline{\partial_t f}\quad\text{and}\quad
\partial_x (h\bar{f}) = (\partial_x h)f\vert_{Z=h} + h\overline{\partial_x f}.
\end{align}
Furthermore, for $f=u$ we have an exact continuity equation
\begin{align}
\label{sw-001}
\partial_t h + \partial_x (h\overline{u})=0,
\end{align}
for any solution of \eqref{sys:water-waves}. This follows from the short computation
\begin{align*}
\partial_x (h\overline{u})&\overset{\eqref{def:barotropic-cpn}\,}{=}\partial_x \int_0^h u(x,Z,t)\mathrm{d}Z=(\partial_x h) u\vert_{Z=h} + \int_0^h \partial_x u(x,Z,t)\mathrm{d}Z\\
&\overset{\eqref{eq:ww-dv-free-1}}{=}(\partial_x h) u\vert_{Z=h} - \int_0^h \partial_Z W(x,Z,t)\mathrm{d}Z
\overset{(\ref{eq:bc-impermeability},\ref{eq:ww-bdry-1})}{=}  -\partial_t h.
\end{align*}
{\par\noindent\bf Step 2: Multiscale analysis.} 
The leading $ \epsilon^{-2} $-order of \eqref{eq:ww-h-1} implies 
\begin{equation}
\label{mltscl-002}
\partial_{ZZ} u=\mathcal{O}(\epsilon^2),
\end{equation} and integrating that in $Z$ using \eqref{bc:fixed-b-1} gives $\partial_Z u=\mathcal{O}(\epsilon^2)$. Integrating \eqref{eq:ww-v-1} in $Z$ and using \eqref{bc:moving-b-v-1} yields that at the leading order
\begin{equation}\label{mltscl-001}
    \begin{split}
    &\bigl(p - 2 \mu \partial_Z W - \mu \partial_x u\bigr)(x,Z,t) = -\epsilon g_\epsilon Z + C(x,t) + \mathcal O(\epsilon^2),\\
        &\quad\text{where}\quad C=\epsilon g_\epsilon h - \mu \partial_x u\vert_{Z=h}-\epsilon \sgamma_\epsilon \partial_{xx}h,
    \end{split}
\end{equation}
at $(x,Z,t)$.
Therefore, integrating \eqref{mltscl-001} in $ Z $ from $ 0 $ to $ h(x,t) $ yields
\begin{equation}\label{mltscl-003}
    \begin{gathered}
    h \overline{p} + \mu \partial_x (h \overline u) - \mu u\vert_{Z=h} \partial_x h = - \epsilon \sgamma_\epsilon h \partial_{xx}h -  \mu h \partial_x u\vert_{Z=h}\\
    + \tfrac{\epsilon}{2} g_\epsilon h^2 + \mathcal O(\epsilon^2),
    \end{gathered}
\end{equation}
where we have used \eqref{eq:ww-dv-free-1} and \eqref{eq:ww-bdry-1}.
Using \eqref{eq:ww-dv-free-1} we can rewrite \eqref{eq:ww-h-1} as
\begin{equation*}
    \dt u + \partial_x (u^2) + \partial_Z (W  u) + \partial_x ( p - 2\mu \partial_x u)
    - \epsilon^{-2}\mu \partial_{ZZ} u - \mu \partial_{Zx}W  = 0,
\end{equation*}
%
where integration in $ Z $ from $ 0 $ to $ h$ using \eqref{eq:ww-bdry-1}, \eqref{bc:fixed-b-1}, and \eqref{bc:moving-b-h-1} yields
\begin{align}
\nonumber        \dt (h\overline{u}) + \partial_x (h \overline{u^2}) + \partial_x(h \overline p - 2 \mu h \overline {\partial_x u}) 
        &= \big[(p - 2\mu \partial_x u) \partial_x h + \mu \partial_x W + \tfrac{1}{\epsilon^2} \mu  \partial_Z u\big]\big\vert_{Z=h} \\ 
\nonumber        &\phantom{=}+
        \big[u\vert_{Z=h} \underbrace{(\dt h + u \partial_x h - W)\vert_{Z=h}}_{=0 \text{ via }\eqref{eq:ww-bdry-1}}\big]\\
\label{mltscl-004} &\!\!\!\overset{\eqref{bc:moving-b-h-1}}{=} - \epsilon \sgamma_\epsilon (\partial_x h)   \partial_{xx}h + \mathcal{O}(\epsilon^2).
\end{align}
Substituting $ h \overline p $ from \eqref{mltscl-003} into \eqref{mltscl-004} yields
    \begin{align}
\nonumber \dt (h\overline{u}) &+ \partial_x (h \overline{u^2}) + \partial_x (\tfrac{\epsilon}{2} g_\epsilon h^2)
         - \epsilon \sgamma_\epsilon h \partial_{xxx} h\\
\nonumber &= \partial_x (2\mu h \overline{\partial_x u} + \mu \partial_x(h\overline u) - \mu u\vert_{Z=h}\partial_x h + \mu  h \partial_x u\vert_{Z=h} ) + \mathcal O(\epsilon^2) \\
    \label{mltscl-005}         &= \partial_x ( 3 \mu  \partial_x(h\overline u) - 3 \mu  u\vert_{Z=h}\partial_x h {+} \mu h \partial_x u\vert_{Z=h} ) + \mathcal O(\epsilon^2),
    \end{align}
where we have used the identity $h \overline{\partial_x u} = \partial_x (h\overline u) - u\vert_{Z=h} \partial_x h$.

{\par\noindent\bf Step 3: Barotropic approximation.} This step will finish the formal derivation of shallow water equations with surface tension. Thanks to \eqref{mltscl-002}, one can derive that
\begin{align*}\label{mltscl-006}
    u(x,Z,t) &= \overline u(x,t) + \int_0^Z \partial_Z u(x,Z',t)\,\mathrm{d}Z' - \dfrac{1}{h} \int_0^h \biggl(\int_0^Z \partial_Z u(x,Z',t)\,\mathrm{d}Z'\biggr)\,\mathrm{d}Z\\
    &= \overline u(x,t) + \mathcal O(\epsilon^2).
\end{align*}
From this approximation we deduce in particular that $u\vert_{Z=h}=\bar{u} + \mathcal O(\epsilon^2)$ and that $\overline{u^2}=\overline{u}^2 + \mathcal O(\epsilon^2)$ and thus \eqref{mltscl-005} yields
\begin{equation}\label{sw-002}
    \begin{gathered}
                 \dt (h\overline{u}) + \partial_x (h \overline{u}^2) + \partial_x (\tfrac{\epsilon}{2} g_\epsilon h^2)
         - \epsilon \sgamma_\epsilon h \partial_{xxx} h = 4 \mu \partial_x ( h \partial_x \overline u) + \mathcal O(\epsilon^2).
    \end{gathered}
\end{equation}

Finally, consider a scaling where $ \epsilon g_\epsilon = g = \mathcal{O}(1) $ and $ \epsilon \sgamma_\epsilon = \sgamma =\mathcal{O}(1) $. Formally passing the limit $ \epsilon \rightarrow 0^+ $ in \eqref{sw-001} and \eqref{sw-002} and renaming $u(x,t):=\bar{u}(x,t)$ leads to the shallow water equations with surface tension
\begin{subequations}\label{sys:sw}
\begin{gather}
    \dt h + \partial_x (h  u) = 0, \label{eq:sw-height}\\
    \label{eq:sw-momentum}
    \dt (h{u}) + \partial_x (h {u}^2) + \partial_x (\tfrac{g}{2} h^2)
     - \sgamma h \partial_{xxx}h  = {4 \mu}\partial_x ( h \partial_x  u).
\end{gather}
to be satisfied for $x\in\omega(t)$, and therefore we have recovered \eqref{eq:sw_mass} and \eqref{eq:sw_momentum}. By expanding \eqref{eq:contact-angle} for small slopes we obtain
\begin{align}
\partial_x h(a(t),t)=\alpha \qquad \text{and}\qquad \partial_x h(b(t),t)=-\alpha,
\end{align}
and the kinematic conditions remain as they were.
\end{subequations}
A rigorous derivation of viscous shallow water equations without surface tension can be found in \cite{bresch_mathematical_2011}. A rigorous derivation of system \eqref{sys:sw} in the case when $ h > 0 $, i.e., the non-degenerate case, follows similarly. 

In this work, our goal is to investigate the case when system \eqref{sys:sw} degenerates. In particular, we focus on the case when $ \omega = \lbrace h > 0 \rbrace = \omega(t) \subset \mathbb R $ is a domain evolving together with the flow, and $ h $ has compact support. At the boundary of the support, these equations are supplemented by the boundary conditions {(\ref{eq:sw_velocity},\ref{eq:sw_height},\ref{eq:sw_slope})}, which were untouched by the shallow water approximation.


\subsection{Conservation laws and contact angle}

We start by deriving a few conservation laws and energy balances for solutions to the shallow water problem.

\bigskip

{\par\noindent\bf Conservation of mass:} Taking the time derivative of the volume of the incompressible fluid, i.e. the integral of $h$ over $ \omega(t) $, yields 
\begin{equation}\label{cnsv:mass}
\dfrac{\mathrm{d}}{\mathrm{d}t} \int_{a(t)}^{b(t)} h \dx \overset{\eqref{eq:sw-height}}{=} - hu\vert_{a}^{b} + h(b) \dot b - h(a) \dot a \overset{\eqref{eq:sw_velocity}}{=} 0. 
\end{equation}

\bigskip

{\par\noindent\bf Balance of momentum:} Assuming constant contact angle, integrating the momentum $p=hu$ and using the divergence form of \eqref{eq:sw-momentum} yields
\begin{equation}\label{cnsv:momentum}
\dfrac{\mathrm{d}}{\mathrm{d}t} \int_{a(t)}^{b(t)} h u \dx = - \int_a^b \sgamma \partial_x h \partial_{xx}h \dx = - \dfrac{\sgamma}{2} \vert \partial_x h \vert^2 \big\vert_a^b = 0.
\end{equation}

\bigskip

{\par\noindent\bf Balance of energy:} Taking the $ L^2 $-inner product of \eqref{eq:sw-momentum} with $ u $ and integrating by part in the resultant lead to, thanks to \eqref{eq:sw-height},
\begin{equation}\label{cnsv:001}
    \dfrac{\mathrm{d}}{\mathrm{d}t} \biggl\lbrace \tfrac{1}{2}\int_{a(t)}^{b(t)}  h \vert u \vert^2 + g \vert h \vert^2  \dx \biggr\rbrace - \sgamma \int_a^b \dt h \partial_{xx} h\dx +  \int_a^b {4 \mu }h \vert \partial_x u \vert^2 \dx = 0.
\end{equation}
Moreover, by applying integration by parts further, one can calculate that
\begin{equation}\label{cnsv:002}
    \begin{split}
        - \sgamma \int_a^b \dt h \partial_{xx} h\dx &= \sgamma \int_a^b \partial_{tx} h \partial_x h \dx - \sgamma (\dt h \partial_x h)\big\vert_a^b \\
        &\!\!\!\overset{}{=} \dfrac{\sgamma}{2} \int_a^b \partial_t\vert\partial_x h\vert^2 \dx + \sgamma (u \vert \partial_x h\vert^2 )\big\vert_a^b \\
        &=  \dfrac{\mathrm{d}}{\mathrm{d}t} \int_a^b \tfrac{\sgamma}{2}\vert \partial_x h \vert^2 \dx + \tfrac{\sgamma}{2} (u \vert \partial_x h \vert^2) \big\vert_a^b,
    \end{split}
\end{equation}
where we have used \eqref{eq:sw_kinematic} and \eqref{eq:sw_velocity}.
%
At this point, we want to explore the impact of using a constant contact angle 
\eqref{eq:sw_slope} 
on the energy balance. For constant contact angle \eqref{eq:sw_slope}, using \eqref{eq:sw_velocity} we find
\begin{equation}\label{cnsv:004}
    (u\vert \partial_x h\vert^2) \big\vert_a^b = (\dot b - \dot a) \alpha^2 = \dfrac{\mathrm{d}}{\mathrm{d}t} \int_a^b  \alpha^2 \dx.
\end{equation}
Next, substituting (\ref{cnsv:002},\ref{cnsv:004}) into \eqref{cnsv:001} leads to
\begin{equation}\label{cnsv:energy}
    \dfrac{\mathrm{d}}{\mathrm{d}t} \biggl[ \dfrac{1}{2}\int_{a}^{b}  h \vert u \vert^2 + g  h^2 + \sgamma (\vert \partial_x h \vert^2 + \alpha^2)  \dx \biggr] = -\int_a^b {4\mu }h \vert \partial_x u \vert^2 \dx.
\end{equation}
This is the energy balance that we stated before in \eqref{eq:SW-E-static}. The static contact angle produces a thermodynamic consistent shallow water model in the sense $\tfrac{\rm d}{\mathrm{d}t}\mathscr{H}\le 0$ with the  Hamiltonian \eqref{eq:sw-energy}.

\bigskip

\begin{rmk}\label{rmk:bc}
    Formally, different equivalent variants of boundary and kinematic conditions at $a,b$ are possible for the water wave equation, i.e. (\ref{eq:sw_slope},\ref{eq:sw_height}) or \eqref{eq:sw_kinematic} or combinations thereof. Similar arguments as in \cite{li_well-posedness_2022}*{remark 2} show that, classical solution to \eqref{sys:sw} with moving boundary  satisfy 
    \begin{equation}\label{cnsv:005}
        \partial_x u\vert_{x=a,b} = 0.
    \end{equation}
    For $s\in\{a,b\}$, taking the time-derivative of the slope at $x=s$ gives
    \begin{align}
            \nonumber\dfrac{\mathrm{d}}{\mathrm{d}t}\partial_x h(s(t),t) \big\vert_{s = a,b} &= (\dt \partial_x h + \dot{s}\partial_x (\partial_x h))\big\vert_{s}
            \overset{\eqref{eq:sw-height}}{=} (- \partial_{x}^2 (uh) + u\partial_{xx}h ) \big\vert_{s}\\&=\big(-2\partial_x h\partial_x u\big)_{s} \overset{\eqref{cnsv:005}}{=} 0.
    \end{align}
    and therefore imposing \eqref{cnsv:005} implies that the contact angle does not change from its initial value. While the rest of the manuscript centers on a constant contact angle and free slip, we now briefly discuss a more general contact line model and the impact of a finite Navier slip length $\mathscr{b}_\epsilon$.
\end{rmk}

\subsection{Dynamic contact angle and Navier slip}
Models with dynamic contact angle use a dynamic stress balance at the contact line, a well-established concept in hydrodynamic models using variational arguments and dissipative processes.
%
A general thermodynamic consistent model for a dynamic contact angle in the spirit of \cites{ren_boundary_2007,peschka_variational_2018} but applied to the shallow water problem replaces \eqref{eq:sw_slope} by%
\begin{equation}
\label{eq:sw_dynamic}
    \nu\dot{a}(t)=\bigl[\alpha^2-\bigl(\partial_x h\bigl(a(t),t\bigr) \bigr)^2\bigr]\quad\text{and}\quad
    \nu\dot{b}(t)=-\bigl[\alpha^2-\bigl(\partial_x h\bigl(b(t),t\bigr)\bigr)^2\bigr].
\end{equation}
For example, in \cite{munch2005lubrication} it is shown that a scale $\mathscr{b}_\epsilon \epsilon^2 = \mathscr{b}=\mathcal{O}(1)$ leads to a modified momentum balance, where instead of \eqref{eq:sw_momentum} one has 
\begin{align}
\label{eq:sw_momentum_ns}
    \partial_t (h{u}) + \partial_x (h {u}^2) + h\partial_x \pi - {4\mu}\partial_x (h \partial_x  u)+\mathscr{b}^{-1}u=0,
\end{align}
while all other equations remain as they were. 
Redoing the previous computation for the energy-dissipation balance \eqref{cnsv:001} for \eqref{eq:sw_momentum_ns} gives
\begin{equation}\label{cnsv:001_ns}
    \dfrac{\mathrm{d}}{\mathrm{d}t} \biggl\lbrace \tfrac{1}{2}\int_{a(t)}^{b(t)}  h \vert u \vert^2 + g \vert h \vert^2  \dx \biggr\rbrace =  \int_a^b \sgamma\dt h \partial_{xx} h-{4 \mu }h \vert \partial_x u \vert^2 + \mathscr{b}^{-1} u^2 \dx.
\end{equation}
Similar as in \eqref{cnsv:004} but now with dynamic contact angle \eqref{eq:sw_dynamic} we get
\begin{equation}
\label{cnsv:XXX}
\begin{split}
    (u\vert \partial_x h\vert^2) \big\vert_a^b &= \dot{b} \Big\{(\partial_x h)\big\vert_{x=b}^2 - \alpha^2 + \alpha^2\Big\} - \dot{a} \Big\{(\partial_x h)\big\vert_{x=a}^2 - \alpha^2 + \alpha^2\Big\}\\
    &\!\!\!\!\overset{\eqref{eq:sw_dynamic}}{=}\nu(\dot{a}^2+\dot{b}^2) + (\dot b - \dot a) \alpha^2
    =\nu(\dot{a}^2+\dot{b}^2) + \dfrac{\mathrm{d}}{\mathrm{d}t} \int_a^b  \alpha^2 \dx.
\end{split}
\end{equation}
Substituting (\ref{cnsv:002},\ref{cnsv:XXX}) into \eqref{cnsv:001_ns} leads to the energy-dissipation balance
\begin{align}\label{cnsv:dyn_energy}
\begin{split}
    \tfrac{\mathrm{d}}{\mathrm{d}t} \mathcal{H} = -\mathscr{D}\quad\text{where}\quad 
\mathscr{D}=\int_a^b \big[{4 \mu }h \vert \partial_x u \vert^2 + \mathscr{b}^{-1}u^2\big]\dx +\nu(\dot{a}^2+\dot{b}^2).
\end{split}
\end{align}
By this computation, we have identified two new terms in the dissipation $\mathscr{D}$. Reconsidering the previous conservation of momentum \eqref{cnsv:momentum} we get 

\begin{equation*}
\dfrac{\mathrm{d}}{\mathrm{d}t} \int_{a(t)}^{b(t)} h u \dx = -\int_a^b \mathscr{b}^{-1}u\,\dx - \dfrac{\sgamma}{2} \vert \partial_x h \vert^2 \big\vert_a^b
 = -\int_a^b \mathscr{b}^{-1}u\,\dx -\frac{\sgamma\nu}{2}(\dot{a}+\dot{b}).
\end{equation*}
At least formally, in the limit $\mathscr{b}^{-1},\nu\to 0$ we recover the energy and momentum balance for the constant contact angle, whereas with finite $\nu,\mathscr{b}^{-1}$ the momentum is not conserved due to friction with the supporting interface at $z=0$. The shallow water system with dynamic angle \eqref{eq:sw_dynamic} is thermodynamic consistent in the sense that the Hamiltonian (free energy) decreases \eqref{cnsv:dyn_energy}. While it is an interesting mathematical question to consider the regularity of solutions as $x\to a,b$ in the (singular) limits of vanishing or infinite $\mu,\nu,\mathscr{b}^{-1}$, in the following we set $\nu,\mathscr{b}^{-1}=0$.

\subsection{Equilibrium}

In this subsection, we look for a stationary equilibrium $q(x,t)=q_\mrm{s}(x)$ with $q_\mrm{s}=(a_\mathrm{s},b_\mathrm{s},\hs,u_\mrm{s})$ of \eqref{sys:sw} with constant/static contact angle \eqref{eq:sw_slope}, i.e. $u_\mrm{s}=0$ and $\alpha>0$. That is
\begin{equation}\label{eqeq:001}
    \partial_x(\tfrac{g}{2} \hs^2) - \sgamma \hs \partial_{xxx} \hs=0 \qquad \text{and} \qquad \partial_x \hs(b_\mrm{s}) = - \partial_x \hs(a_\mrm{s}) = - \alpha,
\end{equation}
in $\omega_{\mrm{s}}:= (a_\mrm{s}, b_\mrm{s}) = \lbrace \hs > 0 \rbrace$.
One can explicitly solve for $ \hs $ from \eqref{eqeq:001} 
\begin{equation}
    \hs(x) = \alpha \lambda \left(\dfrac{e^{2R/\lambda} + 1 -(e^{(x+R)/\lambda} + e^{(R-x)/\lambda})}{e^{2R/\lambda} - 1}\right).
\end{equation}
%
where $\lambda=\sqrt{\sgamma/g}$ is the capillary length and $b_\mrm{s} =- a_\mrm{s} = R$ sets the center of mass to the origin. With dynamic contact angle, this solution is the unique long-time limit up to translation and the droplet radius $R$ is determined by the mass of the droplet. For static contact angle, the long-time behavior can be characterized by translation and by allowing arbitrary constant droplet speeds $u(x,t)=\bar{u}_\mrm{s}\in\mathbb{R}$, $h(x,t)=\hs(x-\bar{u}_\mrm{s}t)$, $a(t)=-R+\bar{u}_\mrm{s}t$, $b(t)=R+\bar{u}_\mrm{s}t$. Generic equilibrium solutions for various $R$ are shown in Figure~\ref{fig:droplet}. For $R\ll \lambda$ the droplet shape is parabolic $\hs(x)\approx\tfrac{\alpha R}{2\lambda}\big(1-(\tfrac{x}{R})^2\big)$, whereas for $L\gg\lambda$ it has a  pancake shape $\hs(x) \approx \alpha\lambda$ away from the contact line.
%

\begin{figure}[hb!]
\centering
\includegraphics[width=\textwidth]{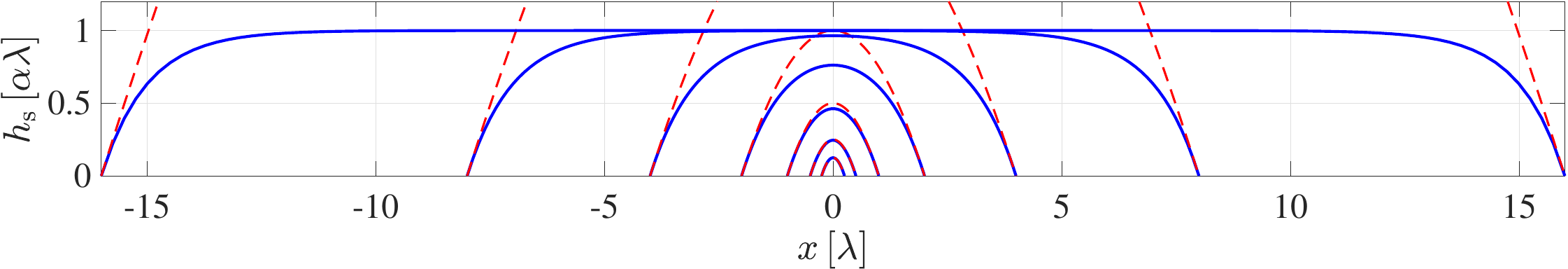}
\caption{Equilibrium solutions $\hs$ for increasing $R/\lambda\in\{\tfrac14,\tfrac12,1,2,4,8,16\}$ (blue lines) and parabolic shapes with same $R/\lambda$ (red dashed lines).}
\label{fig:droplet}
\end{figure}

\subsection{Summary of shallow water equations and main results}\label{sec:sw_and_main}

We summarize the shallow water equations with surface tension in this subsection. To simplify the presentation, we assume that
\begin{equation}\label{def:constants}
        g = 2, ~ \sgamma = 2, ~ \alpha = 1, ~ \mu  = \dfrac{1}{2}.
\end{equation}
Then the shallow water equations \eqref{sys:sw} with surface tension and constant contact angle, together with (\ref{eq:sw_velocity},\ref{eq:sw_height}, \ref{eq:sw_slope}), read
\begin{subequations}\label{sys:sw-st}
\begin{align}
    \dt h + \partial_x (h  u) & = 0, && \text{in} \quad \omega(t), \label{eq:sw-st-height}\\
    \label{eq:sw-st-momentum}
    \dt (h{u}) + \partial_x (h {u}^2) + \partial_x (h^2) - 2 h \partial_{xxx}h & = 2 \partial_x ( h \partial_x  u), && \text{in} \quad \omega(t), \\
    \label{bc:sw-st-cnt-angl}
    \partial_x h\bigl(b(t),t\bigr)= - \partial_x h\bigl(a(t),t\bigr) & = -1, && \text{for} \quad t\in (0,\infty), \\
    \partial_x u \bigl(b(t),t\bigr) = \partial_x u \bigl(a(t),t\bigr) & = 0, && \text{for} \quad t\in (0,\infty), \\ 
    \dot{a} = u \bigl(a(t),t\bigr), \quad \dot b  & = u \bigl(b(t),t\bigr), && \text{for} \quad t\in (0,\infty), 
\end{align}
\end{subequations}
where $ \omega(t) = (a(t),b(t)) $. Alternatively, \eqref{eq:sw-st-momentum} can be written in the conservative form:
\begin{equation}{\tag{\ref*{eq:sw-st-momentum}'}} \label{eq:sw-st-momentum-1}
    \dt (h{u}) + \partial_x (h {u}^2) + \partial_x \bigl\lbrack h^2 -2 h \partial_{xx} h + \vert \partial_x h \vert^2 \bigr\rbrack   = 2 \partial_x ( h \partial_x  u).
\end{equation}
The equilibrium is given by
\begin{equation}\label{eq:equilibrium}
    h_\mrm s(x) =  \dfrac{e^{2} + 1}{e^{2} - 1} -  \dfrac{e^{x+1} + e^{1-x}}{e^{2} - 1} \qquad \forall x \in (-1,1),
\end{equation}
satisfying
\begin{equation}
    \hs''' = \hs'.
\end{equation}
Here we have chosen 
\begin{equation}\label{def:constants-2} a_\mrm s = -1, b_\mrm s = 1. \end{equation}

\smallskip

The main results of this paper can be summarized, in an informal way, as follows:
\begin{thm}[Informal statement of main theorems]
    \begin{enumerate}[label=(\roman*)]
        \item Given general initial data $ (a,b,h,u)\vert_{t=0} = (a_0,b_0,h_0,u_0) $, as long as $ h_0 $ satisfies some convexity condition, there exists a unique local-in-time strong solution to system \eqref{sys:sw-st}. See Theorem \ref{thm-local}, below, in Section \ref{sec:local_well_posedness} for full description
        \item With small enough perturbation, the equilibrium state \eqref{eq:equilibrium} is asymptotically stable. In particular, there exists a global unique strong solution to system \eqref{sys:sw-st} near \eqref{eq:equilibrium}, and the solution converges to the equilibrium state as time goes to infinity. See Theorem \ref{thm-NL}, below, in Section \ref{sec:asymptotic_stability} for a full description. 
    \end{enumerate}
\end{thm}

\section{Preliminaries}\label{sec:preliminaries}
In this section, we recall some weighted embedding inequalities. The following general version of the Hardy inequality whose proof can be found in \cite{hardy_inequalities_2001} will be used often in this paper:

\begin{lem}[Hardy's inequality, $ L^p $ version]\label{lm:hardy-p} Let $ k $ and $ p > 1 $ be given real numbers, and $ g $ be a function with bounded right-hand side in the following inequalities. Then
    \begin{itemize}
        \item if $ k + \dfrac{1}{p} > 1 $, one has
        \begin{equation}\label{ineq:Hardy-3}
            \int_0^1 \bigr( s^{k-1} g(s) \bigr)^p \,\mrm{d}s\leq C_{k, p} \int \biggl( \bigl( s^k \vert g'(s)\vert \bigr)^p + \bigl( s^k \vert g(s)\vert \bigr)^p \biggr) \,\mrm{d}s;
        \end{equation}
        \item if $ k + \dfrac{1}{p} < 1 $, one has
        \begin{equation}\label{ineq:Hardy-4}
            \int_0^1 \bigr( s^{k-1} (g(s) - g(0)) \bigr)^p \,\mrm{d}s\leq C_{k, p} \int \bigl( s^k \vert g'(s)\vert \bigr)^p  \,\mrm{d}s.
        \end{equation}
    \end{itemize}
\end{lem}

Next, the following weighted Poincare inequality can be found in \cite{Feireisl2004}:
\begin{lem}[Weighted Poincare's inequality]
Let $ h_\mrm s $ be given as in \eqref{eq:equilibrium}. Suppose that $ g $ satisfies
$$
\int_{-1}^1 h_\mrm s g\,\idx = 0.
$$
Then one has
\begin{align}
\label{ineq:w-poincare}
    \int_{-1}^{1} h_\mrm s \vert g \vert^2 \,\idx \leq C \int_{-1}^{1} h_\mrm s \vert \partial_{\xi} g \vert^2 \,\idx\\
    \intertext{and}
    \label{ineq:w-poincare-1}
    \int_{-1}^{1} \vert g \vert^2 \,\idx \leq C \int_{-1}^{1} h_\mrm s \vert \partial_{\xi} g \vert^2 \,\idx.
\end{align}

\end{lem}

\section{Lagrangian formulation and main results}\label{sec:lagrangian}

\subsection{Lagrangian coordinates with reference to the equilibrium profile}

To investigate the stability profile \eqref{eq:equilibrium} of the shallow water equations with surface tension \eqref{sys:sw-st}, we introduce the Lagrangian coordinates $ (\xi, t) $ in a way that captures the equilibrium profile as coefficients. Define $ x = \eta(\xi, t) $ by
\begin{equation}\label{def:lagrangian-map}
    \int_{a(t)}^{\eta(\xi,t)} h(x,t)\dx = \int_{-1}^\xi h_\mrm s(x)\dx.
\end{equation}
In particular, we assume
\begin{equation}
    \int_{a(t)}^{b(t)} h(x,t) \dx = \int_{-1}^1 h_\mrm s(x)\dx,
\end{equation}
which, thanks to the conservation of mass \eqref{cnsv:mass}, is a restriction on the initial data.
Then taking the space and time derivatives ($\partial_{\xi} $ and $ \partial_t$) in the $ (\xi,t) $-coordinates leads to, for $ \xi \in I= (-1,1) $,
\begin{gather}
    h(\eta(\xi, t), t) \partial_\xi \eta(\xi,t) = h_\mrm s(\xi), \label{def:h-lagrangian} \\
    \begin{gathered}
    h(\eta(\xi,t),t)\dt \eta(\xi,t) = - \int_{a(t)}^{\eta(\xi,t)} \dt h(x,t)\dx = \int_{a(t)}^{\eta(\xi,t)} \partial_x (h u)(x,t)\dx \\
    = h(\eta(\xi, t), t) u(\eta(\xi, t), t),
    \end{gathered} \label{def:flow-map-lagrangian}
\end{gather}
where \eqref{eq:sw-st-height} is used. Notably, \eqref{def:flow-map-lagrangian} is equivalent to the Lagrangian flow map \cites{Jang2009a,Coutand2011a}, and the benefit of \eqref{def:lagrangian-map} is that it captures the equilibrium \cites{LuoXinZeng2016,LuoXinZeng2015}. Therefore, system \eqref{sys:sw-st} can be recast in the $ (\xi, t) $-coordinates as
\begin{subequations}\label{sys:sw-st-lagrangian}
    \begin{gather}
        h(\xi,t) = \dfrac{h_\mrm s(\xi)}{\partial_\xi \eta(\xi,t)}, \quad \dt \eta(\xi,t) = u(\xi,t),\label{lgeq:cntny}\\
        \begin{gathered}
        h_\mrm s\dt u + \partial_\xi \biggl( \bigl(\dfrac{h_\mrm s}{\partial_\xi \eta}\bigr)^2 \biggr) - 2 h_\mrm s \dfrac{\partial_\xi}{\partial_\xi \eta} \biggl(\dfrac{\partial_\xi}{\partial_\xi \eta}\biggl( \dfrac{\partial_\xi}{\partial_\xi \eta}\biggl(\dfrac{h_\mrm s}{\partial_\xi \eta} \biggr)\biggr)\biggr)\\ = 2 \partial_\xi \biggl( \dfrac{h_\mrm s}{\partial_\xi \eta} \dfrac{\partial_\xi u}{\partial_\xi \eta} \biggr),\end{gathered} \label{lgeq:mmtm}\\
        \partial_\xi \eta\vert_{\xi = -1,1} = 1,\quad \partial_\xi u\vert_{\xi = -1,1} = 0, \label{bc:lagrangian}
    \end{gather}
\end{subequations}
where we have used the same notations for the variables $ h, u $ in the Lagrangian coordinates as in the original Euclidean coordinates. Notice that \eqref{bc:lagrangian} is equivalent to \eqref{bc:sw-st-cnt-angl}, which can be seen from Remark \ref{rmk:bc}.

\bigskip

We consider the solution $ \eta $ to system \eqref{sys:sw-st-lagrangian} near the equilibrium $ \eta_\mrm s(\xi) = \xi $. Thanks to \eqref{lgeq:cntny} and \eqref{bc:lagrangian}, we consider  $ \eta(\xi, t) $ as follows:
\begin{equation}\label{def:perturbation}
    \eta(\xi,t) = \xi + \theta \qquad \text{with} \quad \partial_\xi \theta\vert_{\xi = -1,1} = 0.
\end{equation}
Notice that $ u =\dt \eta = \dt \theta $. Then in terms of the perturbation variable $ \theta $, system \eqref{sys:sw-st-lagrangian} can be written as
\begin{equation}\tag{\ref*{sys:sw-st-lagrangian}'}\label{lgeq:prtbtn}
    \begin{gathered}
        h_\mrm s \partial_{tt} \theta + \biggl( \bigl( \dfrac{h_\mrm s}{1+\theta_\xi} \bigr)^2 \biggr)_\xi - 2 \dfrac{h_\mrm s}{1+\theta_\xi} \biggl(\dfrac{1}{1+\theta_\xi} \biggl(\dfrac{1}{1+\theta_\xi}\biggl( \dfrac{h_\mrm s}{1+\theta_\xi} \biggr)_\xi \biggr)_\xi \biggr)_\xi \\
        = 2 \biggl( \dfrac{h_\mrm s}{(1+\theta_\xi)^2} \theta_{\xi t} \biggr)_\xi, \qquad \text{with} \quad \theta_\xi\vert_{\xi = -1,1} = 0.
    \end{gathered}
\end{equation}
or equivalently, in the conservative form
\begin{equation}\label{lgeq:prtbtn-csv}
    \begin{gathered}
        h_\mrm s \partial_{tt} \theta + \biggl\lbrace \biggl( \dfrac{\hs}{1+\theta_\xi} \biggr)^2 - 2 \dfrac{\hs}{1+\theta_\xi} \dfrac{1}{1 + \theta_\xi} \biggl\lbrack \dfrac{1}{1+\theta_\xi} \biggl( \dfrac{\hs}{1+\theta_\xi} \biggr)_\xi \biggr\rbrack_\xi \\
        + \biggl\lbrack \dfrac{1}{1+\theta_\xi} \biggl( \dfrac{\hs}{1+\theta_\xi} \biggr)_\xi\biggr\rbrack^2 \biggr\rbrace_\xi  
        = 2 \biggl( \dfrac{h_\mrm s}{(1+\theta_\xi)^2} \theta_{\xi t} \biggr)_\xi, \qquad \text{with} \quad \theta_\xi\vert_{\xi = -1,1} = 0.
    \end{gathered}
\end{equation}

\subsection{Main result: Asymptotic stability}\label{sec:asymptotic_stability}
We now state the main stability result in this paper. For simplicity, we denote $I=(-1,1)$ in the rest of the paper. The asymptotic stability of the stationary state of system \eqref{sys:sw-st-lagrangian} can be stated in the Lagrangian coordinates as follows:
\begin{thm}[Asymptotic Stability]\label{thm-NL}
    Let $h_\mrm s$ be the equilibrium defined in \eqref{eq:equilibrium}. There exists a constant $0<\varepsilon \ll 1$  such that if the total initial energy $ \mathfrak E_0 $ defined in \eqref{def:initial_energy}, below, satisfies
    \begin{equation}
         \mathfrak E_0<\varepsilon,
    \end{equation}
    then the system \eqref{sys:sw-st-lagrangian} admits a unique global strong solution $(\eta, u)$ in $I\times [0,\infty)$ with
    \begin{equation}\left\{
    \begin{aligned}
        &\eta, \eta_\xi-1\in C^1([0,\infty);L^2(I)),~~\eta_{\xi\xi},h_\mrm s\partial_\xi^3\eta\in C([0,\infty);L^2(I)),\\
        &h_s^{3/2}\partial_\xi^4\eta\in L^\infty((0,\infty);L^2(I)),~~
        u,u_\xi,h_\mrm s^{1/2}u_t\in C([0,\infty);L^2(I)),\\
        &u_{\xi t},u_{\xi\xi},h_\mrm s^{1/2}u_{tt},h_\mrm s\partial_\xi^3 u\in L^\infty([0,\infty);L^2(I)),
        \end{aligned}  \right.             
    \end{equation}
    Moreover, we have the following asymptotic stability:
    \begin{equation}
         \mathcal E_\mrm{NL}(t)\triangleq \mathcal E_\mrm{NL,1}(t)+\mathcal E_\mrm{NL,2}(t)<Ce^{-C_1t},
    \end{equation}
    for any $t \in [0,\infty)$ and some constants $C, C_1 \in (0,\infty)$, where $\mathcal E_\mrm{NL,1}$ and $\mathcal E_\mrm{NL,2}$ are defined in \eqref{def:energy_nonlinear_1} and \eqref{def:energy_nonlinear_2}, below, respectively.
\end{thm}

\bigskip 
\begin{rmk}
    For general initial height $h_0>0$ in $[a, b]$, we can define a diffeomorphism $\eta_0$ from $[-1,1]$ to $[a,b]$ such that 
    \begin{equation}
        \int_a^{\eta_0(\xi)}h_0(x)\,dx=\int_{-1}^\xi h_s(x)\,dx,
    \end{equation}
    which is agree with \eqref{def:lagrangian-map}.
\end{rmk}

\section{Linear stability analysis}\label{sec:linear_stability}
Assuming $ \theta = \mathcal O (\varepsilon) $ with small $ \varepsilon $ representing the size of perturbation, from \eqref{lgeq:prtbtn}, one can write down the following linear equation:
\begin{equation}\label{lgeq:lnr_sys}
    \begin{gathered}
    h_\mrm s \partial_{tt} \theta_\mrm l - 2\bigl(h_\mrm s^2 \theta_{\mrm l, \xi} \bigr)_\xi
    + 2 h_\mrm s \theta_{\mrm l,\xi} \partial_{\xi\xi\xi} h_\mrm s 
    + 2 h_\mrm s \partial_\xi \bigl(\theta_{\mrm l, \xi} \partial_{\xi\xi}  h_\mrm s \bigr) 
    + 2 h_\mrm s \partial_{\xi\xi}  \bigl( \theta_{\mrm l, \xi} \partial_\xi  h_\mrm s \bigr)  \\
    + 2 h_\mrm s \partial_{\xi\xi\xi} \bigl( h_\mrm s \theta_{\mrm l, \xi} \bigr)
    = 2 \partial_\xi \bigl( h_\mrm s \theta_{\mrm l, \xi t} \bigr), \qquad \text{with} \quad \theta_{\mrm l,\xi} \vert_{\xi = -1,1} = 0,
    \end{gathered}
\end{equation}
where we have substituted \eqref{eqeq:001} with \eqref{def:constants} and \eqref{def:constants-2}, i.e., 
\begin{equation}\label{eq:equilibrium-1}
\begin{gathered}
    \partial_\xi h_\mrm s = \partial_{\xi\xi\xi} h_\mrm s \quad \text{with} \\ 
    \partial_{\xi} h_\mrm s(1) = - \partial_{\xi} h_\mrm s(-1) = -1 \quad \text{and} \quad h_\mrm s(1) = h_\mrm s(-1) = 0,
    \end{gathered}
\end{equation}
which can be solved as \eqref{eq:equilibrium}, and satisfies
\begin{equation}\label{eq:equilibrium-2}
    \partial_{\xi\xi} h_\mrm s = h_\mrm s - \dfrac{e^2 + 1}{e^2 -1} = - \dfrac{e^{\xi+1} + e^{1-\xi}}{e^2-1} < 0.
\end{equation}

Moreover, one can rewrite the `surface tension' terms as
\begin{equation}
    \begin{gathered}
    2 h_\mrm s \theta_{\mrm l,\xi} \partial_{\xi\xi\xi} h_\mrm s 
    + 2 h_\mrm s \partial_\xi \bigl(\theta_{\mrm l, \xi} \partial_{\xi\xi}  h_\mrm s \bigr) 
    + 2 h_\mrm s \partial_{\xi\xi}  \bigl( \theta_{\mrm l, \xi} \partial_\xi  h_\mrm s \bigr) 
    + 2 h_\mrm s \partial_{\xi\xi\xi} \bigl( h_\mrm s \theta_{\mrm l, \xi} \bigr)\\
     = 2 \partial_\xi\bigl( 4 h_\mrm s \partial_{\xi\xi} h_\mrm s \theta_{\mrm l, \xi} - 2 \vert \partial_\xi h_\mrm s \vert^2 \theta_{\mrm l, \xi} + \partial_\xi ( h_\mrm s^2 \theta_{\mrm l, \xi\xi}) \bigr).
    \end{gathered}
\end{equation}
Therefore, \eqref{lgeq:lnr_sys} can be written as
\begin{equation*}\tag{\ref{lgeq:lnr_sys}'}\label{lgeq:lnr_sys-1}
    \begin{gathered}
    h_\mrm s \partial_{tt} \theta_\mrm l
    - 2 \partial_\xi\lbrack ( \hs^2 + 2 \vert \partial_\xi h_\mrm s \vert^2 - 4 h_\mrm s \partial_{\xi\xi} h_\mrm s) \theta_{\mrm l, \xi} \rbrack + 2 \partial_{\xi\xi}(h_\mrm s^2 \theta_{\mrm l, \xi\xi})\\
    = 2 \partial_\xi \bigl( h_\mrm s \theta_{\mrm l, \xi t} \bigr), \qquad \text{with} \quad \theta_{\mrm l,\xi} \vert_{\xi = -1,1} = 0.
    \end{gathered}
\end{equation*}

\subsection{$L^2$ stability}

Taking the $ L^2 $-inner product of \eqref{lgeq:lnr_sys-1} with $ \theta_{\mrm l,t} $ yields
\begin{equation}\label{lnest:003}
    \begin{gathered}
        \dfrac{\mathrm{d}}{\mathrm{d}t} \biggl\lbrace \dfrac{1}{2} \int h_\mrm s \vert \theta_{\mrm l, t} \vert^2 \,\idx + \int ( 2 \vert \partial_\xi h_\mrm s\vert^2 - 4 h_\mrm s \partial_{\xi\xi}h_\mrm s + \hs^2) \vert \theta_{\mrm l, \xi} \vert^2 \,\idx \\
        +  \int \vert h_\mrm s \vert^2 \vert \theta_{\mrm l, \xi\xi} \vert^2 \,\idx \biggr\rbrace + 2 \int h_\mrm s \vert \theta_{\mrm l, \xi t} \vert^2 \,\idx = 0.
    \end{gathered}
\end{equation}
Moreover, by taking the time derivatives in \eqref{lgeq:lnr_sys-1} and performing the same arguments as above, one can derive the same estimates as in \eqref{lnest:003} with $ \theta_\mrm l $ replaced by $ \partial_t^k \theta_\mrm l $, for any $ k \in \mathbb N $; that is
\begin{equation}\label{lnest:004}
     \begin{gathered}
        \dfrac{\mathrm{d}}{\mathrm{d}t} \biggl\lbrace \dfrac{1}{2} \int h_\mrm s \vert \dt^{k+1} \theta_{\mrm l} \vert^2 \,\idx +  \int ( 2 \vert \partial_\xi h_\mrm s\vert^2 - 4 h_\mrm s \partial_{\xi\xi}h_\mrm s + \hs^2 ) \vert \dt^k \theta_{\mrm l, \xi} \vert^2 \,\idx \\
        +  \int \vert h_\mrm s \vert^2 \vert \dt^k \theta_{\mrm l, \xi\xi} \vert^2 \,\idx \biggr\rbrace + 2 \int h_\mrm s \vert \dt^{k+1} \theta_{\mrm l, \xi} \vert^2 \,\idx = 0.
    \end{gathered}   
\end{equation}

On the other hand, from \eqref{lgeq:lnr_sys-1}, one can derive
\begin{equation}\label{lnest:005}
    \dfrac{\mathrm{d}}{\mathrm{d}t} \int h_\mrm s \theta_{\mrm l, t} \,\idx = \int h_\mrm s \theta_{\mrm l, tt} \,\idx = 0.
\end{equation}
Without loss of generality, we assume that
\begin{equation}\label{lnest:006}
    \int h_\mrm s \theta_{\mrm l,t} \,\idx = 0 \qquad \text{and} \qquad \int h_\mrm s \theta_{\mrm l} \,\idx = 0.
\end{equation}

In addition, integrating \eqref{lgeq:lnr_sys-1} from $ \xi = -1 $ to $ \xi $ yields
\begin{equation}\label{lnest:007}
    \begin{gathered}
    2 h_\mrm s \partial_t \theta_{\mrm l, \xi}- 2 \partial_\xi (h_\mrm s^2 \theta_{\mrm l, \xi\xi}) + 2( 2 \vert \partial_\xi h_\mrm s\vert^2  -  4 h_\mrm s \partial_{\xi\xi} h_\mrm s +  h_\mrm s^2 ) \theta_{\mrm l, \xi} \\ = \int_{-1}^\xi (h_\mrm s \partial_{tt} \theta_\mrm l)(\sigma) \,d\sigma.
    \end{gathered}
\end{equation}
Taking the $ L^2 $-inner product of \eqref{lnest:007} with $ \theta_{\mrm l, \xi} $ yields
\begin{equation}\label{lnest:008}
    \begin{gathered}
        \dfrac{\mathrm{d}}{\mathrm{d}t} \int h_\mrm s \vert \theta_{\mrm l, \xi} \vert^2 \,\idx + 2 \int h_\mrm s^2 \vert \theta_{\mrm l, \xi\xi} \vert^2 \,\idx + 2 \int ( 2 \vert \partial_\xi h_\mrm s \vert^2 - 4 h_\mrm s \partial_{\xi\xi} h_\mrm s + h_\mrm s^2 \bigr) \vert \theta_{\mrm l, \xi} \vert^2 \,\idx \\
        = - \int h_\mrm s \partial_{tt} \theta_\mrm l \cdot \theta_{\mrm l} \,\idx = - \dfrac{\mathrm{d}}{\mathrm{d}t} \int h_\mrm s \partial_t \theta_{\mrm l}\cdot \theta_\mrm l \,\idx + \int h_\mrm s \vert \theta_{\mrm l, t} \vert^2 \,\idx.
    \end{gathered}
\end{equation}
Notice that, thanks to \eqref{lnest:006}, applying \eqref{ineq:w-poincare} imples that
\begin{equation}\label{lnest:008-1}
    \int h_\mrm s \vert \theta_{\mrm l, t} \vert^2 \,\idx \lesssim \int h_\mrm s \vert \partial_\xi \theta_{\mrm l, t} \vert^2 \,\idx \quad \text{and} \quad \int h_\mrm s \vert \theta_{\mrm l} \vert^2 \,\idx \lesssim \int h_\mrm s \vert \partial_\xi \theta_{\mrm l} \vert^2 \,\idx. 
\end{equation}
Therefore, there exists a constant $ \mfk c_1 \in (0,\infty) $ such that, after adding \eqref{lnest:003} with $ \mfk c_1 \times \eqref{lnest:008} $, 
\begin{equation}\label{lnest:009}
    \begin{gathered}
        \dfrac{\mathrm{d}}{\mathrm{d}t} \mathcal E_0 + \mathcal D_0 = 0,
    \end{gathered}
\end{equation}
where
\begin{align}
    &\label{lnest:energy-0} \begin{aligned}
        \mathcal E_{0} : = &  \dfrac{1}{2} \int h_\mrm s \vert \theta_{\mrm l, t} \vert^2 \,\idx + \int ( 2 \vert \partial_\xi h_\mrm s\vert^2 - 4 h_\mrm s \partial_{\xi\xi}h_\mrm s + h_\mrm s^2) \vert \theta_{\mrm l, \xi} \vert^2 \,\idx \\
        & \qquad +  \int \vert h_\mrm s \vert^2 \vert \theta_{\mrm l, \xi\xi} \vert^2 \,\idx + \mfk c_1 \int h_\mrm s \vert \theta_{\mrm l, \xi} \vert^2 \,\idx + \mfk c_1 \int h_\mrm s\dt \theta_\mrm l \cdot \theta_\mrm l \,\idx, 
    \end{aligned}
        \\
    &\label{lnest:dissipation-0}\begin{aligned}    
        \mathcal D_{0} := & 2 \int h_\mrm s \vert \theta_{\mrm l, \xi t} \vert^2 \,\idx - \mfk c_1 \int h_\mrm s \vert \theta_{\mrm l, t} \vert^2 \,\idx + 2 \mfk c_1 \int h_\mrm s^2 \vert \theta_{\mrm l, \xi\xi} \vert^2 \,\idx \\
        & \qquad + 2 \mfk c_1 \int ( 2 \vert \partial_\xi h_\mrm s \vert^2 - 4 h_\mrm s \partial_{\xi\xi} h_\mrm s + h_\mrm s^2 \bigr) \vert \theta_{\mrm l, \xi} \vert^2 \,\idx,
    \end{aligned}
\end{align}
satisfying, thanks to \eqref{lnest:008-1},
\begin{gather}\label{lnest:010}
    \begin{gathered}
    \mathcal E_0 \geq \dfrac{1}{2} \int h_\mrm s \vert \theta_{\mrm l, t} \vert^2 \,\idx +  \int ( 2 \vert \partial_\xi h_\mrm s\vert^2 - 4 h_\mrm s \partial_{\xi\xi}h_\mrm s + h_\mrm s^2) \vert \theta_{\mrm l, \xi} \vert^2 \,\idx \\
    \qquad +  \int \vert h_\mrm s \vert^2 \vert \theta_{\mrm l, \xi\xi} \vert^2 \,\idx + \mfk c_1 \int h_\mrm s \vert \theta_{\mrm l, \xi} \vert^2 \,\idx \\ 
    - \mfk c_1 \biggl( \int h_\mrm s \vert \theta_{\mrm l,t} \vert^2 \,\idx \biggr)^{1/2} \cdot \biggl( \int h_\mrm s \vert \theta_{\mrm l} \vert^2 \,\idx \biggr)^{1/2}\\
    \geq ( \dfrac{1}{2} - \dfrac{\mfk c_1}{4 \sigma} ) \int h_\mrm s \vert \theta_{\mrm l, t} \vert^2 \,\idx +  \int ( 2 \vert \partial_\xi h_\mrm s\vert^2 - 4 h_\mrm s \partial_{\xi\xi}h_\mrm s + h_\mrm s^2) \vert \theta_{\mrm l, \xi} \vert^2 \,\idx \\
    \qquad +  \int \vert h_\mrm s \vert^2 \vert \theta_{\mrm l, \xi\xi} \vert^2 \,\idx + \mfk c_1 \biggl( \int h_\mrm s \vert \theta_{\mrm l, \xi} \vert^2 \,\idx - \sigma \int h_\mrm s \vert \theta_\mrm l \vert^2 \,\idx \biggr) \\
    \gtrsim \norm{h_\mrm s^{1/2} \theta_{\mrm l, t}}{L^2}^2  + \norm{\theta_{\mrm l, \xi}}{L^2}^2 + \norm{h_\mrm s \theta_{\mrm l, \xi\xi}}{L^2}^2 + \norm{h_\mrm s^{1/2} \theta_{\mrm l, \xi}}{L^2}^2\end{gathered} \\
    \intertext{and}
    \label{lnest:011}
    \begin{gathered} 
        \mathcal D_0 \gtrsim \norm{h_\mrm s^{1/2} \theta_{\mrm l, \xi t}}{L^2}^2 + \norm{h_\mrm s \theta_{\mrm l, \xi\xi}}{L^2}^2 + \norm{\theta_{\mrm l, \xi}}{L^2}^2 \gtrsim \mathcal E_0. 
    \end{gathered}
    \end{gather}
Therefore, \eqref{lnest:009} yields the exponential decay in time, i.e., 
\begin{gather}
    \label{lnest:012}
    \norm{\theta_{\mrm l}(t)}{L^\infty}^2 \lesssim \norm{h_\mrm s^{1/2} \theta_{\mrm l, t}(t)}{L^2}^2  + \norm{\theta_{\mrm l, \xi}(t)}{L^2}^2 + \norm{h_\mrm s \theta_{\mrm l, \xi\xi}(t)}{L^2}^2 \lesssim e^{- \lambda_0 t}\\
    \intertext{and}
    \label{lnest:013}
    \int e^{\lambda_0 s} \bigl\lbrace \norm{h_\mrm s^{1/2} \theta_{\mrm l, \xi t}(s)}{L^2}^2 + \norm{h_\mrm s \theta_{\mrm l, \xi\xi}(s)}{L^2}^2 + \norm{\theta_{\mrm l, \xi}(s)}{L^2}^2 \bigr\rbrace \,\mrm{d}s\lesssim 1. 
\end{gather}
for some $ \lambda_0 \in (0,\infty) $.

\subsection{Linear Elliptic estimates}\label{sec:linear-elliptic}

To estimate the nonlinearities in the original equation \eqref{lgeq:prtbtn}, it is important to capture the higher-order estimates. Repeating the same arguments as in \eqref{lnest:004}, one can easily derive that the same estimates of the types \eqref{lnest:012} and \eqref{lnest:013} for $ \theta_{\mrm l} $ replaced by $ \partial_t^k \theta_{\mrm l} $, for any $ k \in \mathbb N $, hold, i.e., 
\begin{gather}
    \label{lnest:014}
    \norm{h_\mrm s^{1/2} \partial_t^{k+1}\theta_{\mrm l}(t)}{L^2}^2  + \norm{\partial_t^k \theta_{\mrm l, \xi}(t)}{L^2}^2 + \norm{h_\mrm s \partial_t^k \theta_{\mrm l, \xi\xi}(t)}{L^2}^2 \leq C_k e^{- \lambda_k t} \\
    \intertext{and}
    \int e^{\lambda_k s} \bigl\lbrace \norm{h_\mrm s^{1/2} \partial_t^{k+1} \theta_{\mrm l, \xi}(s)}{L^2}^2 + \norm{h_\mrm s \partial_t^k \theta_{\mrm l, \xi\xi}(s)}{L^2}^2 + \norm{\partial_t^k \theta_{\mrm l, \xi}(s)}{L^2}^2 \bigr\rbrace \,\mrm{d}s\leq C_k
    \label{lnest:015}
\end{gather}
for some $ \lambda_k, C_k \in (0,\infty) $. Without loss of generality, we assume that $ 0 < \lambda_0 \leq \lambda_1 \leq \lambda_2 \leq \cdots $. 

\bigskip 

Next, we aim to show the estimates by shifting the time derivatives to space derivatives, i.e., the Elliptic estimates of \eqref{lgeq:lnr_sys-1}.

Let
\begin{equation}\label{lnest:016}
    m_\mrm s:= 2 \vert \partial_\xi h_\mrm s \vert^2 - 4 h_\mrm s \partial_{\xi\xi}h_\mrm s + h_\mrm s^2  > 0.
\end{equation}
Then \eqref{lnest:007} can be written as
\begin{equation}\label{lnest:017}
    h_\mrm s \partial_t \theta_{\mrm l, \xi} - \partial_\xi (h_\mrm s^2 \theta_{\mrm l, \xi\xi}) + m_\mrm s \theta_{\mrm l, \xi} = \dfrac{1}{2} \int_{-1}^\xi (h_\mrm s \partial_{tt} \theta_\mrm l) (\sigma) \,d\sigma.
\end{equation}

\subsubsection{Estimate of $ \theta_{\mrm l, \xi\xi\xi} $}

Let $ \varepsilon \in (0,1) $ be an arbitrarily small constant. Unless stated explicitly, are independent of $ \varepsilon $.

\bigskip

After rearranging the equation,  
\eqref{lnest:017} can be written as
\begin{equation}\label{lnest:051}
    \begin{gathered}
    - \partial_{\xi}(\hs^2 \theta_{\mrm l, \xi\xi}) + 2 \vert \hs' \vert^2 \theta_{\mrm l, \xi} = (4 \hs \hs'' - \hs^2)\theta_{\mrm l, \xi} - \hs \dt \theta_{\mrm l, \xi} \\
    +\dfrac{1}{2} \int_{-1}^\xi (h_\mrm s \partial_{tt} \theta_\mrm l) (\sigma) \,d\sigma.
    \end{gathered}
\end{equation}
Thanks to the fact that
$$
\int_{-1}^1 \hs \partial_{tt}\theta_{\mrm l} \idx = 0,
$$
one can write
\begin{equation}\label{lnest:017-1}
    \begin{gathered}
    \vert \int_{-1}^\xi (h_\mrm s \partial_{tt} \theta_\mrm l) (\sigma) \,d\sigma \vert  
    =
    \begin{cases}
        \vert \int_{-1}^\xi (h_\mrm s \partial_{tt} \theta_\mrm l) (\sigma) \,d\sigma \vert  & \xi \leq 0,\\
        \vert - \int_{\xi}^1 (h_\mrm s \partial_{tt} \theta_\mrm l) (\sigma) \,d\sigma \vert   & \xi > 0
    \end{cases}
    \\
    \lesssim  \hs \norm{\hs^{1/2} \partial_{tt}\theta_{\mrm l}}{L^2},
    \end{gathered}
\end{equation}
where the last inequality follows by applying H\"older's inequality in the two intervals $ (-1,0) $ and $ (0,1) $, and using the fact that $ \hs(\xi) \simeq (1+\xi)(1-\xi) $.

\bigskip

After dividing \eqref{lnest:051} with $ \hs^{1/2} $, taking the square of both sides and integrating the resultant, i.e., $ \norm{\dfrac{\eqref{lnest:051}}{\hs^{1/2}}}{L^2}^2 $, yield, thanks to \eqref{lnest:017-1},
\begin{equation}\label{lnest:071}
    \begin{gathered}
        \norm{\dfrac{- \partial_{\xi}(\hs^2 \theta_{\mrm l, \xi\xi}) + 2 \vert \hs' \vert^2 \theta_{\mrm l, \xi}}{\hs^{1/2}}}{L^2}^2 \lesssim \norm{\theta_{\mrm l, \xi}}{L^2}^2 + \norm{\dt \theta_{\mrm l, \xi}}{L^2}^2 \\
        + \norm{\hs^{1/2} \partial_{tt} \theta_{\mrm l}}{L^2}^2.
    \end{gathered}
\end{equation}
The left hand side of \eqref{lnest:071} can be calculated as below:
\begin{equation}\label{lnest:072}
    \begin{gathered}
        \norm{\dfrac{- \partial_{\xi}(\hs^2 \theta_{\mrm l, \xi\xi}) + 2 \vert \hs' \vert^2 \theta_{\mrm l, \xi}}{\hs^{1/2}}}{L^2}^2 = 4 \norm{\dfrac{\vert \hs' \vert^2 \theta_{\mrm l, \xi}}{\hs^{1/2}}}{L^2}^2 + \norm{\dfrac{\partial_{\xi}(\hs^2 \theta_{\mrm l, \xi\xi})}{\hs^{1/2}}}{L^2}^2\\
        - 4 \int \dfrac{\partial_{\xi}(\hs^2 \theta_{\mrm l, \xi\xi}) \cdot \vert \hs' \vert^2 \theta_{\mrm l, \xi}}{\hs} \idx = 4 \norm{\dfrac{\vert \hs' \vert^2 \theta_{\mrm l, \xi}}{\hs^{1/2}}}{L^2}^2 + \norm{\hs^{3/2} \theta_{\mrm l, \xi\xi\xi}}{L^2}^2 \\
        + 4 \norm{\hs^{1/2} \hs' \theta_{\mrm l, \xi\xi}}{L^2}^2 + 4 \int \hs^2 \hs' \theta_{\mrm l, \xi\xi} \theta_{\mrm l,\xi\xi\xi} \idx + 4 \int \hs^2 \theta_{\mrm l, \xi\xi} \cdot \partial_{\xi} \biggl( \dfrac{\vert\hs'\vert^2 \theta_{\mrm l, \xi}}{\hs} \biggr) \idx\\
        =  4 \norm{\dfrac{\vert \hs' \vert^2 \theta_{\mrm l, \xi}}{\hs^{1/2}}}{L^2}^2 + \norm{\hs^{3/2} \theta_{\mrm l, \xi\xi\xi}}{L^2}^2 + 4 \norm{\hs^{1/2} \hs' \theta_{\mrm l, \xi\xi}}{L^2}^2  \\
        - 2 \int \hs^2 \hs'' \vert \theta_{\mrm l,\xi\xi}\vert^2 \idx + 4 \int \bigr(2 \hs \hs' \hs'' - (\hs')^3\bigr) \theta_{\mrm l, \xi} \theta_{\mrm l, \xi\xi} \idx \\
        = 4 \norm{\dfrac{\vert \hs' \vert^2 \theta_{\mrm l, \xi}}{\hs^{1/2}}}{L^2}^2 + \norm{\hs^{3/2} \theta_{\mrm l, \xi\xi\xi}}{L^2}^2 + 4 \norm{\hs^{1/2} \hs' \theta_{\mrm l, \xi\xi}}{L^2}^2  \\
        - 2 \int \hs^2 \hs'' \vert \theta_{\mrm l,\xi\xi}\vert^2 \idx - 2 \int \bigr(2 \hs \hs' \hs'' - (\hs')^3\bigr)' \vert \theta_{\mrm l, \xi} \vert^2 \idx.
    \end{gathered}
\end{equation}
Therefore, \eqref{lnest:071} and \eqref{lnest:072} yield
\begin{equation}\label{lnest:3rd-d-1}
    \begin{gathered}
    \norm{\hs^{3/2} \theta_{\mrm l, \xi\xi\xi}}{L^2}^2 + \norm{\hs^{1/2} \theta_{\mrm l, \xi\xi}}{L^2}^2 + \norm{\dfrac{\theta_{\mrm l, \xi}}{\hs^{1/2}}}{L^2}^2 \lesssim \norm{\theta_{\mrm l, \xi}}{L^2}^2 + \norm{\dt \theta_{\mrm l, \xi}}{L^2}^2 \\
    + \norm{\hs^{1/2} \partial_{tt} \theta_{\mrm l}}{L^2}^2. 
    \end{gathered}
\end{equation}

\bigskip 

On the other hand, after dividing \eqref{lnest:051} with $ \hs + \varepsilon $, taking the square of both sides and integrating the resultant, i.e., $ \norm{\dfrac{\eqref{lnest:051}}{\hs+\varepsilon}}{L^2}^2 $, yield, thanks to \eqref{lnest:017-1},
\begin{equation}\label{lnest:052}
    \begin{gathered}
        \norm{\dfrac{- \partial_{\xi}(\hs^2 \theta_{\mrm l, \xi\xi}) + 2 \vert \hs' \vert^2 \theta_{\mrm l, \xi}}{\hs + \varepsilon}}{L^2}^2 \lesssim \norm{\theta_{\mrm l, \xi}}{L^2}^2 + \norm{\dt \theta_{\mrm l, \xi}}{L^2}^2 \\
        + \norm{\hs^{1/2} \partial_{tt} \theta_{\mrm l}}{L^2}^2.
    \end{gathered}
\end{equation}
Meanwhile, the left hand side of \eqref{lnest:052} can be calculated as below:
\begin{equation}\label{lnest:053}
    \begin{gathered}
        \norm{\dfrac{- \partial_{\xi}(\hs^2 \theta_{\mrm l, \xi\xi}) + 2 \vert \hs' \vert^2 \theta_{\mrm l, \xi}}{\hs + \varepsilon}}{L^2}^2 = 4 \norm{\dfrac{\vert\hs'\vert^2 \theta_{\mrm l, \xi}}{\hs + \varepsilon}}{L^2}^2 + \norm{\dfrac{\partial_\xi(\hs^2 \theta_{\mrm l, \xi\xi})}{\hs+\varepsilon}}{L^2}^2 \\
        \underbrace{- 4 \int \dfrac{\partial_\xi(\hs^2 \theta_{\mrm l, \xi\xi}) \cdot \vert \hs' \vert^2 \theta_{\mrm l, \xi}}{(\hs + \varepsilon)^2} \idx}_{I_{01}}.
    \end{gathered}
\end{equation}
To calculate $ I_{01} $, applying integration by parts yields
\begin{equation}
    \begin{gathered}
        I_{01} = 4 \int \hs^2 \theta_{\mrm l, \xi\xi} \cdot \partial_\xi \biggl( \dfrac{\vert \hs' \vert^2 \theta_{\mrm l, \xi}}{(\hs+\varepsilon)^2} \biggr) \idx = 4 \int \dfrac{\hs^2 \vert \hs'\vert^2 \vert\theta_{\mrm l,\xi\xi}\vert^2}{(\hs+\varepsilon)^2}\idx \\
        + 4 \int  \biggl(\dfrac{2 \hs^2 \hs' \hs''}{(\hs +\varepsilon)^2} - \dfrac{2 \hs^2 ( \hs' )^3}{(\hs+\varepsilon)^3} \biggr) \theta_{\mrm l, \xi\xi} \theta_{\mrm l, \xi}  \idx = 4 \int \dfrac{\hs^2 \vert \hs'\vert^2 \vert\theta_{\mrm l,\xi\xi}\vert^2}{(\hs+\varepsilon)^2}\idx \\
        - 4 \int \biggl( \dfrac{\hs^2 \hs' \hs''}{(\hs +\varepsilon)^2} - \dfrac{\hs^2 ( \hs')^3}{(\hs+\varepsilon)^3} \biggr)' \vert\theta_{\mrm l, \xi}\vert^2 \idx\\
        = 4 \int \dfrac{\hs^2 \vert \hs'\vert^2 \vert\theta_{\mrm l,\xi\xi}\vert^2}{(\hs+\varepsilon)^2}\idx + 4 \int \biggl( \dfrac{- 2 \hs \vert\hs'\vert^2 \hs'' - \hs^2 \vert\hs''\vert^2 - \hs^2 \hs' \hs'''}{(\hs+\varepsilon)^2} \\
        + \dfrac{2 \hs \vert \hs' \vert^4 + 5 \hs^2 \vert \hs'\vert^2 \hs''}{(\hs+\varepsilon)^3} - \dfrac{3\hs^2 \vert \hs' \vert^4}{(\hs+\varepsilon)^4} \biggr) \vert\theta_{\mrm l, \xi}\vert^2 \idx.
    \end{gathered}
\end{equation}
Sending $ \varepsilon \rightarrow 0 $, thanks to \eqref{eq:equilibrium-1}, leads to 
\begin{equation}\label{lnest:054}
    \begin{gathered}
    \lim_{\varepsilon\rightarrow0^+} I_{01} = 4 \int \vert\hs'\vert^2 \vert\theta_{\mrm l, \xi\xi}\vert^2 \idx \\
    + 4 \int \bigl( - \vert \hs'' \vert^2 - \hs' \hs + \dfrac{3\vert\hs'\vert^2 \hs''}{\hs} - \dfrac{\vert\hs'\vert^4}{\hs^2} \bigr) \vert \theta_{\mrm l, \xi} \vert^2 \idx.
    \end{gathered}
\end{equation}
Therefore sending $ \varepsilon \rightarrow 0^+ $ in \eqref{lnest:052} and \eqref{lnest:053} yields
\begin{equation}\label{lnest:055}
    \begin{gathered}
         4 \int \bigl( - \vert \hs'' \vert^2 - \hs' \hs + \dfrac{3\vert\hs'\vert^2 \hs''}{\hs}  \bigr) \vert \theta_{\mrm l, \xi} \vert^2 \idx + 4 \norm{\hs' \theta_{\mrm l, \xi\xi}}{L^2}^2 + \underbrace{\norm{\dfrac{\partial_\xi(\hs^2 \theta_{\mrm l, \xi\xi})}{\hs}}{L^2}^2}_{I_{02}}\\ \lesssim \norm{\theta_{\mrm l, \xi}}{L^2}^2 + \norm{\dt \theta_{\mrm l, \xi}}{L^2}^2 
        + \norm{\hs^{1/2} \partial_{tt} \theta_{\mrm l}}{L^2}^2.
    \end{gathered}
\end{equation}
Further more, $ I_{02} $ in \eqref{lnest:055} can be calculated as
\begin{equation}
    \begin{gathered}
        I_{02} = \norm{\hs \theta_{\mrm l, \xi\xi\xi}}{L^2}^2 + 4 \norm{\hs' \theta_{\mrm l, \xi\xi}}{L^2}^2 + 4 \int \hs \hs' \theta_{\mrm l, \xi\xi} \theta_{\mrm l, \xi\xi\xi} \idx \\
        = \norm{\hs \theta_{\mrm l, \xi\xi\xi}}{L^2}^2 + 2 \norm{\hs' \theta_{\mrm l, \xi\xi}}{L^2}^2 - 2 \int \hs \hs'' \vert \theta_{\mrm l, \xi\xi} \vert^2 \idx.
    \end{gathered}
\end{equation}
Therefore, \eqref{lnest:055} implies
\begin{equation}\label{lnest:3rd-d-2}
    \norm{\hs \theta_{\mrm l, \xi\xi\xi}}{L^2}^2 +  \norm{\theta_{\mrm l, \xi\xi}}{L^2}^2 \lesssim \norm{\dfrac{\theta_{\mrm l, \xi}}{\hs^{1/2}}}{L^2}^2 + \norm{\dt \theta_{\mrm l, \xi}}{L^2}^2 
        + \norm{\hs^{1/2} \partial_{tt} \theta_{\mrm l}}{L^2}^2,
\end{equation}
which, together with \eqref{lnest:3rd-d-1}, implies
\begin{equation}\label{lnest:3rd-d-3}
    \norm{\hs \theta_{\mrm l, \xi\xi\xi}}{L^2}^2 +  \norm{\theta_{\mrm l, \xi\xi}}{L^2}^2 + \norm{\dfrac{\theta_{\mrm l, \xi}}{\hs^{1/2}}}{L^2}^2 \lesssim \norm{\theta_{\mrm l, \xi}}{L^2}^2 + \norm{\dt \theta_{\mrm l, \xi}}{L^2}^2 
    + \norm{\hs^{1/2} \partial_{tt} \theta_{\mrm l}}{L^2}^2.
\end{equation}

\smallskip

To sum up, thanks to \eqref{lnest:014} and \eqref{lnest:015} with $ k = 1 $, \eqref{lnest:3rd-d-3} implies
\begin{equation}\label{lnest:3rd-d-ttl}
    \begin{gathered}
        e^{\lambda_1 t} \biggr( \norm{\hs \theta_{\mrm l, \xi\xi\xi}(t)}{L^2}^2 +  \norm{\theta_{\mrm l, \xi\xi}(t)}{L^2}^2 + \norm{\dfrac{\theta_{\mrm l, \xi}(t)}{\hs^{1/2}}}{L^2}^2 \biggr)\\
        + \int_0^t e^{\lambda_1 s} \biggr( \norm{\hs \theta_{\mrm l, \xi\xi\xi}(s)}{L^2}^2 +  \norm{\theta_{\mrm l, \xi\xi}(s)}{L^2}^2 + \norm{\dfrac{\theta_{\mrm l, \xi}(s)}{\hs^{1/2}}}{L^2}^2 \biggr)\,\mrm{d}s\lesssim C_1.
    \end{gathered}
\end{equation}

\bigskip 

Repeating the same argument with $ \theta_{\mrm l} $ replaced by $ \dt \theta_{\mrm l} $, one can conclude that
\begin{equation}\label{lnest:034}
    \begin{gathered}
        e^{\lambda_2 t} \biggr( \norm{\hs \dt \theta_{\mrm l, \xi\xi\xi}(t)}{L^2}^2 +  \norm{\dt \theta_{\mrm l, \xi\xi}(t)}{L^2}^2 + \norm{\dfrac{\dt \theta_{\mrm l, \xi}(t)}{\hs^{1/2}}}{L^2}^2 \biggr)\\
        + \int_0^t e^{\lambda_2 s} \biggr( \norm{\hs \dt \theta_{\mrm l, \xi\xi\xi}(s)}{L^2}^2 +  \norm{\dt \theta_{\mrm l, \xi\xi}(s)}{L^2}^2 + \norm{\dfrac{\dt \theta_{\mrm l, \xi}(s)}{\hs^{1/2}}}{L^2}^2 \biggr)\,\mrm{d}s\lesssim C_2.
    \end{gathered}
\end{equation}
In particular, this implies that
\begin{equation}\label{lnest:035}
    \norm{\dt \theta_{\mrm l, \xi}(t)}{L^\infty}^2 \lesssim \norm{\dt \theta_{\mrm l, \xi\xi}(t)}{L^2}^2 \lesssim C_2 e^{-\lambda_2 t}, 
\end{equation}
and therefore the flow map is invertible. Moreover, the linearized system is asymptotically stable.

\subsubsection{Estimate of $ \theta_{\mrm l, \xi\xi\xi\xi} $}



From \eqref{lgeq:lnr_sys-1}, one can write
\begin{equation}\label{lnest:040}
    \begin{gathered}
        \hs^{2}\theta_{\mrm l, \xi\xi\xi\xi} + 4 \hs \hs' \theta_{\mrm l, \xi\xi\xi} = - (\hs^2)'' \theta_{\mrm l, \xi\xi} + \partial_\xi\lbrack (2 \vert \partial_\xi h_\mrm s \vert^2 - 4 h_\mrm s \partial_{\xi\xi} h_\mrm s) \theta_{\mrm l, \xi} \rbrack \\
        + (\hs^2 \theta_{\mrm l, \xi})_\xi + \partial_\xi(\hs \theta_{\mrm l, \xi t}) - \dfrac{1}{2} \hs \partial_{tt}\theta_{\mrm l} \\
        = \bigl( - 6 \hs \hs'' + \hs^2 \bigr) \theta_{\mrm l, \xi\xi} + \bigl( - 4 \hs \hs''' + 2 \hs \hs' \bigr) \theta_{\mrm l, \xi} \\
        + \hs \theta_{\mrm l, \xi\xi t} + \hs' \theta_{\mrm l, \xi t} - \dfrac{1}{2} \hs \partial_{tt} \theta_{\mrm l}.
    \end{gathered}
\end{equation}
Therefore, after dividing \eqref{lnest:040} with $ \hs^{1/2} $ and taking the $ L^2 $-norm of the resultant, one has that
\begin{equation}\label{lnest:041}
    \begin{gathered}
        \norm{\dfrac{\hs^{2}\theta_{\mrm l, \xi\xi\xi\xi} + 4 \hs \hs' \theta_{\mrm l, \xi\xi\xi}}{\hs^{1/2}}}{L^2}^2 \lesssim \norm{\hs^{1/2} \theta_{\mrm l, \xi\xi}}{L^2}^2 + \norm{\hs^{1/2}\theta_{\mrm l, \xi}}{L^2}^2 \\ + \norm{\hs^{1/2} \theta_{\mrm l, \xi\xi t}}{L^2}^2 
        + \norm{\dfrac{\theta_{\mrm l, \xi t}}{\hs^{1/2}}}{L^2}^2 + \norm{\hs^{1/2} \partial_{tt}\theta_{\mrm l}}{L^2}^2.
    \end{gathered}
\end{equation}
Meanwhile, the left hand side of \eqref{lnest:041} can be calculated as below:
\begin{equation}\label{lnest:042}
    \begin{gathered}
        \norm{\dfrac{\hs^{2}\theta_{\mrm l, \xi\xi\xi\xi} + 4 \hs \hs' \theta_{\mrm l, \xi\xi\xi}}{\hs^{1/2}}}{L^2}^2 = \norm{\hs^{3/2}\theta_{\mrm l, \xi\xi\xi\xi}}{L^2}^2 + 16 \norm{\hs^{1/2}\hs' \theta_{\mrm l, \xi\xi\xi}}{L^2}^2 \\
        + 8 \int \hs^2 \hs' \theta_{\mrm l, \xi\xi\xi} \theta_{\mrm l, \xi\xi\xi\xi} \idx = \norm{\hs^{3/2}\theta_{\mrm l, \xi\xi\xi\xi}}{L^2}^2 + 16 \norm{\hs^{1/2}\hs' \theta_{\mrm l, \xi\xi\xi}}{L^2}^2 \\
        - 4 \int (\hs^2 \hs')' \vert \theta_{\mrm l, \xi\xi\xi} \vert^2 \idx = \norm{\hs^{3/2}\theta_{\mrm l, \xi\xi\xi\xi}}{L^2}^2 \\
        + \int \bigl( 8 \hs (\hs')^2 - 4 \hs^2 \hs''\bigr) \vert \theta_{\mrm l,\xi\xi\xi} \vert^2 \idx \gtrsim \norm{\hs^{3/2}\theta_{\mrm l, \xi\xi\xi\xi}}{L^2}^2 + \norm{\hs^{1/2}\theta_{\mrm l, \xi\xi\xi}}{L^2}^2.
    \end{gathered}
\end{equation}
Therefore, thanks to \eqref{lnest:014}, \eqref{lnest:015}, \eqref{lnest:3rd-d-ttl}, and \eqref{lnest:034},
\eqref{lnest:041} and \eqref{lnest:042} imply
\begin{equation}\label{lnest:043}
    \begin{gathered}
        e^{\lambda_2 t} \bigl( \norm{\hs^{3/2}\theta_{\mrm l, \xi\xi\xi\xi}(t)}{L^2}^2 + \norm{\hs^{1/2}\theta_{\mrm l, \xi\xi\xi}(t)}{L^2}^2 \bigr) \\
        + \int e^{\lambda_2 s} \bigl( \norm{\hs^{3/2}\theta_{\mrm l, \xi\xi\xi\xi}(s)}{L^2}^2 + \norm{\hs^{1/2}\theta_{\mrm l, \xi\xi\xi}(s)}{L^2}^2 \bigr) \,\mrm{d}s\lesssim C_2.
    \end{gathered}
\end{equation}

\section{Nonlinear elliptic estimates}\label{sec:nonlinear_estimates}

In this section, we demonstrate the nonlinear elliptic estimates for the solution to \eqref{lgeq:prtbtn-csv}, i.e., shifting the temporal derivative to the spatial derivation. Let 
\begin{equation}\label{def:energy_nonlinear_1}
    \mathcal E_{\mrm{NL},1} := \sum_{k=0,1,2} \biggl\lbrace \norm{\hs^{1/2} \partial_t^{k+1} \theta}{L^2}^2 + \norm{\hs \partial_t^k \theta_{\xi\xi}}{L^2}^2 + \norm{\partial_t^k \theta_\xi}{L^2}^2 \biggr\rbrace.
\end{equation}
We will show that
\begin{equation}\label{def:energy_nonlinear_2}
    \begin{gathered}
        \mathcal E_{\mrm{NL},2} := \norm{\hs^{3/2} \theta_{\xi\xi\xi\xi}}{L^2}^2 + \norm{\hs^{1/2} \theta_{\xi\xi\xi}}{L^2}^2 \\
        + \sum_{k=0,1} \biggl\lbrace \norm{\hs \dt^k \theta_{\xi\xi\xi}}{L^2}^2 + \norm{\dt^k \theta_{\xi\xi}}{L^2}^2 + \norm{\dfrac{\dt^k \theta_{\xi}}{\hs^{1/2}}}{L^2}^2\biggr\rbrace, 
    \end{gathered}
\end{equation}
i.e, the (weighted) estimates of $ \theta_{\xi\xi\xi\xi} $, $ \partial_t \theta_{\xi\xi\xi} $, and $ \theta_{\xi\xi\xi} $, can be bounded in terms of $ \mathcal E_{\mrm{NL},1} $ provided that $ \mathcal E_{\mrm{NL},2} $ is small.

\smallskip

We first rewrite the equation \eqref{lgeq:prtbtn-csv}, by separating the linear and nonlinear parts. Notice that
\begin{equation}\label{id:nonlinearities-1}
    \dfrac{1}{1+\theta_\xi} = 1 - \theta_\xi + \underbrace{\dfrac{\theta_\xi^2}{1+\theta_\xi}}_{=:g_1(\theta_\xi)}
    \quad \text{and} \quad
    \dfrac{1}{(1+\theta_\xi)^2} = 1 - 2 \theta_\xi + \underbrace{\dfrac{2\theta_\xi + 3}{(1+\theta_\xi)^2} \theta_\xi^2}_{=:g_2(\theta_\xi)},
\end{equation}
where, for small $ \theta_\xi $,
\begin{equation}\label{id:nonlinearities-2}
    g_1(\theta_\xi), g_2(\theta_\xi) = \mathcal O(\theta_\xi^2).
\end{equation}
With these notations, \eqref{lgeq:prtbtn-csv} can be written as
\begin{equation}\label{eq:lin+nonlin}
    \begin{gathered}
        \hs \partial_{tt} \theta - 2 \bigl\lbrack (\hs^2 + 2 \vert \hs' \vert^2 - 4 \hs \hs'') \theta_\xi \bigr\rbrack_\xi + 2 \bigl\lbrack \hs^2 \theta_{\xi\xi} \bigr\rbrack_{\xi\xi} - 2 \bigl( \hs \theta_{\xi t} \bigr)_\xi \\
        + \bigl\lbrack N_1 \bigr\rbrack_\xi - 2 \bigl\lbrack \hs ( - 2\theta_\xi + g_2(\theta_\xi)) \theta_{\xi t} \bigr\rbrack_\xi = 0, 
    \end{gathered}
\end{equation}
where
\begin{equation}\label{id:nonlinearity-3}
    \begin{gathered}
        N_1 := \hs^2 g_2(\theta_\xi) - 4 \hs \theta_\xi \bigl[ \hs' \theta_\xi + (\hs \theta_\xi)_\xi  \bigr]_\xi \\
        - 2 \hs (1-2\theta_\xi + g_2(\theta_\xi)) \bigl\lbrace \theta_\xi (\hs \theta_\xi)_\xi + (1-\theta_\xi +g_1(\theta_\xi)) [\hs g_1(\theta_\xi)]_\xi \\
        + g_1(\theta_\xi) [\hs(1-\theta_\xi)]_\xi \bigr\rbrace_\xi \\
        - 2 \hs g_2(\theta_\xi) \bigl\lbrack \hs' - \hs' \theta_\xi - (\hs \theta_\xi)_\xi \bigr\rbrack_\xi + \bigl\lbrack \hs' \theta_\xi + (\hs \theta_\xi)_\xi \bigr\rbrack^2 \\
        + \bigl\lbrace 2\hs' - 2\hs' \theta_\xi - 2(\hs \theta_\xi)_\xi + \theta_\xi (\hs\theta_\xi)_\xi +(1-\theta_\xi + g_1(\theta_\xi))(\hs g_1(\theta_\xi))_\xi \\ + g_1(\theta_\xi)[\hs (1-\theta_\xi)]_\xi \bigr\rbrace 
        \times 
        \bigl\lbrace \theta_\xi (\hs \theta_\xi)_\xi + (1-\theta_\xi + g_1(\theta_\xi))[\hs g_1(\theta_\xi)]_\xi\\ + g_1(\theta_\xi) [\hs(1-\theta_\xi)]_\xi \bigr\rbrace,
    \end{gathered}
\end{equation}
or, by denoting $ f_i = f_i(y) \in C^\infty(-1/2,1/2) $, $ i = 1,2,\cdots,6 $, 
\begin{equation}\label{id:nonlinearity-4}
    \begin{gathered}
    N_1 = \hs^2 \bigl\lbrace f_1(\theta_\xi) \theta_\xi^2 + f_2(\theta_\xi) \theta_\xi \theta_{\xi\xi\xi} + f_3(\theta_\xi) \theta_{\xi\xi} \theta_{\xi\xi} \bigr\rbrace \\
    + \hs \hs' f_4(\theta_\xi) \theta_\xi \theta_{\xi\xi} + (\hs')^2 f_5(\theta_\xi) \theta_\xi^2 + \hs \hs'' f_6(\theta_\xi) \theta_{\xi}^2.
    \end{gathered}
\end{equation}
To simplify the presentation, hereafter we use
\begin{equation}\label{id:nonlinearity-5}
    f = f(y) \in C^\infty (-1/2,1/2) \quad \text{and} \quad F = F(y) \geq 0 \in C^\infty [0,\infty)
\end{equation}
to denote the smooth functions of the argument, which is different from line to line.

\subsection{Embedding inequalities}
We summarize the weighted-$ L^p $ embedding inequalities used in this section. These inequalities are consequences of Hardy's inequalities (Lemma \ref{lm:hardy-p}) and the Sobolev embedding inequalities. 

\begin{lem}\label{lm:embedding-1}The following inequalities hold:
    \begin{equation}\label{ineq:001}
        \begin{gathered}
            \norm{\theta_\xi}{L^\infty} \lesssim  \norm{\theta_{\xi\xi}}{L^2} \lesssim \norm{\hs \theta_{\xi\xi\xi}}{L^2} + \norm{\hs \theta_{\xi\xi}}{L^2}, \\
            \norm{\theta_{\xi t}}{L^\infty} \lesssim  \norm{\theta_{\xi\xi t}}{L^2} \lesssim \norm{\hs \theta_{\xi\xi\xi t}}{L^2} + \norm{\hs \theta_{\xi\xi t}}{L^2}, \\
            \norm{\frac{\theta_{\xi t}}{\hs^{1/2}} }{L^2}\lesssim  \norm{\frac{\theta_{\xi t}}{\hs}}{L^2} \lesssim  \norm{\theta_{\xi\xi t}}{L^2} \lesssim \norm{\hs \theta_{\xi\xi\xi t}}{L^2} + \norm{\hs \theta_{\xi\xi t}}{L^2},\\
            \norm{\hs \theta_{\xi\xi}}{L^\infty} \lesssim \norm{\hs \theta_{\xi\xi\xi}}{L^2} + \norm{\theta_{\xi\xi}}{L^2}, \\
            \norm{\hs \theta_{\xi\xi t}}{L^\infty} \lesssim \norm{\hs \theta_{\xi\xi\xi t}}{L^2} + \norm{\theta_{\xi\xi t}}{L^2},\\
            \norm{\hs^{1/2} \theta_{\xi\xi}}{L^\infty} \lesssim \norm{\hs^{1/2} \theta_{\xi\xi\xi}}{L^{3/2}} + \norm{\hs^{-1/2} \theta_{\xi\xi}}{L^{3/2}} \\
            \lesssim \norm{\hs^{1/2}\theta_{\xi\xi\xi}}{L^2}  + \norm{\theta_{\xi\xi}}{L^2}.
        \end{gathered}
    \end{equation}
\end{lem}
\begin{proof}
    This is a direct consequence of applying the Sobolev embedding inequalities and Hardy's inequalities.
\end{proof}

\begin{lem}\label{lm:embedding-2}
    The following inequalities hold:
    \begin{gather}
        \norm{\dfrac{\theta_\xi}{\hs}}{L^4} + \norm{\theta_{\xi\xi}}{L^4} \lesssim \norm{\hs^{3/2} \theta_{\xi\xi\xi\xi}}{L^2} + \norm{\hs^{1/2} \theta_{\xi\xi\xi}}{L^2} + \norm{\theta_{\xi\xi}}{L^2}, \label{ineq:002}\\
        \norm{\hs^{1/2} \theta_{\xi tt}}{L^4} \lesssim \norm{\hs \theta_{\xi\xi tt}}{L^2} + \norm{\theta_{\xi tt}}{L^2}.
        \label{ineq:003}
    \end{gather}
\end{lem}
\begin{proof}
Recall that $ \theta_\xi\vert_{\xi=-1,1} = 0 $. Thanks to Hardy's inequalities and the Sobolev inequalities one has that 
    \begin{equation}
    \begin{gathered}
        \norm{\dfrac{\theta_\xi}{\hs}}{L^4}^4 \underbrace{\lesssim}_{\mathclap{\eqref{ineq:Hardy-4}~ \text{with} ~ k = 0, p = 4}} \norm{\theta_{\xi\xi}}{L^4}^4 \lesssim \norm{\hs \theta_{\xi\xi\xi}}{L^4}^4 + \norm{\hs \theta_{\xi\xi}}{L^4}^4 \\
        \lesssim \norm{\hs^{3/2} \theta_{\xi\xi\xi}}{L^\infty}^2 \norm{\hs^{1/2} \theta_{\xi\xi\xi}}{L^2}^2 + \norm{\hs \theta_{\xi\xi}}{L^\infty}^2 \norm{\hs \theta_{\xi\xi}}{L^2}^2 \\
        \lesssim \bigl( \norm{\hs^{3/2} \theta_{\xi\xi\xi\xi}}{L^2}^2 + \norm{\hs^{1/2} \theta_{\xi\xi\xi}}{L^2}^2 + \norm{\theta_{\xi\xi}}{L^2}^2 \bigr)^2,
    \end{gathered}
    \end{equation}
    which proves \eqref{ineq:002}.
    
    \smallskip
    
    To prove \eqref{ineq:003}, applying Hardy's inequalities and the Sobolev inequalities yields
    \begin{equation}
        \begin{gathered}
            \norm{\hs^{1/2} \theta_{\xi tt}}{L^4}^4 \lesssim \norm{\hs \theta_{\xi tt}}{L^\infty}^2 \norm{\theta_{\xi tt}}{L^2}^2 \lesssim (\norm{\hs \theta_{\xi\xi tt}}{L^2}^2 + \norm{\theta_{\xi tt}}{L^2}^2 ) \norm{\theta_{\xi tt}}{L^2}^2.
        \end{gathered}
    \end{equation}
    This finishes the proof of \eqref{ineq:003}.
\end{proof}

\subsection{Estimates of $\theta_{\xi\xi\xi\xi} $}
\label{sec:ell-est-001}

Similar to \eqref{lnest:040}, one can write, from \eqref{eq:lin+nonlin}, that 
\begin{equation}\label{e-est:001}
    \begin{gathered}
        \hs^2 \theta_{\xi\xi\xi\xi} + 4 \hs \hs' \theta_{\xi\xi\xi} = ( - 6 \hs \hs'' + \hs^2 ) \theta_{\xi\xi} + ( -4 \hs \hs''' + 2 \hs \hs' ) \theta_\xi \\
        + \hs \theta_{\xi\xi t} + \hs' \theta_{\xi t} - \dfrac{1}{2} \hs \partial_{tt} \theta - \dfrac{1}{2} \bigl\lbrack N_1 \bigr\rbrack_\xi + \bigl\lbrack \hs ( -2\theta_\xi + g_2(\theta_\xi)) \theta_{\xi t} \bigr\rbrack_\xi.
    \end{gathered}
\end{equation}
Then applying the same arguments as in \eqref{lnest:041}--\eqref{lnest:042} yields
\begin{equation}\label{e-est:002}
    \begin{gathered}
        \norm{\hs^{3/2}\theta_{\xi\xi\xi\xi}}{L^2}^2 + \norm{\hs^{1/2} \theta_{\xi\xi\xi}}{L^2}^2 \lesssim \norm{\dfrac{\hs^2 \theta_{\xi\xi\xi\xi} + 4 \hs \hs' \theta_{\xi\xi\xi}}{\hs^{1/2}}}{L^2}^2\\
        \lesssim \norm{\hs^{1/2} \theta_{\xi\xi}}{L^2}^2 + \norm{\hs^{1/2} \theta_{\xi}}{L^2}^2  + \norm{\hs^{1/2}\theta_{\xi\xi t}}{L^2}^2 + \norm{\dfrac{\theta_{\xi t}}{\hs^{1/2}}}{L^2}^2 + \norm{\hs^{1/2}\theta_{tt}}{L^2}^2 \\
        + \norm{\dfrac{(N_1)_\xi}{\hs^{1/2}}}{L^2}^2 + F(\norm{\theta_{\xi}}{L^\infty}) \norm{\theta_{\xi}}{L^\infty}^2 \norm{\dfrac{\theta_{\xi t}}{\hs^{1/2}}}{L^2}^2 \\
        + F(\norm{\theta_{\xi}}{L^\infty}) \norm{\theta_{\xi}}{L^\infty}^2 \norm{\hs^{1/2} \theta_{\xi\xi t}}{L^2}^2 
        + F(\norm{\theta_{\xi}}{L^\infty}) \norm{\theta_{\xi t}}{L^\infty}^2 \norm{\hs^{1/2} \theta_{\xi\xi}}{L^2}^2.
    \end{gathered}
\end{equation}
To calculate $ \norm{\dfrac{(N_1)_\xi}{\hs^{1/2}}}{L^2}^2 $, from \eqref{id:nonlinearity-4}, one can calculate
\begin{equation}\label{e-est:003}
    \begin{gathered}
        (N_1)_\xi = \hs^2 \lbrace f(\theta_\xi) \theta_\xi \theta_{\xi\xi} + f(\theta_\xi) \theta_\xi \theta_{\xi\xi\xi\xi} + f(\theta_\xi) \theta_{\xi\xi} \theta_{\xi\xi\xi} + f(\theta_\xi) \theta_{\xi\xi}\theta_{\xi\xi}\theta_{\xi\xi}   \rbrace \\
        + \hs \hs' \lbrace f(\theta_\xi) \theta_\xi^2  + f(\theta_\xi) \theta_\xi \theta_{\xi\xi\xi} + f(\theta_\xi) \theta_{\xi\xi}\theta_{\xi\xi} \rbrace \\
         + ( (\hs')^2 + \hs \hs'') \lbrace f(\theta_\xi) \theta_\xi \theta_{\xi\xi} \rbrace
        + ( (\hs')^2 + \hs \hs'')' f(\theta_\xi) \theta_\xi^2.
    \end{gathered}
\end{equation}
Therefore one has that
\begin{equation}\label{e-est:004}
    \begin{gathered}
        \norm{\dfrac{(N_1)_\xi}{\hs^{1/2}}}{L^2}^2 \lesssim F(\norm{\theta_\xi}{L^\infty}) \biggl\lbrace \norm{\theta_\xi}{L^\infty}^2 \norm{\hs^{3/2}\theta_{\xi\xi}}{L^2}^2 + \norm{\theta_\xi}{L^\infty}^2 \norm{\hs^{3/2}\theta_{\xi\xi\xi\xi}}{L^2}^2 \\
        + \norm{\hs \theta_{\xi\xi}}{L^\infty}^2 \norm{\hs^{1/2} \theta_{\xi\xi\xi}}{L^2}^2 + \norm{\hs \theta_{\xi\xi}}{L^\infty}^2\norm{\hs^{1/4} \theta_{\xi\xi}}{L^4}^4 + \norm{\theta_\xi}{L^\infty}^2 \norm{\hs^{1/2}\theta_\xi}{L^2}^2 \\
        + \norm{\theta_\xi}{L^\infty}^2 \norm{\hs^{1/2} \theta_{\xi\xi\xi}}{L^2}^2 + \norm{\hs^{1/4}\theta_{\xi\xi}}{L^4}^4 + \norm{\dfrac{\theta_\xi}{\hs}}{L^4}^2 \norm{\hs^{1/2}\theta_{\xi\xi}}{L^4}^2 \\
        + \norm{\theta_{\xi}}{L^\infty}^2 \norm{\dfrac{\theta_\xi}{\hs^{1/2}}}{L^2}^2
        \biggr\rbrace.
    \end{gathered}
\end{equation}
Consequently, one can conclude from \eqref{e-est:002}--\eqref{e-est:004} that
\begin{equation}\label{e-est:4th-dxi}
    \begin{gathered}
    \norm{\hs^{3/2}\theta_{\xi\xi\xi\xi}}{L^2}^2 + \norm{\hs^{1/2} \theta_{\xi\xi\xi}}{L^2}^2 \lesssim \norm{\theta_{\xi\xi}}{L^2}^2 + \norm{\theta_{\xi\xi t}}{L^2}^2 + \norm{\dfrac{\theta_{\xi t}}{\hs^{1/2}}}{L^2}^2 \\
    + \mathcal E_{\mrm{NL},1} + F(\mathcal{E}_{\mrm{NL},2}) \mathcal{E}_{\mrm{NL},2}^2,
    \end{gathered}
\end{equation}
thanks to \eqref{ineq:001} and \eqref{ineq:002}.

\subsection{Estimates of $ \theta_{\xi\xi\xi t} $ and $ \theta_{\xi\xi\xi} $}

Integrating \eqref{eq:lin+nonlin} from $\xi = -1 $ to $ \xi $ yields
\begin{equation}\label{eq:lin+nonlin-2}
    \begin{gathered}
        2 \hs \dt \theta_{\xi} - 2 \partial_\xi (\hs^2 \theta_{\xi\xi}) + 2 (2\vert\hs'\vert^2 - 4 \hs \hs'' + \hs^2 ) \theta_\xi \\ = \int_{-1}^\xi (\hs \partial_{tt}\theta)(\sigma)\,d\sigma + N_1 - 2 \hs (-2\theta_\xi + g_2(\theta_\xi))\theta_{\xi t}.
    \end{gathered}
\end{equation}
Then, similar to \eqref{lnest:051}, \eqref{eq:lin+nonlin-2} can be written as 
\begin{equation}\label{e-est:101}
    \begin{gathered}
        - \partial_\xi (\hs^2 \theta_{\xi\xi}) + 2 \vert\hs'\vert^2 \theta_\xi = (4\hs \hs'' - \hs^2) \theta_\xi - \hs\dt \theta_{\xi} + \dfrac{1}{2} \int_{-1}^\xi (\hs \partial_{tt}\theta)(\sigma)\,d\sigma\\
        + \dfrac{1}{2} N_1 - \hs(-2\theta_\xi + g_2(\theta_\xi))\theta_{\xi t}.
    \end{gathered}
\end{equation}
Repeating the same arguments as in \eqref{lnest:071}--\eqref{lnest:3rd-d-3} leads to
\begin{equation}\label{e-est:102}
    \begin{gathered}
        \sum_{k=0,1} \biggl\lbrace  \norm{\hs \dt^k \theta_{\xi\xi\xi}}{L^2}^2 + \norm{\dt^k \theta_{\xi\xi}}{L^2}^2 + \norm{\dfrac{\dt^k \theta_{\xi}}{\hs^{1/2}}}{L^2}^2 \biggr\rbrace \lesssim \mathcal E_{\mrm{NL},1}\\
        + F(\mathcal E_{\mrm{NL},2})\mathcal E_{\mrm{NL},2}^2 + \norm{\dfrac{N_1}{\hs}}{L^2}^2 + \norm{\dfrac{\dt N_1}{\hs}}{L^2}^2.
    \end{gathered}
\end{equation}
It suffices to calculate $ \norm{\dfrac{\dt N_1}{\hs}}{L^2}^2 $. The estimate of  $ \norm{\dfrac{N_1}{\hs}}{L^2} ^2 $ follows similarly. 

From \eqref{id:nonlinearity-4}, direct calculation yields
\begin{equation}\label{e-est:103}
    \begin{gathered}
        \dt N_1 = \hs^2 \bigl\lbrace f(\theta_\xi) \theta_\xi \theta_{\xi t} + f(\theta_\xi) \theta_{\xi t} \theta_{\xi\xi\xi} + f(\theta_\xi) \theta_{\xi} \theta_{\xi\xi\xi t} + f(\theta_\xi) \theta_{\xi t} \theta_{\xi\xi} \theta_{\xi\xi} \\
        + f(\theta_\xi) \theta_{\xi\xi} \theta_{\xi\xi t} \bigr\rbrace \\
        + \hs \hs' f(\theta_\xi) \theta_{\xi t} \theta_{\xi\xi} + \hs \hs' f(\theta_\xi) \theta_{\xi}\theta_{\xi\xi t} 
        + (\vert \hs'\vert^2 + \hs \hs'') f(\theta_\xi) \theta_{\xi} \theta_{\xi t}.
    \end{gathered}
\end{equation}
Therefore one has that
\begin{equation}\label{e-est:104}
    \begin{gathered}
        \norm{\dfrac{\dt N_1}{\hs}}{L^2}^2 \lesssim F(\norm{\theta_\xi}{L^\infty}) \biggl\lbrace \norm{\theta_{\xi}}{L^\infty}^2 \norm{\hs \theta_{\xi t}}{L^2}^2 + \norm{\theta_{\xi t}}{L^\infty}^2 \norm{\hs \theta_{\xi\xi\xi}}{L^2}^2 \\
        + \norm{\theta_{\xi}}{L^\infty}^2 \norm{\hs\theta_{\xi\xi\xi t}}{L^2}^2 + \norm{\theta_{\xi t}}{L^\infty}^2 \norm{\hs \theta_{\xi\xi}}{L^\infty}^2 \norm{\theta_{\xi\xi}}{L^2}^2 + \norm{\hs \theta_{\xi\xi}}{L^\infty}^2 \norm{\theta_{\xi\xi t}}{L^2}^2 \\
        + \norm{\theta_{\xi t}}{L^\infty}^2 \norm{\theta_{\xi\xi}}{L^2}^2 + \norm{\theta_{\xi}}{L^\infty}^2 \norm{\theta_{\xi\xi t}}{L^2}^2 + \norm{\dfrac{\theta_{\xi}}{\hs}}{L^4}^2 \norm{\theta_{\xi t}}{L^4}^2
        \biggr\rbrace.
    \end{gathered}
\end{equation}
Consequently, one can conclude from \eqref{e-est:102}--\eqref{e-est:104} that
\begin{equation}\label{e-est:3rd-dxi}
    \begin{gathered}
        \sum_{k=0,1} \biggl\lbrace  \norm{\hs \dt^k \theta_{\xi\xi\xi}}{L^2}^2 + \norm{\dt^k \theta_{\xi\xi}}{L^2}^2 + \norm{\dfrac{\dt^k \theta_{\xi}}{\hs^{1/2}}}{L^2}^2 \biggr\rbrace \lesssim \mathcal E_{\mrm{NL},1}\\
        + F(\mathcal E_{\mrm{NL},2})\mathcal E_{\mrm{NL},2}^2.
    \end{gathered}
\end{equation}

\smallskip

In summary, from \eqref{e-est:4th-dxi} and \eqref{e-est:3rd-dxi}, we have shown that
\begin{equation}\label{e-est:total-1}
    \mathcal E_{\mrm{NL},2} \lesssim \mathcal E_{\mrm{NL},1}
    + F(\mathcal E_{\mrm{NL},2})\mathcal E_{\mrm{NL},2}^2.
\end{equation}

\section{Nonlinear {\it \textbf{a priori}} estimates and asymptotic stability}\label{sec:nonlinear_apriori}

\subsection{Energy estimates}
We start with rewrite \eqref{lgeq:prtbtn-csv} as
\begin{equation}\label{lgeq:prtbtn-csv-2}
    \begin{gathered}
        \hs \partial_{tt} \theta + \biggl\lbrace \dfrac{\hs^2}{(1+\theta_\xi)^2} + \dfrac{(\hs')^2}{(1+\theta_\xi)^4} - \dfrac{2\hs \hs''}{(1+\theta_{\xi})^4} + \dfrac{5}{4} \hs^2 \biggl\lbrack \biggl( \dfrac{1}{(1+\theta_\xi)^2} \biggr)_\xi \biggr\rbrack^2 \\
        - 2 \hs \biggl(\dfrac{1}{1+\theta_\xi}\biggr)_t 
        \biggr\rbrace_\xi - \biggl\lbrace \dfrac{\hs^2}{2} \biggl(\dfrac{1}{(1+\theta_\xi)^4}\biggr)_\xi\biggr\rbrace_{\xi\xi} = 0,\\
        \qquad \text{with} \quad \theta_{\xi} \vert_{\xi = -1, 1} = 0. 
    \end{gathered}
\end{equation}
Similar to \eqref{id:nonlinearities-1}, one has
\begin{equation}\label{id:nonlinearities-6}
    \dfrac{1}{(1+\theta_\xi)^k} = 1 - k \theta_\xi + g_k(\theta_\xi), \qquad i = 1,2,3,4, 
\end{equation}
where, for small $ \theta_\xi $,
\begin{equation}
    g_k(\theta_\xi) = \mathcal O(\theta_\xi^2). 
\end{equation}
Then one can separate \eqref{lgeq:prtbtn-csv-2} into the linear and nonlinear parts, by writing
\begin{equation}
    \begin{gathered}
        \hs \partial_{tt} \theta - 2 \bigl\lbrace (\hs^2 + 2 (\hs')^2 - 4 \hs \hs'') \theta_\xi +  \hs \theta_{\xi t} \bigr\rbrace_\xi + 2 (\hs^2 \theta_{\xi\xi})_{\xi\xi}\\
        + (M_1)_\xi + (M_2)_{\xi\xi} = 0,
    \end{gathered}
\end{equation}
where
\begin{align}
    & \label{id:nonlinear-7} \begin{aligned}
        M_1  := & \hs^2 g_2(\theta_\xi) + \lbrack (\hs')^2 -2 \hs \hs'' \rbrack g_4(\theta_\xi)\\ & + \dfrac{5}{4} \hs^2 \lbrack (-2\theta_\xi + g_2(\theta_\xi))_\xi \rbrack^2 + 2 \hs (g_1(\theta_\xi))_t,
    \end{aligned}
    \\
    & \label{id:nonlinear-8} \begin{aligned}
        M_2 := & - \dfrac{\hs^2}{2} (g_4(\theta_\xi))_\xi, 
    \end{aligned}
\end{align}
or using \eqref{id:nonlinearity-5}
\begin{align}
& \label{id:nonlinear-9}
    \begin{aligned}
        M_1 = & \hs^2 \lbrace f(\theta_\xi) \theta_\xi^2 + f(\theta_\xi) \theta_\xi \theta_{\xi\xi} +  f(\theta_\xi) \theta_{\xi\xi}^2 \rbrace + \lbrack (\hs')^2 -2 \hs \hs'' \rbrack f(\theta_\xi) \theta_\xi^2 \\
        & + \hs f(\theta_\xi) \theta_\xi \theta_{\xi t},
        \end{aligned}\\
& \label{id:nonlinear-10}
        \begin{aligned}
        M_2 = & \hs^2 f(\theta_\xi) \theta_\xi \theta_{\xi\xi}.
    \end{aligned}
\end{align}

Now we are ready to establish the estimates of $ \mathcal E_{\mrm{NL},1} $. 
In particular, let 
\begin{align}
    &\label{lnest:energy-k} \begin{aligned}
        \mathcal E_{k} : = &  \dfrac{1}{2} \int h_\mrm s \vert \dt^{k+1} \theta \vert^2 \,\idx + \int ( 2 \vert \hs'\vert^2 - 4 \hs \hs'' + h_\mrm s^2) \vert \dt^k \theta_{\xi} \vert^2 \,\idx \\
        & \quad +  \int \vert h_\mrm s \vert^2 \vert \dt^k \theta_{\xi\xi} \vert^2 \,\idx + \mfk c_1 \int h_\mrm s \vert \dt^k \theta_{\xi} \vert^2 \,\idx + \mfk c_1 \int h_\mrm s \dt^{k+1} \theta \cdot \dt^k \theta \,\idx, 
    \end{aligned}
        \\
    &\label{lnest:dissipation-k}\begin{aligned}    
        \mathcal D_{k} := & 2 \int h_\mrm s \vert \dt^{k+1} \theta_{\xi} \vert^2 \,\idx - \mfk c_1 \int h_\mrm s \vert \dt^{k+1}\theta \vert^2 \,\idx + 2 \mfk c_1 \int h_\mrm s^2 \vert \dt^k \theta_{\xi\xi} \vert^2 \,\idx \\
        & \qquad + 2 \mfk c_1 \int ( 2 \vert \hs' \vert^2 - 4 \hs \hs'' + h_\mrm s^2 \bigr) \vert \dt^k \theta_{ \xi} \vert^2 \,\idx.
    \end{aligned}
\end{align}
Then similar to \eqref{lnest:010} and \eqref{lnest:011}
\begin{equation}\label{est:701}
    \begin{gathered}
    \mathcal D_k \gtrsim \norm{\hs^{1/2} \dt^{k+1} \theta_\xi}{L^2}^2 + \norm{\hs \dt^k \theta_{\xi\xi}}{L^2}^2 + \norm{\dt^k \theta_\xi}{L^2}^2 \\
    \gtrsim \mathcal E_k \gtrsim \norm{\hs^{1/2} \dt^{k+1} \theta}{L^2}^2 + \norm{\hs \dt^k \theta_{\xi\xi}}{L^2}^2 + \norm{\dt^k \theta_\xi}{L^2}^2,
    \end{gathered}
\end{equation}
and therefore
\begin{equation}\label{est:702}
    \sum_{k=0,1,2} \mathcal E_k \lesssim \mathcal E_{\mrm{NL},1} \lesssim \sum_{k=0,1,2} \mathcal E_k. 
\end{equation}

\smallskip

We calculate the estimate of $ \mathcal E_2 $. The estimates of $ \mathcal E_0 $ and $ \mathcal E_1 $ follow with similar arguments. 

After applying $ \dt^2 $ to \eqref{eq:lin+nonlin-2} and repeating the same arguments as in \eqref{lnest:003}--\eqref{lnest:009}, one can conclude
\begin{equation}\label{est:703}
    \begin{gathered}
    \dfrac{\mathrm{d}}{\mathrm{d}t} \mathcal E_2 + \mathcal D_2 = \int (M_1)_{tt} \theta_{\xi ttt} \idx - \int (M_2)_{tt}  \theta_{\xi\xi ttt} \idx \\ + \mathfrak c_1 \int (M_1)_{tt} \theta_{\xi tt} \idx - \mathfrak c_1 \int (M_2)_{tt} \theta_{\xi\xi tt} \idx.
    \end{gathered}
\end{equation}
Therefore, it suffices to estimate the right-hand side of \eqref{est:703}. 

\bigskip 

{\par\noindent\bf Estimates of $\int (M_1)_{tt} \theta_{\xi ttt} \idx $ and $ \int (M_1)_{tt} \theta_{\xi tt} \idx $}

Applying H\"older's inequality yields
\begin{equation}\label{est:704}
    \begin{gathered}
    \vert\int (M_1)_{tt} \theta_{\xi ttt} \idx\vert +  \vert\int (M_1)_{tt} \theta_{\xi tt} \idx\vert \\ \lesssim \norm{\dfrac{(M_1)_{tt}}{\hs^{1/2}}}{L^2} ( \norm{\hs^{1/2}\theta_{\xi ttt}}{L^2} + \norm{\hs^{1/2}\theta_{\xi tt}}{L^2}).
    \end{gathered}
\end{equation}
From \eqref{id:nonlinear-9}, one can calculate
\begin{equation}\label{est:705}
    \begin{aligned}
    & (M_1)_{tt} = \hs^2 \lbrace f(\theta_\xi) \theta_\xi \theta_{\xi tt} + f(\theta_\xi) \theta_{\xi t}^2 + f(\theta_\xi) \theta_\xi \theta_{\xi\xi tt} + f(\theta_\xi) \theta_{\xi t} \theta_{\xi\xi t} \\
    & \qquad + f(\theta_\xi)\theta_{\xi tt} \theta_{\xi\xi} + f(\theta_\xi) \theta_{\xi t}^2 \theta_{\xi\xi} + f(\theta_\xi) \theta_{\xi\xi} \theta_{\xi\xi tt} + f(\theta_\xi) \theta_{\xi\xi t}^2\\
    & \qquad + f(\theta_\xi) \theta_{\xi t}\theta_{\xi\xi} \theta_{\xi\xi t} + f(\theta_\xi) \theta_{\xi tt} \theta_{\xi\xi}^2+ f(\theta_\xi) \theta_{\xi t}^2\theta_{\xi\xi}^2
    \rbrace \\
    & \quad + \hs \lbrace f(\theta_\xi) \theta_{\xi} \theta_{\xi ttt} + f(\theta_\xi) \theta_{\xi t} \theta_{\xi tt} + f(\theta_\xi) \theta_{\xi t}^3
    \rbrace \\
    & \quad + [ (\hs')^2 - 2\hs \hs'']\lbrace f(\theta_\xi) \theta_\xi \theta_{\xi tt} + f(\theta_\xi) \theta_{\xi t}^2
    \rbrace.
    \end{aligned}
\end{equation}
Therefore, one can calculate
\begin{equation}\label{est:706}
    \begin{gathered}
        \norm{\dfrac{(M_1)_{tt}}{\hs^{1/2}}}{L^2} \lesssim F(\norm{\theta_{\xi}}{L^\infty})\biggl\lbrace  
        \norm{\theta_{\xi}}{L^\infty} \norm{\hs^{3/2} \theta_{\xi tt}}{L^2} + \norm{\theta_{\xi t}}{L^\infty} \norm{\hs^{3/2} \theta_{\xi t}}{L^2} \\
        + \norm{\theta_{\xi}}{L^\infty} \norm{\hs^{3/2} \theta_{\xi\xi tt}}{L^2} + \norm{\theta_{\xi t}}{L^\infty} \norm{\hs^{3/2} \theta_{\xi\xi t}}{L^2} +
        \norm{\hs^{1/2}  \theta_{\xi\xi}}{L^\infty} \norm{\hs\theta_{\xi tt}}{L^2}\\
        + \norm{\theta_{\xi t}}{L^\infty}^2 \norm{\hs^{3/2} \theta_{\xi\xi}}{L^2} + \norm{\hs^{1/2} \theta_{\xi\xi}}{L^\infty} \norm{\hs \theta_{\xi\xi tt}}{L^2} + \norm{\hs \theta_{\xi\xi t}}{L^\infty} \norm{\hs^{1/2} \theta_{\xi\xi t}}{L^2}\\
        + \norm{\theta_{\xi t}}{L^\infty} \norm{\hs \theta_{\xi\xi}}{L^\infty} \norm{\hs^{1/2} \theta_{\xi\xi t}}{L^2} + \norm{\hs^{1/2} \theta_{\xi\xi}}{L^\infty }^2\norm{\hs^{1/2}\theta_{\xi tt}}{L^2}\\
        + \norm{\theta_{\xi t}}{L^\infty}^2 \norm{\hs \theta_{\xi\xi}}{L^\infty} \norm{\hs^{1/2} \theta_{\xi\xi}}{L^2}
        \\
        + \norm{\theta_{\xi}}{L^\infty} \norm{\hs^{1/2} \theta_{\xi ttt}}{L^2} + \norm{\theta_{\xi t}}{L^\infty} \norm{\hs^{1/2} \theta_{\xi tt}}{L^2} + \norm{\theta_{\xi t}}{L^\infty}^2 \norm{\hs^{1/2}\theta_{\xi t}}{L^2} 
        \\
        + \norm{\dfrac{\theta_{\xi}}{\hs}}{L^4} \norm{\hs^{1/2} \theta_{\xi tt}}{L^4} + \norm{\theta_{\xi t}}{L^\infty} \norm{\dfrac{\theta_{\xi t}}{\hs^{1/2}}}{L^2}
        \biggr\rbrace.
    \end{gathered}
\end{equation}
Therefore, \eqref{est:704}--\eqref{est:706} implies
\begin{equation}\label{est:707}
    \begin{gathered}
    \vert\int (M_1)_{tt} \theta_{\xi ttt} \idx\vert +  \vert\int (M_1)_{tt} \theta_{\xi tt} \idx\vert \\
    \lesssim P(\mathcal E_{\mrm{NL},1} + \mathcal E_{\mrm{NL},2}) \sum_{k=0,1,2}\mathcal D_k,
    \end{gathered}
\end{equation}
thanks to \eqref{ineq:001}--\eqref{ineq:003}. Hereafter $ P $ is a function of the argument such that
\begin{equation}\label{id:nonlinearity-5-1}
    P = P(y) \geq 0 \in C^\infty[0,\infty) \quad \text{and} \quad P(0) = 0.
\end{equation}

\smallskip

{\par\noindent\bf Estimates of $\int (M_2)_{tt}  \theta_{\xi\xi ttt} \idx$ and $\int (M_2)_{tt}  \theta_{\xi\xi tt} \idx$}

From \eqref{id:nonlinear-8} and \eqref{id:nonlinear-10}, one can calculate
\begin{equation}\label{est:708}
    \begin{gathered}
        (M_2)_{tt} = - \dfrac{\hs^2}{2} g'_4(\theta_\xi) \theta_{\xi\xi tt} \\
        -  \underbrace{\dfrac{\hs^2}{2} \lbrace g_4''(\theta_\xi) \theta_{\xi t} \theta_{\xi\xi t} + g_4''(\theta_\xi) \theta_{\xi tt} \theta_{\xi\xi} 
        + g_4'''(\theta_\xi) \theta_{\xi t}^2 \theta_{\xi\xi} \rbrace}_{=:M_3}. 
    \end{gathered}
\end{equation}
Therefore, 
\begin{equation}\label{est:709}
    \begin{gathered}
        \int (M_2)_{tt}  \theta_{\xi\xi ttt} \idx =  \dfrac{\mathrm{d}}{\mathrm{d}t} 
        \mathcal E_\delta 
        + \int \hs^2 \lbrace \dfrac{1}{4} g_4''(\theta_\xi) \theta_{\xi t} \theta_{\xi\xi tt}^2  \rbrace \idx + \int (M_3)_t \theta_{\xi\xi tt} \idx, 
    \end{gathered}
\end{equation}
where
\begin{equation}\label{def:energy-delta}
    \mathcal E_\delta : = - \dfrac{1}{4} \int \hs^2 g'_4(\theta_\xi)  \vert \theta_{\xi\xi tt} \vert^2 \idx - \int M_3 \theta_{\xi\xi tt}\idx.
\end{equation}
Direct calculation yields 
\begin{equation}\label{est:710}
    \begin{gathered}
    \mathcal E_\delta \lesssim F(\norm{\theta_\xi}{L^\infty})\lbrace \norm{\theta_{\xi}}{L^\infty} \norm{\hs \theta_{\xi\xi tt}}{L^2}^2 \\
    + \norm{\theta_{\xi t}}{L^\infty} \norm{\hs \theta_{\xi\xi t}}{L^2}\norm{\hs \theta_{\xi\xi tt}}{L^2} 
    + \norm{\theta_{\xi tt}}{L^2} \norm{\hs \theta_{\xi\xi}}{L^\infty}\norm{\hs \theta_{\xi\xi tt}}{L^2} \\
    + \norm{\theta_{\xi t}}{L^\infty}^2 \norm{\hs \theta_{\xi\xi}}{L^2}\norm{\hs \theta_{\xi\xi tt}}{L^2} \rbrace \lesssim P(\mathcal E_{\mrm{NL},1} + \mathcal E_{\mrm{NL},2}) \mathcal E_{\mrm{NL},1}.
    \end{gathered}
\end{equation}
On the other hand, one has that
\begin{equation}
    \begin{gathered}
        (M_3)_t = \hs^2 \lbrace f(\theta_\xi) \theta_{\xi t} \theta_{\xi t} \theta_{\xi\xi t} + f(\theta_\xi) \theta_{\xi tt} \theta_{\xi\xi t} + f(\theta_\xi) \theta_{\xi t} \theta_{\xi\xi tt} \\
        + f(\theta_\xi) \theta_{\xi t} \theta_{\xi tt} \theta_{\xi\xi} + f(\theta_\xi) \theta_{\xi ttt} \theta_{\xi\xi} + f(\theta_\xi) \theta_{\xi t} \theta_{\xi t} \theta_{\xi t} \theta_{\xi\xi} \\
        + f(\theta_\xi) \theta_{\xi t} \theta_{\xi tt} \theta_{\xi\xi}
        \rbrace.
    \end{gathered}
\end{equation}
Therefore, 
\begin{equation}\label{est:711}
    \begin{gathered}
        \vert \int \hs^2 \lbrace \dfrac{1}{4} g_4''(\theta_\xi) \theta_{\xi t} \theta_{\xi\xi tt}^2  \rbrace \idx + \int (M_3)_t \theta_{\xi\xi tt} \idx \vert \\
        \lesssim F(\norm{\theta_{\xi}}{L^\infty}) \norm{\theta_{\xi t}}{L^\infty}\norm{\hs \theta_{\xi\xi tt}}{L^2}^2 \\
        + F(\norm{\theta_{\xi}}{L^\infty})\lbrace
        \norm{\theta_{\xi t}}{L^\infty}^2\norm{\hs \theta_{\xi\xi t}}{L^2} + \norm{\theta_{\xi tt}}{L^2} \norm{\hs \theta_{\xi\xi t}}{L^\infty} \\
        + \norm{\theta_{\xi t}}{L^\infty} \norm{\theta_{\xi tt}}{L^2} \norm{\hs \theta_{\xi\xi}}{L^\infty} +\norm{\hs^{1/2}\theta_{\xi ttt}}{L^2} \norm{\hs^{1/2} \theta_{\xi\xi}}{L^\infty} \\
        + \norm{\theta_{\xi t}}{L^\infty}^3 \norm{\hs \theta_{\xi\xi}}{L^2}
        + \norm{\theta_{\xi t}}{L^\infty} \norm{\theta_{\xi tt}}{L^2}\norm{\hs \theta_{\xi\xi}}{L^\infty}
        \rbrace \times \norm{\hs \theta_{\xi\xi tt}}{L^2} \\
        \lesssim P(\mathcal E_{\mrm{NL},1} + \mathcal E_{\mrm{NL},2}) \sum_{k=0,1,2}\mathcal D_{k}
    \end{gathered}
\end{equation}

Following similar arguments implies that
\begin{equation}\label{est:712}
    \begin{gathered}
        \vert \int (M_2)_{tt}  \theta_{\xi\xi tt} \idx \vert \lesssim P(\mathcal E_{\mrm{NL},1} + \mathcal E_{\mrm{NL},2}) \sum_{k=0,1,2}\mathcal D_{k}. 
    \end{gathered}
\end{equation}

\bigskip 

{\noindent\bf Summary}

Collecting \eqref{est:703}, \eqref{est:707}, \eqref{est:709}--\eqref{est:712} leads to
\begin{equation}
    \dfrac{\mathrm{d}}{\mathrm{d}t}(\mathcal E_2 + \mathcal E_\delta) + \mathcal D_2 \leq P(\mathcal E_{\mrm{NL},1} + \mathcal E_{\mrm{NL},2}) \sum_{k=0,1,2}\mathcal D_{k}.
\end{equation}
Repeating the same arguments for $ \mathcal E_0 $ and $ \mathcal E_1 $, one can conclude that
\begin{equation}\label{est:a-priori}
    \dfrac{\mathrm{d}}{\mathrm{d}t}(\sum_{k=0,1,2}\mathcal E_k + \mathcal E_\delta) + (1- P(\mathcal E_{\mrm{NL},1} + \mathcal E_{\mrm{NL},2})) \sum_{k=0,1,2}\mathcal D_k \leq 0. 
\end{equation}

\subsection{Continuity argument and proof of Theorem \ref{thm-NL}}

Now we are at the right place to demonstrate the continuity arguments, which lead to the global stability and asymptotic stability theory. We first start with the {\it a priori} assumption; that is, for some $ \epsilon \in (0,1) $, such that for $ \forall T \in (0,\infty) $, 
\begin{equation}\label{cont-arg:assumption}
    \sup_{0\leq t \leq T} \lbrace \mathcal E_{\mrm{NL},1}(t) + \mathcal E_{\mrm{NL},2}(t) \rbrace \leq \varepsilon.
\end{equation}

\smallskip

Then, for $ \varepsilon $ small enough, \eqref{e-est:total-1} and \eqref{est:a-priori}, together with \eqref{est:701}, \eqref{est:702}, and \eqref{est:710}, imply that 
\begin{equation}\label{est:720}
    \sum_{k=0,1,2}\mathcal E_k + \mathcal E_\delta \geq \mfk d_1 \sum_{k=0,1,2}\mathcal E_k, \qquad \mathcal E_{\mrm{NL},2} \leq \mfk d_2 \mathcal E_{\mrm{NL},1},
\end{equation}
and
\begin{equation}
    \dfrac{\mrm{d}}{\mrm{d}t}(\sum_{k=0,1,2}\mathcal E_k + \mathcal E_\delta) + \mfk d_3 (\sum_{k=0,1,2}\mathcal E_k + \mathcal E_\delta) \leq 0 
\end{equation}
for some constants $ \mfk d_1, \mfk d_2, \mfk d_3 \in (0,\infty) $, 
and therefore
\begin{equation}\label{est:721}
    \sup_{0\leq t\leq T} e^{\mfk d_3 t}(\sum_{k=0,1,2}\mathcal E_k(t) + \mathcal E_\delta(t)) \leq (\sum_{k=0,1,2}\mathcal E_k(0) + \mathcal E_\delta(0)) =: \mathfrak E_0.
\end{equation}
Here $ \mathfrak E_0 $ is the total initial energy, i.e., 
\begin{equation}\label{def:initial_energy}
    \begin{gathered}
    \mathfrak E_0 := \sum_{k=0,1,2} \biggl\lbrace \dfrac{1}{2} \int h_\mrm s \vert \dt^{k+1} \theta_\mrm{in} \vert^2 \,\idx + \int ( 2 \vert \hs'\vert^2 - 4 \hs \hs'' + h_\mrm s^2) \vert \dt^k \theta_{\mrm{in},\xi} \vert^2 \,\idx \\
    \qquad +  \int \vert h_\mrm s \vert^2 \vert \dt^k \theta_{\mrm{in},\xi\xi} \vert^2 \,\idx + \mfk c_1 \int h_\mrm s \vert \dt^k \theta_{\mrm{in},\xi} \vert^2 \,\idx \\
    \qquad + \mfk c_1 \int h_\mrm s \dt^{k+1} \theta_\mrm{in} \cdot \dt^k \theta_\mrm{in} \,\idx \biggr\rbrace 
    + \mathcal E_\delta(0).
        \end{gathered}
\end{equation}
where $ \mathcal E_\delta(0) = \mathcal E_\delta\vert_{t=0} $ is the error term due to the nonlinearity, defined in \eqref{def:energy-delta}.
Here $ \theta_{\mrm{in}} $ and $ \dt \theta_{\mrm{in}} $ are the initial data for equation \eqref{lgeq:prtbtn}. The higher-order derivatives in time are defined inductively using the equation.

On the other hand, estimates \eqref{est:702}, \eqref{est:720}, and \eqref{est:721} imply
\begin{equation}\label{est:722}
    \sup_{0\leq t \leq T} e^{\mfk d_3 t} \lbrace \mathcal E_{\mrm{NL},1} + \mathcal E_{\mrm{NL},2}\rbrace \leq \mfk d_4 \mfk E_0,
\end{equation}
for some constant $ \mfk d_4 \in (0,\infty) $. Therefore, for small enough initial data
\begin{equation}\label{est:723}
    \mfk E_0 \leq \dfrac{\varepsilon}{2\mfk d_4},
\end{equation}
estimate \eqref{est:722} implies
\begin{equation}
    \sup_{0\leq t \leq T} e^{\mfk d_3 t} \lbrace \mathcal E_{\mrm{NL},1} + \mathcal E_{\mrm{NL},2}\rbrace \leq \dfrac{\varepsilon}{2},
\end{equation}
and consequently, this closes the {\it a priori} assumption \eqref{cont-arg:assumption}.

\smallskip 
With the above {\it a priori} estimates with initial data satisfying \eqref{est:723}, one can apply standard Galerkin's method to construct a local-in-time solution. Furthermore, one can apply the standard continuity argument and conclude the proof of theorem \ref{thm-NL}.

\section{Local well-posedness for general initial data}\label{sec:local_well_posedness}

\subsection{Lagrangian formulation and main theory: Local well-posedness}

The goal of this section is to investigate the local-in-time well-posedness theory of system \eqref{sys:sw-st} with general initial data; that is, with initial data not close to the equilibrium given by \eqref{eq:equilibrium}. Indeed, we only assume $h_0$ satisfies some regularity and convexity condition, see \eqref{h_0}, below. To formulate the shallow water equations in the Lagrangian coordinates with reference to the initial height profile, define $ x = \eta(\xi, t) $ by
\begin{equation}\label{def:lagrangian-map-general}
	\int_{a(t)}^{\eta(\xi, t)} h(x,t)\dx = \int_{-1}^\xi h_0 (x)\dx. 
\end{equation}
Then repeating the derivation from \eqref{def:lagrangian-map} to \eqref{sys:sw-st-lagrangian}, one can write down the shallow water equations in the $ (\xi, t) $-coordinates as follows:
\begin{subequations}\label{sys:sw-lagrangian-general}
	\begin{gather}
		h(\xi, t) = \dfrac{h_0(\xi)}{\eta_\xi(\xi,t)},\quad \dt \eta(\xi,t) = u(\xi,t), \label{Local:fp}\\
		h_0\dt u + \partial_\xi \biggl( \biggl(\dfrac{h_0}{\eta_\xi}\biggr)^2 \biggr) - 2 h_0\frac{ \partial_\xi}{\eta_\xi}\biggl(\frac{\partial_\xi}{\eta_\xi}\biggl(\frac{\partial_\xi}{\eta_\xi}\biggl(\frac{h_0}{\eta_\xi}\bigg)\bigg)\bigg) = 2 \partial_\xi \biggl( \dfrac{h_0}{\eta_\xi } \dfrac{u_\xi}{\eta_\xi} \biggr),\label{Local:eq}\\
	\partial_\xi \eta \vert_{\xi = -1,1} = 1, \quad \partial_\xi u\vert_{\xi = -1,1} = 0.\label{Local:bc}
	\end{gather}
	Here we have abused the notation and used $ u = u(\xi, t) = u(\eta(\xi,t),t) $ to denote the velocity in both the Lagrangian and the Euclidean coordinates. Moreover, to be consistent with our estimates, inspired by \cite{bresch2003some}, we rewrite the surface tension term as follows
	\begin{equation}
		\begin{gathered}
			-2 h_0\frac{ \partial_\xi}{\eta_\xi}\biggl(\frac{\partial_\xi}{\eta_\xi}\biggl(\frac{\partial_\xi}{\eta_\xi}\biggl(\frac{h_0}{\eta_\xi}\bigg)\bigg)\bigg)
			= -\partial_\xi\biggl(\frac{\partial_\xi}{\eta_\xi}\biggl(\frac{\partial_\xi}{\eta_\xi}\biggl(\frac{h_0^2}{\eta_\xi^2}\bigg)\bigg)-3\biggl(\frac{\partial_\xi}{\eta_\xi}\biggl(\frac{h_0}{\eta_\xi}\bigg)\bigg)^2\bigg)\\
			=\partial_{\xi\xi}\biggl(\frac{2h_0^2\eta_{\xi\xi}}{\eta_\xi^5}\biggr)
			+\partial_\xi\biggl(\frac{5h_0^2\eta_{\xi\xi}^2}{\eta_\xi^6}\biggr)
			+\partial_\xi\biggl(\frac{|\partial_\xi h_0|^2-2h_0\partial_\xi^2h_0 }{\eta_\xi^4}\biggr).
		\end{gathered}
	\end{equation}
	Therefore, \eqref{Local:eq} can be written as
	\begin{equation}{\tag{\ref*{Local:eq}'}} \label{Local:eq-re}
		\begin{gathered}
			h_0\dt u + \biggl\lbrace  \dfrac{h_0^2}{\eta_\xi^2}  +  \dfrac{|h_0'|^2-2h_0 h_0'' }{\eta_\xi^4} +\frac{5h_0^2\eta_{\xi\xi}^2}{\eta_\xi^6}\biggr\rbrace_\xi + \biggl\lbrace \dfrac{2h_0^2\eta_{\xi\xi}}{\eta_\xi^5} \biggr\rbrace_{\xi\xi} \\
			-2 \partial_\xi \biggl( \dfrac{h_0}{\eta_\xi } \dfrac{u_\xi}{\eta_\xi} \biggr)  
			 =  0.
		\end{gathered}
	\end{equation}
\end{subequations}
Notice that \eqref{Local:eq-re} is consistent with \eqref{lgeq:prtbtn-csv-2}. 
Notice that, since $ h_0 $ is no longer the equilibrium profile, compared to \eqref{eq:lin+nonlin}, the linear part of \eqref{Local:eq-re} has an extra term, i.e.,
\begin{equation}\label{Local:extra-term}
	(h_0^2 + \vert h_0'\vert^2 - 2 h_0 h_0'')' = 2h_0 (h_0' - h_0''') \neq 0. 
\end{equation}

Fortunately, this term is $ \mathcal O(h_0) $ near the boundary, and therefore the weighted estimates involving negative power of $ h_0 $ in section \ref{sec:linear-elliptic} (i.e., a weight of $ h_0^{-1} $ or $ h_0^{-1/2} $) are bounded. Therefore one can expect the elliptic estimates of \eqref{Local:eq-re} to be similar to those of  \eqref{eq:lin+nonlin} as in section \ref{sec:nonlinear_estimates}. For the energy estimate, instead of using the smallness of perturbation, one should use the smallness of time to control the nonlinearities. In particular, we prove the following local well-posedness result:
\begin{thm}[Local Well-posedness]\label{thm-local}
    Suppose the initial data $(u_0, h_0)$ satisfies 
\begin{equation}\label{u_0}
    u_0\in H^3(I),~~\partial_\xi u_0\in H_0^1(I),
\end{equation}
\begin{equation}\label{h_0}
    h_0\in H^4(I),~~\partial_{\xi\xi}h_0\leq 0,~~
    C_1 d(\xi)\leq h_0\leq C_2 d(\xi)~~\forall \xi\in \Bar{I},
\end{equation}
for some positive constants $C_1$ and $C_2$, where $d(\xi)=d(\xi,\partial I)$ is 
the distant function from $\xi$ to the initial boundary. $\Phi(t)$ defined in \eqref{Phi} below is a energy functional, suppose $\Phi(0)\leq\infty$. Then there exist a small time $T^\ast>0$ such that system \eqref{sys:sw-st-lagrangian} admits a unique strong solution $(\eta, u)$ in $I\times [0,T]$, with
\begin{equation}\label{8.6}\left\{
    \begin{aligned}
        &\eta,\eta_\xi\in C^1([0,T^\ast];L^2(I)),~~\eta_{\xi\xi},h_0\partial_\xi^3\eta\in C([0,T^\ast];L^2(I)),\\
        &h_0^{3/2}\partial_\xi^4\eta\in L^\infty([0,T^\ast];L^2(I)),~~
        u,u_\xi,h_0^{1/2} u_t\in C([0,T^\ast];L^2(I)),\\
        &u_{\xi t},u_{\xi\xi},h_0^{1/2}u_{tt},h_0\partial_\xi^3 u\in L^\infty([0,T^\ast];L^2(I)),
        \end{aligned}  \right.             
    \end{equation}
    and
\begin{equation}\label{8.7}
	\sup_{0\leq t\leq T^\ast}\Phi(t)\leq C,
\end{equation}
where $C$ depends on $\norm{h_0}{H^4}$, $\norm{u_0}{H^3}$ and $T^\ast$.
\end{thm}

\smallskip For the completeness of this paper, we will sketch the local-in-time estimates, which lead to the local-in-time well-posedness theory, in the following. 

\smallskip

The energy functional for the solution to \eqref{Local:eq-re} is defined as
\begin{equation}\label{Phi}
\begin{gathered}
    \Phi(t) = \Phi(u,\eta,t):=\sum_{k=0}^2\| h_0^{1/2} 
    \partial^k_tu(\cdot,t)\|_{L^2}^2 +\sum_{k=0}^2\|h_0\partial_t^k\partial^2_\xi\eta(\cdot,t)\|_{L^2}^2 \\
    +\sum_{k=0}^2\|h_0\partial_t^k\partial_\xi\eta(\cdot,t)\|_{L^2}^2 
    +\sum_{k=0}^2\|(|\partial_\xi h_0|^2-2h_0\partial_\xi^2h_0)^{1/2}\partial_t^k\partial_\xi\eta(\cdot,t)\|_{L^2}^2\\
    +\sum_{k=0}^1 \|h_0 \dt^k \partial^3_\xi  \eta(\cdot,t)\|_{L^2}^2 +\|h_0^{3/2}\partial^4_\xi\eta(\cdot,t)\|_{L^2}^2.
\end{gathered}
\end{equation}
Notice that $ \Phi $ is the analogy of $ \mathcal E_{\mrm{NL},1} + \mathcal E_{\mrm{NL},2} $ defined in \eqref{def:energy_nonlinear_1} and \eqref{def:energy_nonlinear_2}.
Then similar to Lemmas \ref{lm:embedding-1} and \ref{lm:embedding-2}, one has the following embedding inequalities: 
\begin{lem}\label{lm:embedding-3}
	The following inequalities hold:
	\begin{equation}\label{8.2}
		\begin{gathered}
		\norm{u_\xi}{L^\infty}+ \norm{\frac{u_\xi}{h_0^{1/2}}}{L^2} + \norm{\frac{u_\xi}{h_0}}{L^2} \lesssim  \norm{u_{\xi\xi}}{L^2} \lesssim \norm{h_0 \partial_\xi^3u}{L^2} + \norm{h_0 u_{\xi\xi}}{L^2} \\
   \lesssim \Phi(t)^{1/2}, \\
			\|u_{\xi t}\|_{L^2}\lesssim \|h_0u_{\xi t}\|_{L^2}+\|h_0\partial_\xi^2u_t\|_{L^2} \lesssim\Phi(t)^{1/2},\\
			\norm{h_0^{1/2} u_{\xi t}}{L^4} \lesssim \norm{h_0 u_{\xi\xi t}}{L^2} + \norm{u_{\xi t}}{L^2}\lesssim\Phi(t)^{1/2},\\
            \norm{h_0u_{\xi\xi}}{L^\infty}\lesssim \norm{h_0\partial_\xi^3 u}{L^2}+ \norm{u_{\xi\xi}}{L^2}\lesssim\Phi(t)^{1/2},
		\end{gathered}
	\end{equation}
similarly, 
\begin{equation}\label{8.3}
		\begin{gathered}
			\norm{\eta_{\xi\xi}}{L^2}\lesssim \norm{h_0\eta_{\xi\xi}}{L^2}+\norm{h_0\partial_\xi^3\eta}{L^2}  \Phi(t)^{1/2},\\
			\norm{h_0^{1/2}\partial_\xi^3\eta}{L^2}
			\lesssim \norm{h_0^{3/2}\partial_\xi^3\eta}{L^2}
			+\norm{h_0^{3/2}\partial_\xi^4\eta}{L^2}
			\lesssim \Phi(t)^{1/2},\\
            \norm{h_0^{1/2} \eta_{\xi\xi}}{L^\infty} \lesssim \norm{h_0^{1/2} \partial^3_\xi \eta}{L^2} + \norm{\eta_{\xi\xi}}{L^2}
			\lesssim \Phi(t)^{1/2},\\
			\norm{\eta_{\xi\xi}}{L^4} \lesssim \norm{h_0^{3/2} \partial_\xi^4\eta}{L^2} + \norm{\eta_{\xi\xi}}{L^2}\lesssim \Phi(t)^{1/2},\\
            \norm{h_0\partial_\xi^3\eta}{L^4} \lesssim \norm{h_0^{3/2} \partial_\xi^4\eta}{L^2} + \norm{h_0^{1/2}\partial^3_\xi\eta}{L^2}\lesssim \Phi(t)^{1/2}.
		\end{gathered}
	\end{equation}
\end{lem}
Moreover,
\begin{lem}
If $\eta(\xi,t)$ and $\xi$ satisfy \eqref{def:lagrangian-map-general} , then it holds that
	\begin{equation}\label{8.4}
			\norm{\eta_{\xi\xi}}{L^2}+ \norm{h_0\partial_\xi^3\eta}{L^2}\lesssim  t \sup_{0\leq s \leq t} \Phi(t)^{1/2}.
	\end{equation}
\end{lem}
\begin{proof}
    It follows from \eqref{def:lagrangian-map-general} that $\eta(\xi,0)=\xi$, thus
\begin{equation}
    \eta_{\xi\xi}(\xi,0)=\partial_\xi^3\eta(\xi,0)=0.
\end{equation}
Direct calculation together with Minkowski’s inequality yields
\begin{equation}
    \norm{\eta_{\xi\xi}(\cdot, t)}{L^2}\leq\int_0^t\norm{u_{\xi\xi}(\cdot,s)}{L^2}ds
    \leq t \sup_{0\leq s \leq t} \Phi(s)^{1/2},
\end{equation}
\begin{equation}
    \norm{h_0\partial_\xi^3\eta(\cdot, t)}{L^2}\leq\int_0^t\norm{h_0\partial_\xi^3 u(\cdot,s)}{L^2}ds
    \leq t \sup_{0\leq s \leq t} \Phi(s)^{1/2}.
\end{equation}
\end{proof}

\subsection{ {\it \textbf{A priori}} estimate}
The main aim of this section is to derive the key {\it a priori} bound, i.e., there exists $ T \in (0,\infty) $ such that
\begin{equation}
	\sup_{0\leq t \leq T} \Phi(t) \leq 2 M,
\end{equation}
where $ M := P(\Phi(0)) $
for some polynomial $P$, to be determined later. 

\subsubsection{The {\it \textbf{a priori}} assumption}
Assume that there exists a suitably small $T\in(0,1)$, to be determined, such that 
\begin{equation}
	\sup_{0\leq t \leq T} \Phi(t) \leq M,
\end{equation}
for some $ M \in(0,\infty) $. It follows from \eqref{Local:fp} that
\begin{equation}
	\eta(\xi,t)=\xi+\int_0^t u(\xi,s)\,\mrm{d}s,~~(\xi,t)\in(I\times [0,T]),
\end{equation}
which leads to, for $ t \in (0,T) $, thanks to \eqref{8.2}, that
\begin{equation}\label{eta-ul}
    |\eta_\xi(\xi,t)-1|\leq\int_0^t\|u_\xi(\cdot,s)\|_{L^\infty}ds\leq C_1M^{1/2} T \leq \frac{1}{2},
\end{equation}
provided that $ T $ is small enough. Therefore, without loss of generality, we assume that
\begin{equation}\label{eta_ulb}
	\frac{1}{2}\leq\eta_\xi(\xi,t)\leq \frac{3}{2},~~(\xi,t)\in(I\times [0,T]),
\end{equation}
for the remaining part of this section. 

To simplify the notation, we use $ P = P(\cdot) $ to represent a generic polynomial, which will be determined in the end. 

\subsubsection{Temporal derivative estimates}

{\noindent\bf Basic energy estimate.}
Thanks to the boundary condition \eqref{Local:bc}, taking the $L^2$-inner product of \eqref{Local:eq-re} with $u$ yields
\begin{equation}\label{Local:1}
	\begin{gathered}
		\dfrac{\mathrm{d}}{\mathrm{d}t} \int \biggl\lbrace \frac{1}{2}h_0 u^2 +\frac{h_0^2}{\eta_\xi}+ \frac{(|\partial_\xi h_0|^2-2h_0\partial_\xi^2h_0)}{3\eta_\xi^3} +\frac{h_0^2\eta_{\xi\xi}^2}{\eta_\xi^5} \biggr\rbrace d\xi \\
		+2\int \frac{h_0u_\xi^2}{\eta_\xi^2}d\xi
		= - \int\partial_\xi\biggl(\frac{u|\partial_\xi h_0|^2}{\eta_\xi^4}\biggr)d\xi,
	\end{gathered}
\end{equation}
where 
\begin{equation}\label{Local:11}
\begin{gathered}
    - \int\partial_\xi\biggl(\frac{u|\partial_\xi h_0|^2}{\eta_\xi^4}\biggr)d\xi\lesssim\norm{h_0}{H^3}^2(\norm{u}{H^1}^2+\norm{\eta_{\xi\xi}}{L^2}^2 + 1) \\ 
    \lesssim \Phi(t) + 1.
 \end{gathered}
\end{equation}
Therefore, 
integrating \eqref{Local:1} from 0 to t, together with \eqref{eta_ulb} and \eqref{Local:11}, yields
\begin{equation}\label{Local:111}
	\begin{gathered}
	\norm{h_0^{1/2}u}{L^2}^2+\norm{h_0\eta_\xi}{L^2}^2 + \norm{(|\partial_\xi h_0|^2-2h_0\partial_\xi^2h_0)^{1/2}\eta_\xi}{L^2}^2 +\norm{h_0\eta_{\xi\xi}}{L^2}^2\\
		\lesssim \Phi(0)+ t (\sup_{0\leq s \leq t}\Phi(s) + 1),
	\end{gathered}
\end{equation}
for any $ t \in [0,T] $. 

{\noindent\bf Estimate of $ \partial_t^2 u $.}
After applying $\partial_t^2$ to \eqref{Local:eq-re}, taking the $L^2$-inner product of the resultant with $\partial_t^2u$ yields
\begin{equation}\label{Local:3}
	\begin{gathered}
		\dfrac{\mathrm{d}}{\mathrm{d}t} \int \biggl\lbrace \frac{1}{2}h_0|\partial^2_tu|^2 +\frac{h_0^2 u_{\xi t}^2}{\eta_\xi^3}
		+\frac{2(|\partial_\xi h_0|^2-2h_0\partial_\xi^2h_0)u_{\xi t}^2}{\eta_\xi^5} +\frac{h_0^2|\partial_\xi^2u_t|^2}{\eta_\xi^5} \biggr\rbrace \,\idx\\
        + 2\int \frac{h_0|\partial_t^2u_\xi|^2}{\eta_\xi^2}d\xi\\
		=\int \biggl\lbrace -\frac{3h_0^2u_{\xi t}^2u_\xi}{\eta_\xi^4}
		-\frac{10(|\partial_\xi h_0|^2-2h_0\partial_\xi^2h_0)u_{\xi t}^2u_\xi}{\eta_\xi^6} -\frac{5h_0^2|\partial_\xi^2u_t|^2u_\xi}{\eta_\xi^6} \\
		+ \frac{6h_0^2u_\xi^2\partial^2_tu_\xi}{\eta_\xi^4} + \frac{20(\vert h_0'\vert^2 - 2 h_0 h_0'')u_\xi^2 \partial_t^2 u_\xi}{\eta_\xi^6}
		\\
        +5\partial_{tt}\biggl(\frac{h_0^2\eta_{\xi\xi}^2}{\eta_\xi^6}\biggr)\partial_t^2u_\xi d\xi -\frac{12h_0u_\xi u_{\xi t} \partial^2_tu_\xi}{\eta_\xi^3}
		+\frac{12h_0u_\xi^3\partial^2_tu_\xi}{\eta_\xi^4} \\
     - \bigl\lbrack \dfrac{20 h_0^2 u_{\xi\xi} u_\xi}{\eta_\xi^6} + \dfrac{10 h_0^2 \eta_{\xi\xi} \dt u_{\xi}}{\eta_\xi^6} - \dfrac{60 h_0^2 \eta_{\xi\xi} u_\xi^2 }{\eta_\xi^7} \bigr\rbrack_\xi \dt^2 u_\xi \biggr\rbrace \,\idx
        =: \sum_{i=1}^{9} \int I_{i} \,\idx.
	\end{gathered}
\end{equation}
Thanks to \eqref{eta_ulb}, with the help of \eqref{8.2}, \eqref{8.3} and Young's inequality, one can calculate that, for any $ \varepsilon \in (0,1) $ and $ C_\varepsilon \simeq \frac{1}{\varepsilon} $, 
\begin{align}
& \label{Local:31}
\sum_{i=1}^3\int I_{i}\,\idx \leq C\|u_\xi\|_{L^\infty}\Phi(t)\leq CP(\Phi(t)),
\\
\label{Local:32}
& \begin{aligned} \int I_{4}\,\idx \leq &  \varepsilon\|h_0^{1/2} \partial_t^2u_\xi\|_{L^2}^2+C_\varepsilon\norm{u_\xi}{L^\infty}^2\norm{u_\xi}{L^2}^2 \\ & \leq \varepsilon\|h_0^{1/2} \partial_t^2u_\xi\|_{L^2}^2
    + C_\varepsilon P(\Phi(t)), \end{aligned}
\\
& \begin{aligned} \int I_{5}\,\idx \leq & \varepsilon\| h_0^{1/2}  \partial_t^2u_\xi\|_{L^2}^2+C_\varepsilon \norm{u_\xi}{L^\infty}^2\norm{\frac{u_\xi}{h_0^{1/2}}}{L^2}^2 \\ & \leq \varepsilon\|h_0^{1/2} \partial_t^2u_\xi\|_{L^2}^2
    + C_\varepsilon P(\Phi(t)), \end{aligned}
\\
& \begin{aligned}
    \int I_{6}\,\idx & \leq \varepsilon\|h_0^{1/2} \partial_t^2u_\xi\|_{L^2}^2\\& \quad  +C_\varepsilon (1+\|u_\xi\|_{L^\infty}^8 +\|h_0 u_{\xi\xi} \|_{L^\infty}^8 +\| h_0^{1/2}\eta_{\xi\xi}\|_{L^\infty}^8) \\
    & \qquad \times (\|\eta_{\xi\xi}\|_{L^2}^2
    +\|h_0\partial_\xi^2u_t\|_{L^2}^2
    +\|u_{\xi\xi}\|_{L^2}^2+\|u_{\xi t}\|_{L^2}^2)\\
    & \leq\varepsilon\|h_0^{1/2} u_{\xi t}\|_{L^2}^2+C_\varepsilon P\bigl( \Phi(t)\bigr),
\end{aligned}
\\
& \begin{aligned}
\int I_{7} \,\idx & \leq \varepsilon\| h_0^{1/2} \partial_t^2u_\xi\|_{L^2}^2+C_\varepsilon \norm{u_\xi}{L^\infty}^2\norm{u_{\xi t}}{L^2}^2 \\
& \leq \varepsilon\| h_0^{1/2} \partial_t^2u_\xi\|_{L^2}^2
    + C_\varepsilon P(\Phi(t)),
    \end{aligned}
\\
& \begin{aligned}
\int I_{8}\,\idx & \leq \varepsilon\| h_0^{1/2} \partial_t^2u_\xi\|_{L^2}^2+C_\varepsilon \norm{u_\xi}{L^\infty}^4\norm{u_\xi}{L^2}^2 \\
& \leq \varepsilon\|h_0^{1/2} \partial_t^2u_\xi\|_{L^2}^2
    + C_\varepsilon P(\Phi(t)),
    \end{aligned}
\\
\label{Local:33}
& \begin{aligned}
\int I_{9}\,\idx & \leq \varepsilon\| h_0^{1/2} \partial_t^2u_\xi\|_{L^2}^2 \\
&\quad + C_\varepsilon (1+\|u_\xi\|_{L^\infty}^8 + \|h_0u_{\xi\xi}\|_{L^\infty}^8 + \| h_0^{1/2}\eta_{\xi\xi}\|_{L^\infty}^8 )\\
&\qquad \times (\|h_0\partial_\xi^3 u\|_{L^2}^2
+\|\eta_{\xi\xi}\|_{L^2}^2
+\|h_0\partial_\xi^3\eta\|_{L^2}^2
+\norm{h_0\partial_\xi^3\eta}{L^4}^4\\
& \qquad\quad +\norm{h_0^{1/2} u_{\xi t}}{L^4}^4
+\|u_{\xi\xi}\|_{L^2}^2
+\|h_0\partial_\xi^2u_t\|_{L^2}^2
+\|u_{\xi t}\|_{L^2}^2)\\
&\leq\varepsilon\| h_0^{1/2} u_{\xi t}\|_{L^2}^2+C_\varepsilon P\bigl( \Phi(t)\bigr).
\end{aligned}
\end{align}
Thanks to \eqref{Local:31}--\eqref{Local:33}, 
after choosing $\varepsilon$ small enough, integrating \eqref{Local:3} in $ t $ yields that, for any $ t\in [0,T] $, 
\begin{equation}\label{Local:311}
\begin{gathered}
    \norm{h_0^{1/2}\partial^2_tu}{L^2}^2+\norm{h_0u_{\xi t}}{L^2}^2+\norm{(|h_0'|^2-2h_0 h_0'')^{1/2}u_{\xi t}}{L^2}^2+ \norm{h_0\partial^2_\xi u_t}{L^2}^2\\
    \lesssim \Phi(0) + t P\bigl(\sup_{0\leq s\leq t} \Phi(s)\bigr).
\end{gathered}
\end{equation}
{\noindent\bf Estimate of $ \dt u $.} 
Since 
\begin{equation}
	\dt u(\xi,t)=\dt u(\xi,0)+\int_0^t \partial^2_t u(\xi,s)ds,
\end{equation}
it then follows from Cauchy’s inequality and Fubini’s theorem that, for any $ t \in [0,T) $, 
\begin{equation}\label{Local:211}
	\begin{gathered}
		\norm{h_0^{1/2}u_t}{L^2}^2
		\lesssim \Phi(0)+t\int_0^t\norm{h_0^{1/2}u_{tt}(s)}{L^2}^2ds
		\lesssim \Phi(0)+tP\bigl(\sup_{0\leq s\leq t} \Phi(s)\bigr).
	\end{gathered}
\end{equation}
Similar arguments also imply that, for any $ t \in [0,T] $, 
\begin{equation}\label{Local:2111}
	\begin{gathered}
    \norm{h_0u_{\xi}}{L^2}^2+\norm{(|h_0'|^2-2h_0 h_0'')^{1/2}u_{\xi}}{L^2}^2 +\norm{h_0\partial^2_\xi u}{L^2}^2 \\
		\lesssim \Phi(0)+ t P\bigl(\sup_{0\leq s\leq t} \Phi(s)\bigr).
	\end{gathered}
\end{equation}

{\noindent\bf Summary of temporal derivative estimates.} Combing \eqref{Local:111}, \eqref{Local:311}, \eqref{Local:211}, and \eqref{Local:2111} leads to
\begin{equation}\label{Local:dte}
\begin{gathered}
\sum_{k=0}^2\| h_0^{1/2}		\partial^k_tu(\cdot,t)\|_{L^2}^2
+\sum_{k=0}^2\|h_0\partial_t^k\partial^2_\xi\eta(\cdot,t)\|_{L^2}^2\\
+\sum_{k=0}^2\|h_0\partial_t^k\partial_\xi\eta(\cdot,t)\|_{L^2}^2
+\sum_{k=0}^2\|(|h_0'|^2-2h_0 h_0'')^{1/2}\partial_t^k\partial_\xi\eta(\cdot,t)\|_{L^2}^2\\
\lesssim \Phi(0)+ t P\bigl(\sup_{0\leq s\leq t} \Phi(s)\bigr),
\end{gathered}
\end{equation}
for any $ t \in [0,T] $.
 Moreover, Hardy's inequality \eqref{ineq:Hardy-3} also gives
\begin{equation}\label{Local:1111}
 	\norm{\eta_\xi}{L^2}^2+\norm{u_\xi }{L^2}^2+\norm{u_{\xi t}}{L^2}^2
 	\lesssim \Phi(0)+t P\bigl(\sup_{0\leq s\leq t} \Phi(s)\bigr).
\end{equation}

\subsubsection{Elliptic estimates}

We now turn to the elliptic estimates. First, notice that, 
integrating \eqref{Local:eq-re} from $-1$ to $1$ yields
\begin{equation}
    \int_{-1}^1 h_0 \dt u \,\idx + (\vert h_0'\vert^2)\big\vert_{\xi=-1}^{1} = 0,
\end{equation}
and therefore
\begin{equation}\label{Local:balance of momentum}
    \int_{-1}^1 h_0\partial^2_t u d\xi=0,
\end{equation}
thanks to \eqref{h_0} and \eqref{Local:bc}.

{\par\noindent\bf Estimate of $ \partial_\xi^3 u $.}
After applying $\partial_t$ to \eqref{Local:eq-re}, one has
\begin{equation} \label{dtre}
\begin{gathered}
{ h_0\partial^2_t u -2 \partial_\xi \bigl(h_0^2u_\xi\eta_\xi^{-3}\bigr)}
-4\partial_\xi\bigl((|h_0'|^2 -2h_0 h_0'')u_\xi\eta_\xi^{-5}\bigr) \\
{-2 \partial_\xi \bigl(h_0u_{\xi t}\eta_\xi^{-2}\bigr)+4 \partial_\xi \bigl(h_0u_\xi^2\eta_\xi^{-3}\bigr) } \\ 
+ 2\partial^2_\xi\bigl(h_0^2u_{\xi\xi}\eta_\xi^{-5}\bigr)
{-10\partial_\xi^2\bigl(h_0^2\eta_{\xi\xi}u_\xi\eta_\xi^{-6}\bigr)}
{+10\partial_\xi\bigl(h_0^2u_{\xi\xi}\eta_{\xi\xi}\eta_\xi^{-6}\bigr)}\\
{-30\partial_\xi\bigl(h_0^2\eta_{\xi\xi}^2u_\xi\eta_\xi^{-7}\bigr)} = 0 .
\end{gathered}
\end{equation}
Integrating \eqref{dtre} from $y = -1 $ to $ \xi$ and multiplying the resulting equation by $\eta_\xi^5\partial_\xi^3 u$ yields, after a complex but straightforward calculation, that 
\begin{equation} \label{ixdtre}
	\begin{gathered}
	L:= 	2h_0^2\vert\partial_\xi^3u\vert^2+4h_0\partial_\xi h_0u_{\xi\xi}\partial_\xi^3u
        -4(|\partial_\xi h_0|^2-2h_0\partial_\xi^2h_0)u_\xi\partial_\xi^3u\\
		={ -\eta_\xi^5\partial_\xi^3u\int_{-1}^\xi h_0\partial^2_t u(y)dy +2 h_0^2u_\xi\eta_\xi^{2}\partial_\xi^3u}
        {+2h_0u_{\xi t}\eta_\xi^{3} \partial_\xi^3u
		-4h_0  u_\xi^2\eta_\xi^{2}\partial_\xi^3u}\\ 
		{+ 10h_0^2\partial_\xi^3u u_{\xi\xi}\eta_{\xi\xi}\eta_\xi^{-1}
		-30h_0^2\partial_\xi^3u\eta_{\xi\xi}^2u_\xi\eta_\xi^{-2}}
		{+10\partial_\xi^3u\partial_\xi\bigl(h_0^2\eta_{\xi\xi}\bigr)u_\xi\eta_\xi^{-1}} =:R.
	\end{gathered}
\end{equation}
Since $h_0$ is concave after integration by parts one gets
\begin{equation} \label{ixdtre-L}
	\begin{aligned}
        \int L(\xi,t) \,\idx=&2\norm{h_0\partial_\xi^3u}{L^2}^2
        +\int 2h_0\partial_\xi h_0\partial_\xi(u_{\xi\xi}^2)d\xi\\
        &+\int 4(|\partial_\xi h_0|^2
        -2h_0\partial_\xi^2h_0)u_{\xi\xi}^2d\xi
        -\int 4h_0\partial_\xi^3 h_0 \partial_\xi (u_\xi^2) d\xi\\
        =&2\norm{h_0\partial_\xi^3u}{L^2}^2
        +2\norm{\partial_\xi h_0 u_{\xi\xi}}{L^2}^2\\
        &-\int 10h_0\partial_\xi^2h_0 u_{\xi\xi}^2 d\xi
        +\int \partial_\xi(h_0\partial_\xi^3 h_0)u_\xi^2 d\xi\\
        \geq &2\norm{h_0\partial_\xi^3u}{L^2}^2
        +2\norm{\partial_\xi h_0 u_{\xi\xi}}{L^2}^2\\
        &-\int 10h_0\partial_\xi^2h_0 u_{\xi\xi}^2 d\xi
        -C\norm{h_0}{H^4}^2\norm{u_\xi}{L^2}^2.
	\end{aligned}
\end{equation}
and
\begin{equation}\label{Local:e31}
    \norm{h_0}{H^4}^2\norm{u_\xi}{L^2}^2\lesssim \Phi(0)+tP\bigl(\sup_{0\leq s\leq t} \Phi(t)^{1/2}\bigr),
\end{equation}
due to \eqref{Local:1111}.

Next, to estimate $ \int R(\xi, t)\,\idx $,  thanks to \eqref{eta_ulb}, \eqref{8.2}, \eqref{8.3}, \eqref{8.4}, \eqref{Local:1111} and \eqref{Local:eq}, \eqref{lnest:017-1} we have
\begin{align}
& \label{Local:e32}
\begin{aligned}
    \norm{h_0^{-1}\eta_x^5\int_{-1}^\xi h_0\partial^2_t u(y)dy}{L^2}^2\lesssim &\norm{\eta_x}{L^\infty}^5\norm{h_0^{1/2}\partial^2_t u}{L^2}^2\\
    \lesssim &P\bigl(\Phi(0)\bigr)+tP\bigl(\sup_{0\leq s\leq t} \Phi(t)^{1/2}\bigr),
\end{aligned}\\
&
\begin{aligned}
    \norm{h_0u_\xi\eta_\xi^2+u_{\xi t}\eta_\xi^3}{L^2}^2
    \lesssim &\norm{h_0 u_\xi}{L^2}^2\norm{\eta_\xi}{L^\infty}^4+\norm{u_{\xi t}}{L^2}^2\norm{\eta_\xi}{L^\infty}^6\\
    \lesssim &P\bigl(\Phi(0)\bigr)+tP\bigl(\sup_{0\leq s\leq t} \Phi(t)^{1/2}\bigr),
\end{aligned}\\
&
\label{Local:e34}
\begin{aligned}
    \norm{u_\xi^2\eta_\xi^2}{L^2}^2
    \lesssim &\norm{\eta_\xi}{L^\infty}^4\norm{u_\xi}{L^2}^2\norm{u_\xi}{L^\infty}^2
    \lesssim \norm{u_\xi}{L^2}^3\norm{u_{\xi\xi}}{L^2}\\
    \lesssim &\norm{u_\xi}{L^2}^3(\norm{h_0u_{\xi\xi}}{L^2}+\norm{h_0\partial^3_\xi u}{L^2})\\
    \leq &\varepsilon\norm{h_0\partial_\xi^3 u}{L^2}^2+P\bigl(\Phi(0)\bigr)+tP\bigl(\sup_{0\leq s\leq t} \Phi(t)^{1/2}\bigr),
\end{aligned}
\end{align}
where we have used the simple fact that
\begin{equation}
	u_\xi^2(\xi)=2\int_{-1}^\xi u_\xi(y) u_{\xi\xi}(y)dy
	\lesssim \norm{u_\xi}{L^2}\norm{u_{\xi\xi}}{L^2},
\end{equation}
due to boundary condition \eqref{Local:bc}.
\begin{align}
&
\label{Local:e35}
\begin{aligned}
    \norm{h_0u_{\xi\xi}\eta_{\xi\xi}\eta_\xi^{-1}
        +h_0\eta_{\xi\xi}^2u_\xi\eta_\xi^{-2}}{L^2}^2
    \lesssim &\norm{h_0(u_{\xi\xi}+\eta_{\xi\xi})}{L^\infty}^2\norm{\eta_{\xi\xi}}{L^2}^2\\
    \lesssim &tP\bigl(\sup_{0\leq s\leq t} \Phi(t)^{1/2}\bigr),
\end{aligned}\\
&
\label{Local:e36}
\begin{aligned}
 \norm{h_0^{-1}\partial_\xi\bigl(h_0^2\eta_{\xi\xi}\bigr)u_\xi\eta_\xi^{-1}}{L^2}^2
    & \lesssim (1+\norm{h_0}{H^2}^2)\norm{u_\xi}{L^\infty}^2(\norm{\eta_{\xi\xi}}{L^2}^2+\norm{h_0\partial_\xi^3\eta}{L^2}^2)\\
    & \lesssim tP\bigl(\sup_{0\leq s\leq t} \Phi(t)^{1/2}\bigr).
\end{aligned}
\end{align}
Then choosing $\varepsilon$ small enough, it follows from \eqref{Local:e31}-\eqref{Local:e34}, \eqref{Local:e35}, \eqref{Local:e36} and Young's inequality that 
\begin{equation}\label{ixdtre-R}
    \int R \,\idx \lesssim \norm{h_0\partial_\xi^3u}{L^2}^2
    +P\bigl(\Phi(0)\bigr)+tP\bigl(\sup_{0\leq s\leq t} \Phi(t)^{1/2}\bigr).
\end{equation}
Therefore \eqref{ixdtre-R} together with \eqref{ixdtre-L} and \eqref{Local:e31} yields
\begin{equation} \label{Local:e3s}
	\begin{gathered}
		\norm{h_0\partial_\xi^3 u}{L^2}^2
		\lesssim P\bigl(\Phi(0)\bigr)+tP\bigl(\sup_{0\leq s\leq t} \Phi(t)^{1/2}\bigr),
	\end{gathered}
\end{equation}
which together with \eqref{Local:2111} and Hardy's inequality \eqref{ineq:Hardy-3} gives
\begin{equation} \label{Local:e3s1}
	\begin{gathered}
		\norm{u_{\xi\xi}}{L^2}^2\lesssim \norm{h_0u_{\xi\xi}}{L^2}^2+\norm{h_0\partial^3_\xi u}{L^2}^2
		\lesssim P\bigl(\Phi(0)\bigr)+tP\bigl(\sup_{0\leq s\leq t} \Phi(t)^{1/2}\bigr),
	\end{gathered}
\end{equation}
and due to \eqref{Local:bc} and \eqref{ineq:Hardy-4} one gets
\begin{equation} \label{Local:e3s2}
	\begin{gathered}
		\norm{h_0^{-1}u_{\xi}}{L^2}^2\lesssim \norm{u_{\xi\xi}}{L^2}^2
		\lesssim P\bigl(\Phi(0)\bigr)+tP\bigl(\sup_{0\leq s\leq t} \Phi(t)^{1/2}\bigr).
	\end{gathered}
\end{equation}

\smallskip

{\noindent\bf Estimate of $\partial_\xi^4 \eta $.} Following the estimates in section \ref{sec:ell-est-001}, after multiplying \eqref{Local:eq-re} with $h_0^{-1/2}\eta_\xi^5$ and rearranging the resultant, one has that
{
	\begin{equation}\label{Local:dtrere}
		\begin{gathered}
			2( h_0^{3/2} \eta_{\xi\xi\xi\xi} + 4 h_0^{1/2} h_0' \eta_{\xi\xi\xi}) = 2 \underbrace{\bigl(- h_0^{1/2} h_0' \eta_{\xi}^3 + h_0^{1/2} h_0''' \eta_\xi \bigr)}_{:=J_1} - h_0^{1/2} u_t \eta_\xi^5 \\
			+ 2 h_0^{3/2} \eta_\xi^2 \eta_{\xi\xi} - 12 h_0^{1/2} h_0'' \eta_{\xi\xi} + 2 h_0^{1/2}\eta_{\xi}^3 u_{\xi\xi} + 2 h_0^{-1/2} h_0' \eta_{\xi}^3 u_{\xi}\\
			+ 20 \dfrac{h_0^{3/2} \eta_{\xi\xi}\eta_{\xi\xi\xi}}{\eta_\xi} + 20 \dfrac{h_0^{1/2} h_0' \eta_{\xi\xi}^2}{\eta_\xi} + \dfrac{10 \eta_\xi^5}{h_0^{1/2}} \biggl( \dfrac{h_0^2 \eta_{\xi\xi}}{\eta_\xi^6}\biggr)_\xi \eta_{\xi\xi} \\
			- 5 \eta_{\xi}^5h_0^{3/2} \biggl( \dfrac{\eta_{\xi\xi}^2}{\eta_\xi^6} \biggr)_\xi
			- 4 h_0^{1/2} \eta_{\xi}^2 \eta_{\xi\xi} u_\xi.
		\end{gathered}
	\end{equation}
}
Repeating calculation similar to \eqref{lnest:042}, since $h_0$ is concave, the $ L^2 $-norm of the left hand side of \eqref{Local:dtrere} satisfies 
\begin{equation} \label{Local:421}
	\begin{gathered}
		\norm{h_0^{3/2}\partial_\xi^4 \eta + 4h_0^{1/2}\partial_\xi h_0 \partial_\xi^3\eta}{L^2}^2  \gtrsim \norm{h_0^{3/2}\partial_\xi^4 \eta}{L^2}^2,
	\end{gathered}
\end{equation}
and thanks to \eqref{eta_ulb}, \eqref{8.2}, \eqref{8.3}, \eqref{8.4}, \eqref{Local:2111}, \eqref{Local:e3s1} and \eqref{Local:e3s2} we have that the right hand side of \eqref{Local:dtrere} can be estimated as follows:
\begin{align}
&\norm{J_1}{L^2}^2
\lesssim\norm{h_0}{H^4}^2\norm{\eta_{\xi}}{L^2}^2
\lesssim  P\bigl(\Phi(0)\bigr)+ tP\bigl(\sup_{0\leq s\leq t} \Phi(s)^{1/2}\bigr),
\\
&\norm{h_0^{1/2} u_t\eta_\xi^5}{L^2}^2\lesssim \norm{h_0^{1/2} u_t}{L^2}^2\lesssim P\bigl(\Phi(0)\bigr)+tP\bigl(\sup_{0\leq s\leq t} \Phi(s)^{1/2}\bigr),
\\
&\norm{h_0^{3/2}\eta_{\xi\xi}\eta_\xi^{2}}{L^2}^2
\lesssim \norm{h_0}{H^2}^{3}\norm{\eta_\xi}{H^1}^2
\lesssim P\bigl(\Phi(0)\bigr)+tP\bigl(\sup_{0\leq s\leq t} \Phi(s)^{1/2}\bigr),
\\
&\norm{h_0^{1/2}\partial_\xi^2 h_0\eta_{\xi\xi}}{L^2}^2
\lesssim\norm{h_0}{H^2}^{3}\norm{\eta_{\xi\xi}}{L^2}^2
\lesssim t P\bigl(\sup_{0\leq s\leq t} \Phi(s)^{1/2}\bigr),
\\
&\norm{h_0^{1/2}u_{\xi\xi}\eta_\xi^3}{L^2}^2
\lesssim \norm{h_0}{H^1}\norm{u_{\xi\xi}}{L^2}^2
\lesssim P\bigl(\Phi(0)\bigr)+ tP\bigl(\sup_{0\leq s\leq t} \Phi(s)^{1/2}\bigr),
\\
&\begin{aligned}
\norm{h_0^{-1/2}\partial_\xi h_0u_\xi\eta_\xi^3}{L^2}^2
\lesssim &\norm{h_0}{H^2}^{3}\norm{h_0^{-1}u_\xi}{L^2}^2\\
\lesssim & P\bigl(\Phi(0)\bigr)+tP\bigl(\sup_{0\leq s\leq t} \Phi(s)^{1/2}\bigr),
\end{aligned}\\
&\norm{h_0^{3/2}\eta_{\xi\xi}\eta_{\xi\xi\xi}\eta_\xi^{-1}}{L^2}^2
\lesssim \norm{h_0^{1/2}\eta_{\xi\xi}}{L^\infty}^2
\norm{h_0\partial_\xi^3\eta}{L^2}^2
\lesssim tP\bigl(\sup_{0\leq s\leq t} \Phi(s)^{1/2}\bigr),
\\
&\norm{h_0^{1/2}\partial_\xi h_0\eta_{\xi\xi}^2\eta_\xi^{-1}}{L^2}^2
\lesssim \norm{h_0^{1/2}\eta_{\xi\xi}}{L^\infty}^2
\norm{\eta_{\xi\xi}}{L^2}^2
\lesssim  tP\bigl(\sup_{0\leq s\leq t} \Phi(s)^{1/2}\bigr),
\\
&\begin{aligned}    &\norm{h_0^{-1/2}\eta_\xi^5\eta_{\xi\xi}\partial_\xi\bigl(h_0^2 \eta_{\xi\xi}\eta_\xi^{-6}
\bigr)}{L^2}^2+\norm{h_0^{3/2}\eta_\xi^5\partial_\xi\bigl(\eta_{\xi\xi}^2\eta_\xi^{-6}
\bigr)}{L^2}^2\\ 
\lesssim & (\norm{h_0^{1/2}\eta_{\xi\xi}}{L^\infty}^2+\norm{h_0^{1/2}\eta_{\xi\xi}}{L^\infty}^4)(\norm{\eta_{\xi\xi}}{L^2}^2
+\norm{h_0\partial_\xi^3\eta}{L^2}^2)
\\
\lesssim  & tP\bigl(\sup_{0\leq s\leq t} \Phi(s)^{1/2}\bigr),
\end{aligned}\\
&\begin{aligned}\label{Local:4210}
	    \norm{h_0^{1/2}u_{\xi}\eta_{\xi\xi}\eta_\xi^{2}}{L^2}^2
		\lesssim & \norm{h_0}{H^1}\norm{u_\xi}{L^\infty}\norm{\eta_{\xi\xi}}{L^2}^2\\
		\lesssim & P\bigl(\Phi(0)\bigr)+ tP\bigl(\sup_{0\leq s\leq t} \Phi(s)^{1/2}\bigr),
\end{aligned}
\end{align}
Then it follows from \eqref{Local:dtrere} and \eqref{Local:421}-\eqref{Local:4210} that
\begin{equation} \label{Local:444}
	\begin{gathered}
		\norm{h_0^{3/2}\partial_\xi^4 \eta}{L^2}^2
		\lesssim P\bigl(\Phi(0)\bigr)+tP\bigl(\sup_{0\leq s\leq t} \Phi(s)^{1/2}\bigr).
	\end{gathered}
\end{equation}

\subsubsection{Summary: The local-in-time {\it \textbf{a priori}} estimates}
Now combing \eqref{8.4}, \eqref{Local:dte}, \eqref{Local:e3s} and \eqref{Local:444} together, we finally arrive at
\begin{equation}
	\Phi(t)\leq P\bigl(\Phi(0)\bigr)+CtP\bigl(\sup_{0\leq s\leq t} \Phi(s)^{1/2}\bigr),
\end{equation}
for any $t\in [0,T]$, where $C$ is a positive constant only depends on $\norm{h_0}{H^3}$, and $P$ is a generic polynomial satisfying all the estimates above. Thus for sufficiently small $T\geq 0$, 
\begin{equation}
	\sup_{0\leq t\leq T}\Phi(t)\leq 2P\bigl(\Phi(0)\bigr).
\end{equation}
\subsection{Proof of theorem \ref{thm-local}}
With the above {\it a priori} estimates and initial condition \eqref{u_0}, \eqref{h_0}, suppose the initial functional energy
$\Phi(0)<\infty$, one can apply standard Galerkin's method to construct a strong solution $(\eta,u)$ satisfying \eqref{8.6} and \eqref{8.7}, see \cite{li_well-posedness_2022} for further details.

\subsection{Uniqueness}

Let $(u,\eta)$ and $(v,\zeta)$ be two solutions to the system on $[0, T]$ with initial data $(h_0, u_0)$ satisfying the same estimate. Their corresponding relationships are:
\begin{equation}
	(\eta,\zeta)(\xi,t)=\xi+\int_0^t (u,v)(\xi,s)\,\mrm{d}s.
\end{equation}
Let 
\begin{equation}
	w=u-v,~~\chi=\eta-\zeta.
\end{equation}
Then $(w,\chi)$ satisfies
\begin{equation} \label{L-U}
	\begin{gathered}
		h_0\dt w + \partial_\xi \biggl( \frac{h_0^2(\eta_\xi+\zeta_\xi)\chi_\xi}{\eta_\xi^2\zeta_\xi^2}\biggr)-2 \partial_\xi \biggl(\frac{h_0w_\xi}{\eta_\xi^2}+ \frac{h_0v_\xi(\eta_\xi+\zeta_\xi)\chi_\xi}{\eta_\xi^2\zeta_\xi^2} \biggr) \\ 
		=-\partial_\xi\biggl(\partial_\xi\biggl(\frac{2h_0^2\chi_{\xi\xi}}{\eta_\xi^5}+\frac{2h_0^2\zeta_{\xi\xi}P_1(\eta_\xi,\zeta_\xi)\chi_\xi}{\eta_\xi^5\zeta_\xi^5}\biggr)\biggr)\\
		-\partial_\xi\biggl(
		\frac{5h_0^2(\eta_{\xi\xi}+\zeta_{\xi\xi})\chi_{\xi\xi}}{\eta_\xi^6}
		+\frac{5h_0^2\zeta_{\xi\xi}^2P_2(\eta_\xi,\zeta_\xi)\chi_\xi}{\eta_\xi^6\zeta_\xi^6}\biggr)\\
		-\partial_\xi\biggl((|\partial_\xi h_0|^2-2h_0\partial_\xi^2h_0)\frac{(\eta_\xi+\zeta_\xi)(\eta_\xi^2+\zeta_\xi^2)\chi_\xi }{\eta_\xi^4\zeta_\xi^4}\biggr),
	\end{gathered}
\end{equation}
with initial data
\begin{equation}
	(w,\chi)(\xi,0)=(0,0),
\end{equation}
and boundary condition
\begin{equation}
	(w_\xi,\chi_\xi)(-1,t)= (w_\xi,\chi_\xi)(1,t)=(0,0),
\end{equation}
where polynomial 
\begin{gather*}
	P_1(\eta_\xi,\zeta_\xi)=\eta_\xi^4+\eta_\xi^3\zeta_\xi+\eta_\xi^2\zeta_\xi^2+\eta_\xi\zeta_\xi^3+\zeta_\xi^4, \\
\intertext{and}
	P_2(\eta_\xi,\zeta_\xi)=(\eta_\xi^3+\zeta_\xi^3)(\eta_\xi^2+\eta_\xi\zeta_\xi+\zeta_\xi^2).
\end{gather*}

Define
\begin{equation}
	\begin{gathered}
		\Phi_0(w,\chi,t)=\norm{\sqrt h_0 
			w(\cdot,t)}{L^2}^2+\norm{ h_0 
			\chi_\xi(\cdot,t)}{L^2}^2+\norm{ h_0 
			\chi_{\xi\xi}(\cdot,t)}{L^2}^2\\
		\norm{(|\partial_\xi h_0|^2-2h_0\partial_\xi^2h_0)^{1/2}\chi_\xi(\cdot,t)}{L^2}^2.    
	\end{gathered}
\end{equation}

Multiplying \eqref{L-U} both sides by $w$, integrating the resulting equation on $(-1,1)\times [0,T]$, using $\chi_t=w$ after integration by parts we finally find that
\begin{equation}
	\begin{gathered}
		\sup_{0\leq s\leq t}\Phi_0(w,\chi,s)\leq Ct\sup_{0\leq s\leq t}\Phi_0(w,\chi,s)+Ct\sup_{0\leq s\leq t}\int_\Omega h_0\chi_\xi^2(s)\,\idx\\
		\leq Ct\sup_{0\leq s\leq t}\Phi_0(w,\chi,s),
	\end{gathered}
\end{equation}
for all $t\in [0,T]$, where we have used Hardy’s inequality and C depends on $\Phi(u,\eta,t)$ and $\Phi(v,\zeta,t)$. Finally $w=0$ thus $\chi=0$ follows from \eqref{h_0} and the fact $(u,v)(\xi, t)\in C([-1,1]\times[0,T])$.

\section*{Acknowledgments}
DP was supported by the German Research Foundation (DFG) through the project 422792530 within the Priority Program SPP 2171. JL was supported in part
by Zheng Ge Ru Foundation, Hong Kong RGC Earmarked Research Grants CUHK-14301421, CUHK-14300819,
CUHK-14302819, CUHK-14300917, the key project of NSFC (Grant No. 12131010)  the Shun Hing Education and Charity Fund.

\bibliographystyle{plain}
\bibliography{short}

\end{document}